\documentclass[11pt]{article}

\usepackage[english]{babel}
\usepackage[utf8]{inputenc}
\usepackage{amsmath,amsfonts,amssymb,amsthm,mathrsfs}
\usepackage{minitoc}
\usepackage[normalem]{ulem}
 \usepackage[colorlinks=true]{hyperref}
\renewcommand{\leq}{\leqslant}
\renewcommand{\geq}{\geqslant}
\usepackage[mono=false]{libertine}
\usepackage[cal=euler, calscaled=.95, scr=boondoxo]{mathalfa}
\useosf
\usepackage{mathpazo}
\usepackage{mathtools}
\newcommand{\bluecirc}{{\color{blue}\bullet}\mathllap{\color{blue}\circ}}
\newcommand{\edge}{\bullet\hspace{-0.7mm}-\hspace{-0.7mm}\bullet}
\def\llbracke{\{\hspace{-.20em} \{ }
\def\rrbracke{ \} \hspace{-.20em}\}}

\usepackage{wasysym}

\usepackage[dvipsnames, svgnames]{xcolor}
\usepackage[a4paper,vmargin={3.5cm,3.5cm},hmargin={2.5cm,2.5cm}]{geometry}
\usepackage[font=sf, labelfont={sf,bf}, margin=1cm]{caption}
\usepackage{graphicx,graphics}
\usepackage{epsfig}
\usepackage{latexsym}
\usepackage{stmaryrd}
\usepackage{ae,aecompl}

\usepackage[english]{babel}

\usepackage{pstricks}
\usepackage{enumerate}
\usepackage{tikz}
\usepackage{fontawesome}
\usetikzlibrary{arrows,automata}

\renewcommand{\P}{\mathbb{P}}
\newcommand{\E}{\mathbb{E}}
\DeclareFixedFont{\beaupetit}{T1}{ftp}{b}{n}{2cm}

\newtheorem{theorem}{Theorem}[]

\newtheorem{proposition}[]{Proposition}

\newtheorem{lemma}[]{Lemma}
\newtheorem{corollary}[]{Corollary}
\newtheorem{conjecture}[]{Conjecture}
\theoremstyle{definition}

\newtheorem*{remark}{Remark}

\usepackage{todonotes}

\newcommand{\R}{\mathbb{R}}

\newcommand{\FM}{\mathscr{F}\hspace{-.20em}\mathscr{M}}

\title{{\bf Parking on Cayley trees  \&  Frozen Erd{\H{o}}s--R\'enyi}}
\author{Alice Contat\footnote{Universit\'e Paris-Saclay. E-mail: \texttt{alice.contat@universite-paris-saclay.fr}} \  and Nicolas Curien\footnote{Universit\'e Paris-Saclay and Institut Universitaire de France. E-mail: \texttt{nicolas.curien@gmail.com}}} 
             
\begin{document}
            \date{}
             \maketitle

\begin{abstract}Consider a uniform rooted Cayley tree $T_{n}$ with $n$ vertices and let $m$ cars arrive sequentially, independently, and uniformly on its vertices. Each car tries to park on its arrival node, and if the spot is already occupied, it drives towards the root of the tree and parks as soon as possible. Lackner \& Panholzer \cite{LaP16} established a phase transition for this process when $ m \approx \frac{n}{2}$. In this work, we couple this model with a variant of the classical Erd{\H{o}}s--R\'enyi random graph process. This enables us to describe the phase transition for the size of the  components of parked cars using a modification of the multiplicative coalescent which we name the  \emph{frozen multiplicative coalescent}.  The geometry of critical parked clusters is also studied. Those trees are very different from Bienaym\'e--Galton--Watson trees and should converge towards the growth-fragmentation trees canonically associated to the $3/2$-stable process that already appeared in the study of random planar maps. \end{abstract}
\begin{figure}[!h]
 \begin{center}
\includegraphics[width=16cm]{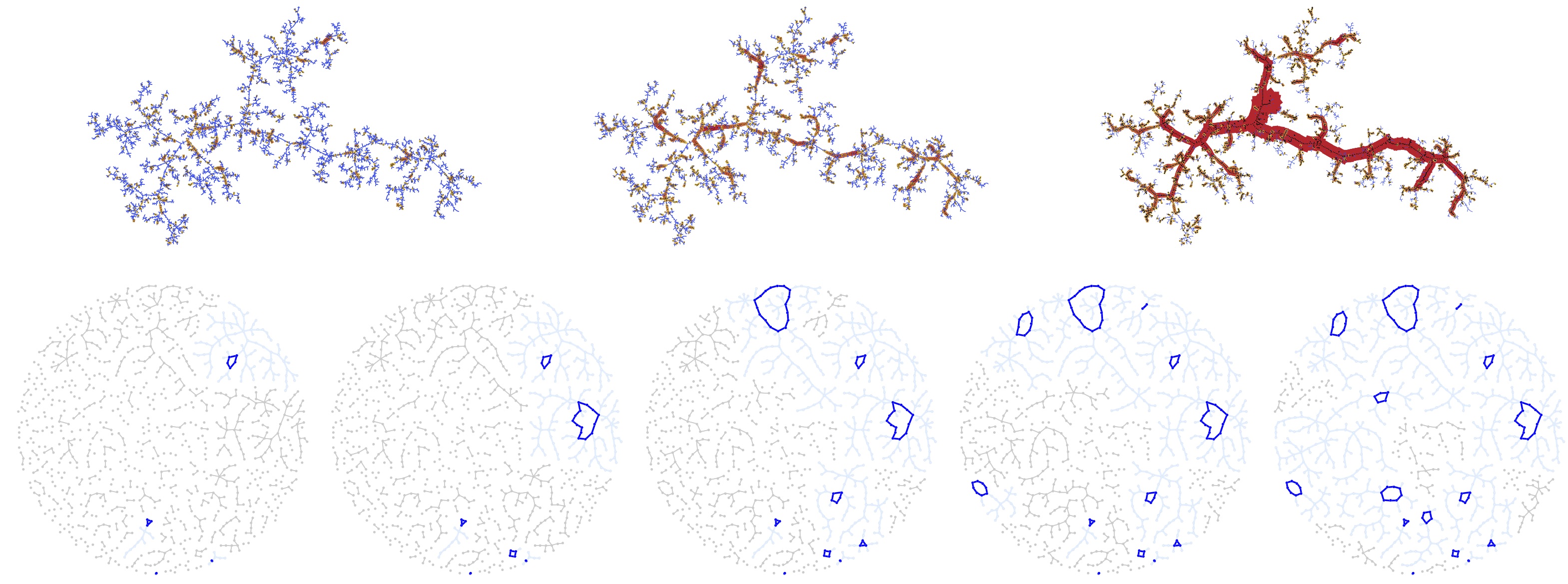}
           \caption{First line: Parking on a random Cayley tree with $10000$ vertices when resp.\ $4000,5000$ and $6000$ cars have arrived (color and thickness indicate the flux of cars along the edges). Second line: The frozen Erd{\H{o}}s--R\'enyi process at stages $400,500,600,700$ and $800$ on a graph with $1000$ vertices. }
 \end{center}
 \end{figure}

\section*{Introduction}

In this paper we establish a  connection between the parking process on a random Cayley tree and a certain modification of the classical Erd{\H{o}}s--R\'enyi random graph  obtained by freezing or more precisely ``slowing down'' components with surplus. This unexpected relationship enables us to understand  the phase transition for parking established in \cite{LaP16} and in return gives a new point of view on the Erd{\H{o}}s--R\'enyi  random graph and the multiplicative coalescent process. Our coupling works by redirecting and discarding certain edges in the random graph process in order to construct step-by-step the underlying tree to accommodate the parking process (using a Markovian or ``peeling'' construction). The geometry of the parked components at criticality is built by a ``multiplicative'' merging similar to the construction of the minimal spanning tree \cite{ABBGM13} but gives rise to random trees which we believe to converge towards the growth-fragmentation trees \cite{Ber15} that already appeared in the study of random planar maps \cite{BBCK18,BCK18}. This conjecture is further supported  by deep analogies between the enumeration of planar maps and that of fully parked trees with outgoing flux.

\paragraph{Parking on random trees.} Let us first recall the model of parking on a Cayley tree first studied in \cite{LaP16}. Consider a finite tree $ \mathfrak{t}$ with a root vertex. We interpret the vertices of  $\mathfrak{t}$ as being parking spots (each vertex can accommodate only one car) and we let cars arrive sequentially, independently and uniformly over the vertices of $ \mathfrak{t}$. Each car tries to park on its arrival node, unless the spot is taken in which case it drives towards the root of the tree in seek of the first available parking spot. If during its descent to the root vertex  no free spot is found, then the car exits the tree without parking, see Figure \ref{fig:parkingtikz}.

\begin{figure}[!h]
 \begin{center}
\begin{tikzpicture}[ xscale = 0.7, yscale = 0.95, node distance=1cm]

\tikzstyle{every state}=[fill=black!10,draw=none, text=blue]
\node[state] at (0,0)  (P1) {};
\node[state] at (-1,2) (P2)   {};
\node[state] at (1,2) (P3)  {};
\node[state] at (1,4) (P4) {};
\node[state] at (3,4) (P5)  {};
\node[state] at (5,4) (P6)  {};
\node[state] at (0,6) (P7)   {};
\node[state] at (2,6) (P8)   {};
\node[state] at (5,6) (P9)   {};
\node[state] at (-1,4) (P10)   {};
\node[state] at (-3,4) (P11)   {};

\path[->, >=stealth, shorten >=1pt]
(P2) edge (P1)
(P3) edge (P1)
(P4) edge (P3)
(P5) edge (P3)
(P6) edge (P3)
(P7) edge (P4)
(P8) edge (P4)
(P9) edge (P6)
(P10) edge (P2)
(P11) edge (P2);

\draw (-2.1,2) node {\faCar};
\draw (-2.8,2) node {\faCar};
\draw (-3.5,2) node {\faCar};
\draw (-2.1,1.5) node {$2$};
\draw (-2.8,1.5) node {$4$};
\draw (-3.5,1.5) node {$8$};

\draw (0,6.75) node {\faCar};
\draw (0.5,6.75) node {$6$};

\draw (3,4.75) node {\faCar};
\draw (3,5.25) node {\faCar};
\draw (3.5,4.75) node {$1$};
\draw (3.5,5.25) node {$3$};

\draw (5,6.75) node {\faCar};
\draw (5,7.25) node {\faCar};
\draw (5.5,6.75) node {$5$};
\draw (5.5,7.25) node {$9$};

\draw (6.1,4) node {\faCar};
\draw (6.6,4) node {$7$};
\draw (1, -1.75) node [white] {$9$};
\end{tikzpicture}
\qquad
\begin{tikzpicture}[ xscale = 0.7, yscale = 0.95, node distance=1cm]
\tikzstyle{every state}=[fill=black!10,draw=none, text=black]
\node[state] at (0,0)  (P1) {};
\node[state] at (-1,2) (P2)   {};
\node[state] at (1,2) (P3)  {};
\node[state] at (1,4) (P4) {};
\node[state] at (3,4) (P5)  {};
\node[state] at (5,4) (P6)  {};
\node[state] at (0,6) (P7)   {};
\node[state] at (2,6) (P8)   {};
\node[state] at (5,6) (P9)   {};
\node[state] at (-1,4) (P10)   {};
\node[state] at (-3,4) (P11)   {};
\path[->, >=stealth, shorten >=1pt]
(P2) edge (P1)
(P3) edge (P1)
(P4) edge (P3)
(P5) edge (P3)
(P6) edge (P3)
(P7) edge (P4)
(P8) edge (P4)
(P9) edge (P6)
(P10) edge (P2)
(P11) edge (P2);

\tikzstyle{every state}=[fill=black!30,draw=none, text=black]
\node[state] at (3,4) (P5)  {\faCar};
\draw (3.9,4) node {$1$};
\node[state] at (-1,2) (P2)   {\faCar};
\draw (-0.1,2) node {$2$};
\node[state] at (1,2) (P3)  {\faCar};
\draw (1.9,2) node {$3$};
\node[state] at (0,0)  (P1) {\faCar};
\draw (0.9,0) node {$4$};
\node[state] at (5,6) (P7)   {\faCar};
\draw (5.9,6) node {$5$};
\node[state] at (0,6) (P7)   {\faCar};
\draw (0.9,6) node {$6$};
\node[state] at (5,4) (P6)  {\faCar};
\draw (5.9,4) node {$7$};
\draw (0.5,-1.25) node {\faCar};
\draw (0.5, -1.75) node  {$8$};
\draw (1.2,-1.25) node {\faCar};
\draw (1.2, -1.75) node  {$9$};
\draw[->]  (2.1,-1.25) -- (3.5, -1.25);

\path[->, >=stealth, shorten >=1pt, color=DarkRed]
(P5) edge [ultra thick] node[left, near start] {1}(P3)
(P2) edge [ultra thick] node[left] {2}(P1)
(P9) edge [ultra thick] node[left] {1}(P6)
(P6) edge [ultra thick] node[below right] {1}(P3)
(P3) edge [ultra thick] node[right] {1} (P1);
\end{tikzpicture}
 \caption{\label{fig:parkingtikz}On the left,  a rooted tree with $11$ vertices (the root vertex is the bottom vertex) where all edges are oriented towards the root vertex, together with $9$ cars  arriving on its vertices. On the right, the result of the (sequential) parking of the $9$ cars. The flux of cars along each edge is indicated. Notice that two cars did not manage to park and exited the tree.}
 \end{center}
 \end{figure}

Of course when the underlying tree is a discrete line, this corresponds to the famous one-dimensional parking process of Konheim \& Weiss \cite{konheim1966occupancy} which is now part of the folklore in probability \cite{chassaing2002phase}. The study of parking on more general trees was only recently initiated by Lackner \& Panholzer \cite{LaP16}  where the underlying tree was a uniform Cayley tree of fixed size rooted at a uniform vertex (see also \cite{butler2017parking,king2019parking,king2019prime} for related works in combinatorics). Recall that a Cayley tree of size $n$ is a (unordered) tree over the labeled vertices $\{1,2, \dots , n\}$. This model was later studied from a probabilist angle in \cite{GP19}. Since then, a body of work with an increasing level of generality has emerged \cite{panholzer2020parking,CG19,CH19,bahl2021diffusion} ultimately considering critical conditioned Bienaym\'e--Galton--Watson tree (with finite variance) for the underlying tree and independent car arrivals whose laws may depend on the degree of the vertices \cite{contat2020sharpness}. See \cite{BBJ19,chen2019derrida} for the case of supercritical trees. In this broad context, it was shown  that  a sharp \emph{phase transition} appears for the parking process: there is a critical ``density'' of cars  (depending on the combinatorial details of the model) such that below this density, almost all cars manage to park, whereas above this density, a positive proportion of cars do not find a parking spot. See \cite{contat2020sharpness,CH19} for precise statements. The goal of this work is to provide scaling limits for the \emph{critical} and \emph{near-critical} dynamics of the parking process in the special case of uniform Cayley trees with i.i.d\ uniform car arrivals where the critical density is $\frac{1}{2}$, see \cite{LaP16}. Perhaps surprisingly, this will be done by relating the model to the ubiquitous Erd{\H{o}}s--R\'enyi random graph.

\paragraph{Frozen Erd{\H{o}}s--R\'enyi.} Fix $n \geq 1$. Over the vertex set $\{1,2, \dots , n \}$, consider for $i \geq 1$ independent identically distributed oriented edges $\vec{E}_{i}=( X_{i}, Y_{i})$  where both endpoints are independent and uniform over $\{1,2, \dots , n\}$. We denote by $ E_{i}$ the unoriented version of the oriented edge $ \vec{E}_{i}$. Notice in particular that we may have $X_{i} = Y_{i}$  and  $\vec{E}_{i}= \vec{E}_{j}$ for $i \ne j$. For $m \geq 0$, the \emph{Erd{\H{o}}s-R\'enyi} random graph\footnote{Actually, in many places in the literature, the Erd{\H{o}}s--R\'enyi random graph is a simple  graph where self-loops and multiple edges are forbidden, but this small variant is more natural probabilistically as it was noticed already in  \cite[Section 2.3.1]{bhamidi2014augmented} or \cite{limic2017playful}.} is the random multigraph $G(n,m)$ whose vertex set is $ \{1,2,\dots , n\}$ and whose unoriented edge set is the multiset $\llbracke E_{i} : 1 \leq i \leq m \rrbracke$.  

%The surplus of a \emph{connected} multigraph $ \mathfrak{g}$ is defined as $\# \mathrm{Edges}( \mathfrak{g}) - \# \mathrm{Vertices}( \mathfrak{g})+1$ and corresponds to the number of ``cycles'' created when building $ \mathfrak{g}$.

 We now define the \emph{frozen} Erd{\H{o}}s--R\'enyi process $(F(n,m) : m \geq 0)$, which is obtained from the above graph process $( G(n,m) : m \geq 0)$ by ``freezing'' or more precisely slowing down the components which are not trees. The vertices of $F(n,m)$ will be of two types: standard ``white", or frozen ``blue" vertices. Initially $F(n,0)$ is made of the $n$ labeled white vertices $\{1,2, \dots , n\}$. As in the $(G(n,m) : m \geq 0)$ process, we let the (same) edges $\vec{E}_{i} = (X_{i}, Y_{i})$ arrive sequentially for $i \geq 1$ but discard some of them and color the vertices in $F(n, \cdot)$ according to the following rule, see Figure \ref{fig:transitions}:  for $m \geq 1$ 
\begin{itemize}
\item if both endpoints of $E_{m}$ are white vertices then the edge $E_{m}$ is added to $F(n,m-1)$ to form $F(n,m)$. If this addition creates a cycle in the graph then the vertices of its component are declared frozen and colored in blue.
\item if both endpoints of $E_{m}$ are blue (frozen vertices), then $E_{m}$ is discarded.
\item if $E_{m}$ connects a white and a blue vertex, then $E_{m}$ is kept if $ \vec{E}_{m}$ goes from the white to the blue vertex. If so, the new connected component is declared frozen and colored in blue.
\end{itemize}

\begin{figure}[!h]
 \begin{center}
 \includegraphics[width=16cm]{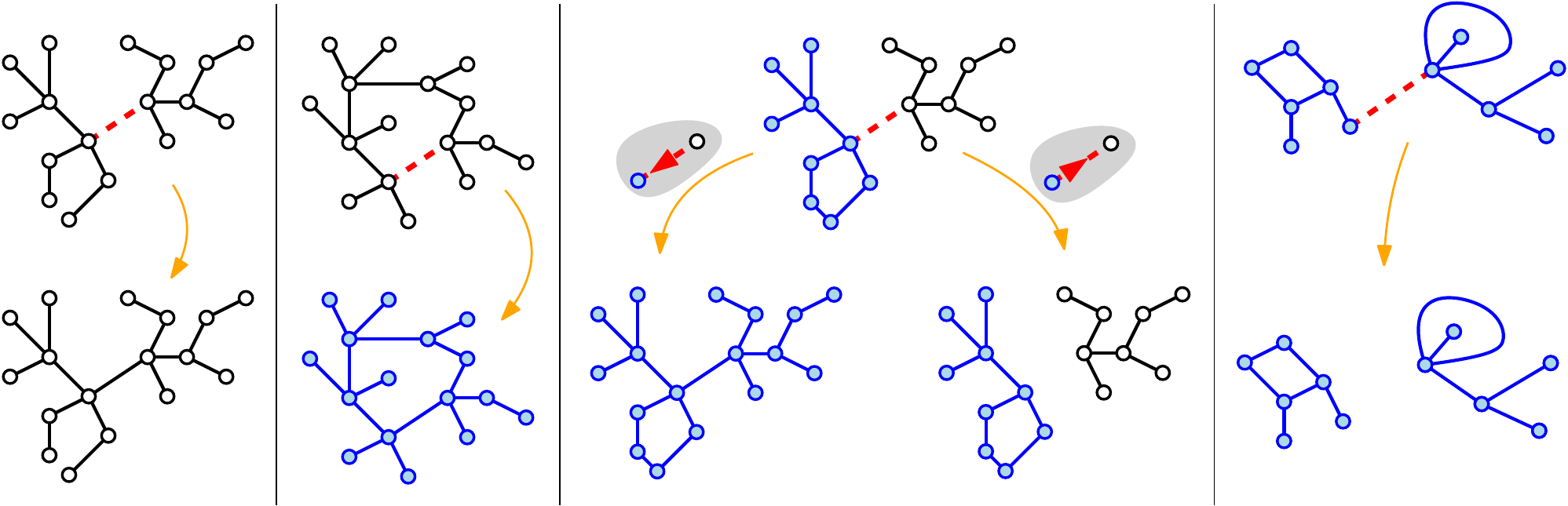}
 \caption{ \label{fig:transitions}Illustration of the transitions in the frozen Erd{\H{o}}s-R\'enyi process. The new edge to be examined is in dotted red. If this edge appears between two white tree-components, it is kept (first and second figures on the left). When a cycle is created, the component is colored in blue and becomes frozen (second figure). An edge appearing between a frozen blue and  a white component is kept if it goes from white to blue and the entire new component is declared frozen.  All edges between frozen components are discarded.}
 \end{center}
 \end{figure}
 A more general version of the frozen process depending on a parameter $p \in [0,1]$ can be defined (see Section \ref{sec:generalfrozen}) by keeping edges between white and blue components with probability $p$.
 Different models of ``frozen'' percolation have already been considered on the Erd{\H{o}}s--R\'enyi random graph \cite{crane2015cluster,rath2009mean,rath2009erdos} or on other graphs \cite{aldous2000percolation,frozen,kiss2015frozen}, but to the best of our knowledge, the above random graph processes are new. One interesting feature of the frozen process is that for any $m \geq 0$, conditionally on the frozen part of $F(n,m)$, the ``forest part'' made of the white components is a uniform forest given its number of vertices and edges, see Proposition \ref{prop:freeforest}. We shall refer to this property as the \emph{free forest property}. The geometry of large critical uniform random forests has been studied in particular by Luczak \cite{luczak1992components} using counting results of R\'enyi and Britikov \cite{britikov1988asymptotic,renyi1959some} and more recently by  Martin \& Yeo \cite{martin2018critical} using an exploration process converging to an inhomogeneous diffusion with reflecting boundary. We shall revisit and shed new light on those results using random walks coding and $3/2$-stable (conditioned) stable processes, see Section \ref{sec:scalingforest}.
 
   In the case of the Erd{\H{o}}s--R\'enyi random graph, Aldous proved in a famous paper \cite{aldous1997brownian} that the process of the component sizes in $G(n, m)$ exhibits a phase transition in the critical window $ m = \frac{n}{2} + \frac{\lambda}{2}  n^{2/3}$ for $\lambda \in \mathbb{R}$. The same critical window will appear in this work and so  to lighten notation, when we have a discrete process $(X(n,m) : m \geq 0)$ where $n$ denotes the fixed ``size'' of the system and $m=1,2,3,\dots$ is an evolving parameter, we shall denote its continuous time analog by a mathrm letter
  \begin{eqnarray}\label{eq:notationcriticalwindow} \mathrm{X}_n( \lambda) = X\left(n, \left\lfloor \frac{n}{2} + \frac{\lambda}{2} n^{2/3} \right\rfloor \vee 0\right), \quad \mbox{ for } \lambda \in \mathbb{R}.  \end{eqnarray}
The parameter $\lambda$ will often be called the ``time" parameter and will enable us to compare processes of different sizes in the same time window. This will e.g.~apply to $G(n,m)$ and  $F(n,m)$ to yield $ \mathrm{G}_n(\lambda)$ and $ \mathrm{F}_{n}(\lambda)$. With this notation, Aldous proved that after renormalizing the component sizes of $ \mathrm{G}_{n}( \lambda)$ by $n^{-2/3}$, the resulting process converges to the multiplicative coalescent which is a random c\`adl\`ag  process $( \mathscr{M}( \lambda) : \lambda \in \mathbb{R})$ with values in $ \ell^{2}$ intuitively starting from ``dust'' as time $-\infty$ and such that every pair of  particles of mass $x$ and $y$ merges to a new particle of mass $x+y$ at a rate $xy$, see Figure \ref{fig:ratefrozen}. Using Aldous' work \cite{aldous1997brownian} and its extensions \cite{bhamidi2014augmented,broutin2016new}, we are able to prove (Theorem \ref{thm:FMC}) a similar result for the component sizes in $  \mathrm{F}_{n}( \lambda)$ and refer to the scaling limit $( \FM( \lambda) : \lambda \in \mathbb{R})$ as the \emph{frozen multiplicative coalescent}. This however requires careful cutoffs and controls since the dynamics of the frozen Erd{\H{o}}s--R\'enyi is not ``monotonous". Similar ideas have been used by Rossignol in \cite{rossignol2021scaling} to define a split/merge stationary dynamics on the scaling limit of critical random graphs.
\begin{figure}[h!] 
\begin{center}
\includegraphics[width=15cm]{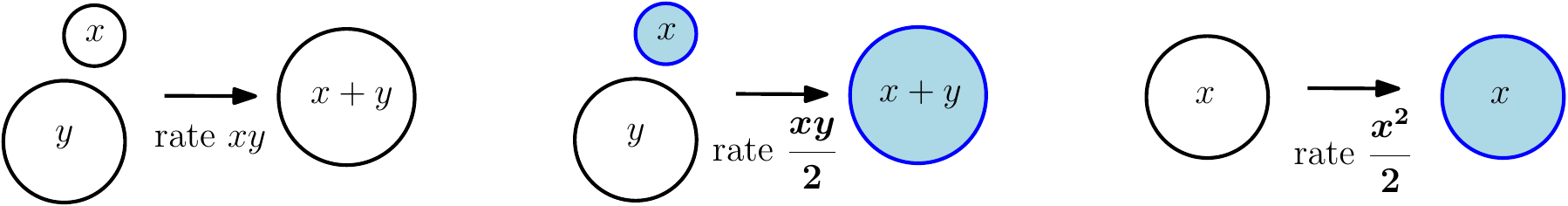}
\caption{Dynamics of the frozen multiplicative coalescent $ \FM$. The interaction between standard ``white'' particles is the same as in the multiplicative coalescent (left), but the interaction between white and frozen ``blue'' particles is slowed down (middle). Besides, a white particle can become blue at rate proportional to its mass squared (right).  \label{fig:ratefrozen}}
\end{center}
\end{figure}

To be a bit more precise, the particles of the frozen multiplicative coalescent $ \FM ( \lambda)$ at time $ \lambda$ are of two types: the frozen (blue) particles whose decreasing masses are in $\ell^{1}$ and non-frozen (white) particles whose decreasing masses form a sequence in $\ell^{2}$. Then $ \FM$ evolves heuristically according to the same dynamics as that of { $\mathrm{F}_n( \cdot)$}: every pair of white particles of mass $x$ and $y$ merge to a new white particle of mass $x+y$ at a rate $xy$, whereas a blue particle of mass $x$ merges with a white particle of mass $y$ to form a blue particle of mass $x+y$ at a rate $ \frac{xy}{2}$. Also, a white particle of mass $x$ becomes frozen ``if it creates an internal cycle'' which appears with a rate $\frac{x^{2}}{2}$, see Figure \ref{fig:ratefrozen}.  The process $  \FM$ inherits a Markovian property from that of  $\mathrm{F}_n( \cdot)$ and in particular the process of the total mass of the frozen particles is a Feller pure-jump process with an explicit jump kernel close to that of a $\tfrac{1}{2}$-stable subordinator, see Proposition \ref{prop:fellerX}.  Since $ \FM$ is naturally coupled with the multiplicative coalescent, it gives a new perspective on the multiplicative coalescent (see Part \ref{sec:comments}).  We wonder whether the dynamics of the frozen multiplicative coalescent $ \FM$ can be described by ``merging the excursion lengths of random functions'' as it is the case for the multiplicative coalescent \cite{Armendariz,broutin2016new,limic2019eternal,Bravo} or its version with linear deletion \cite{martin2017rigid}.

\paragraph{Coupling parking on Cayley trees and the frozen Erd{\H{o}}s--R\'enyi.} As announced above, the main input of this paper is to construct an explicit coupling between the dynamical parking process on a uniform rooted Cayley tree ${T}_{n}$ and the frozen Erd{\H{o}}s--R\'enyi process $F(n,\cdot)$ so that the components match up.  On the tree side, this coupling consists in considering the underlying tree $ {T}_{n}$ as unknown and exploring its oriented edges one after the other to perform the parking process. To do this we develop a general Markovian or ``peeling'' exploration of Cayley trees (Section \ref{sec:markovpeeling}) similar to that of \cite{contat21surprizing,CurStFlour} and which may have further applications. To be a bit more precise, for $m \geq 0$ consider $$T_{ \mathrm{near}}(n,m) \subset T_{n}$$ the subforest of $T_{n}$ spanned by all the edges {emanating} from a vertex containing one of the first $m$ cars, see Figure \ref{fig:components}. Then we prove in Proposition \ref{prop:CouplingCayleyFrozen} that we can couple the parking process on $T_{n}$ with the frozen Erd{\H{o}}s--R\'enyi $F(n, \cdot)$ so that after merging the frozen components of $F(n, m)$ it has the same components as $T_{ \mathrm{near}}(n,m)$ --but the geometry inside the components is totally different--. See Section \ref{sec:couplingtrees} for details. The construction is easier to understand when the underlying tree $ {T}_{n}$ is replaced by a uniform random mapping $ {M}_{n}$ and we start with this case in Section \ref{sec:couplingmapping}.

As a direct consequence of our coupling we prove:
\begin{proposition}[Complete parking and acyclicity of $G(n,m)$]\label{prop:acyclicity}For $n \geq 1$ and $m \geq 0$ we have
$$ \mathbb{P}(m \mbox{ i.i.d.\ uniform cars manage to park on }T_{n}) \left( 1- \frac{m}{n}\right) = \mathbb{P}(G(n,m) \mbox{ is acyclic}).$$
\end{proposition}
Combining this proposition with classical tree enumeration going back to R\'enyi  \cite{renyi1959some} and Britikov \cite{britikov1988asymptotic}, we recover the counting results of Lackner \& Panholzer \cite[Theorems 3.2 \& 4.5 \& 4.6]{LaP16} which were derived using (sometimes delicate) analytic combinatorics and singularity analysis, see Section \ref{sec:enumconsq}. Another consequence concerns the scaling limit of the component sizes in the parking process: let us denote by $C_{i}(n,m)$ for $i \geq 1$  the non-increasing sizes (number of vertices) of the components of $T_{ \mathrm{near}}(n,m)$ of $ {T}_{n}$ when $m$ cars have arrived. We put the component of the root vertex aside and denote its size by $C_{*}(n,m)$. We also write $D(n,m)$ for the number of cars among the first $m$ that did not manage to park (the letter $D$ stands for ``discarded''). With our convention \eqref{eq:notationcriticalwindow} we prove:

\begin{theorem}[Dynamical scaling limit for {the} component sizes and {the} outgoing flux]  \label{thm:composantes}We have the following convergence in distribution for the Skorokhod topology on $ \mathbb{C}\mathrm{adlag}( \mathbb{R}, \ell^{2} \times \mathbb{R}_{+} \times \mathbb{R}_{+})$ 
$$ \left( \begin{array}{l} n^{-2/3} \cdot  \mathrm{C}_{n,i} \left(\lambda \right), \ \ i \geq 1\\
n^{-2/3} \cdot  \mathrm{C}_{n,*} \left(  \lambda \right)\\
n^{-1/3} \cdot  \mathrm{D}_{n} \left(  \lambda \right)\\
\end{array}\right)_{\lambda \in \mathbb{R}}  \xrightarrow[n\to\infty]{(d)} \quad \left( \begin{array}{l}  \mathscr{C}_{i}(\lambda), \ \ i \geq 1\\
  \mathscr{C}_{*}( \lambda)\\
  \mathscr{D}( \lambda)\\
\end{array}\right)_{\lambda \in \mathbb{R}.}$$
The processes $ \mathscr{C}_{i}, \mathscr{C}_{*}$ and $ \mathscr{D}$ are built from the frozen multiplicative coalescent as follows:
\begin{itemize}
\item $ (\mathscr{C}_{i}( \lambda) : i \geq 1)$ is the non-increasing sequence of  masses of the white particles in $ \FM( \lambda)$,
\item $ \mathscr{C}_{*}( \lambda)$ is the sum of the masses of the blue particles in $ \FM( \lambda)$,
\item $ \displaystyle \mathscr{D}( \lambda) = \frac{1}{2} \int_{-\infty}^{\lambda} \mathrm{d}s \,  \mathscr{C}_{*}(s)$.
\end{itemize}
\end{theorem}
\begin{figure}[!h]
 \begin{center}
 \includegraphics[height=6cm]{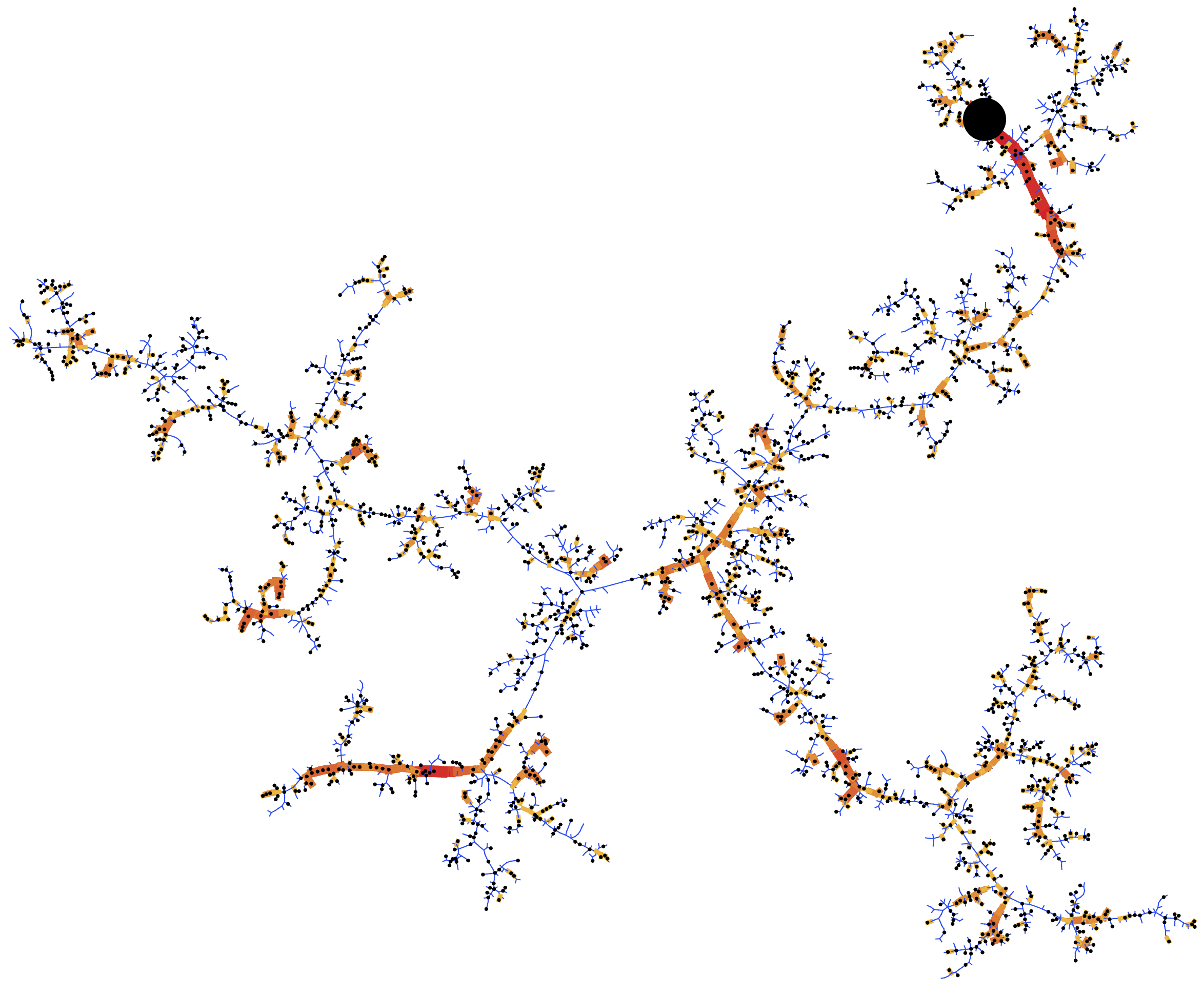} \hspace{2cm}
  \includegraphics[height=6cm]{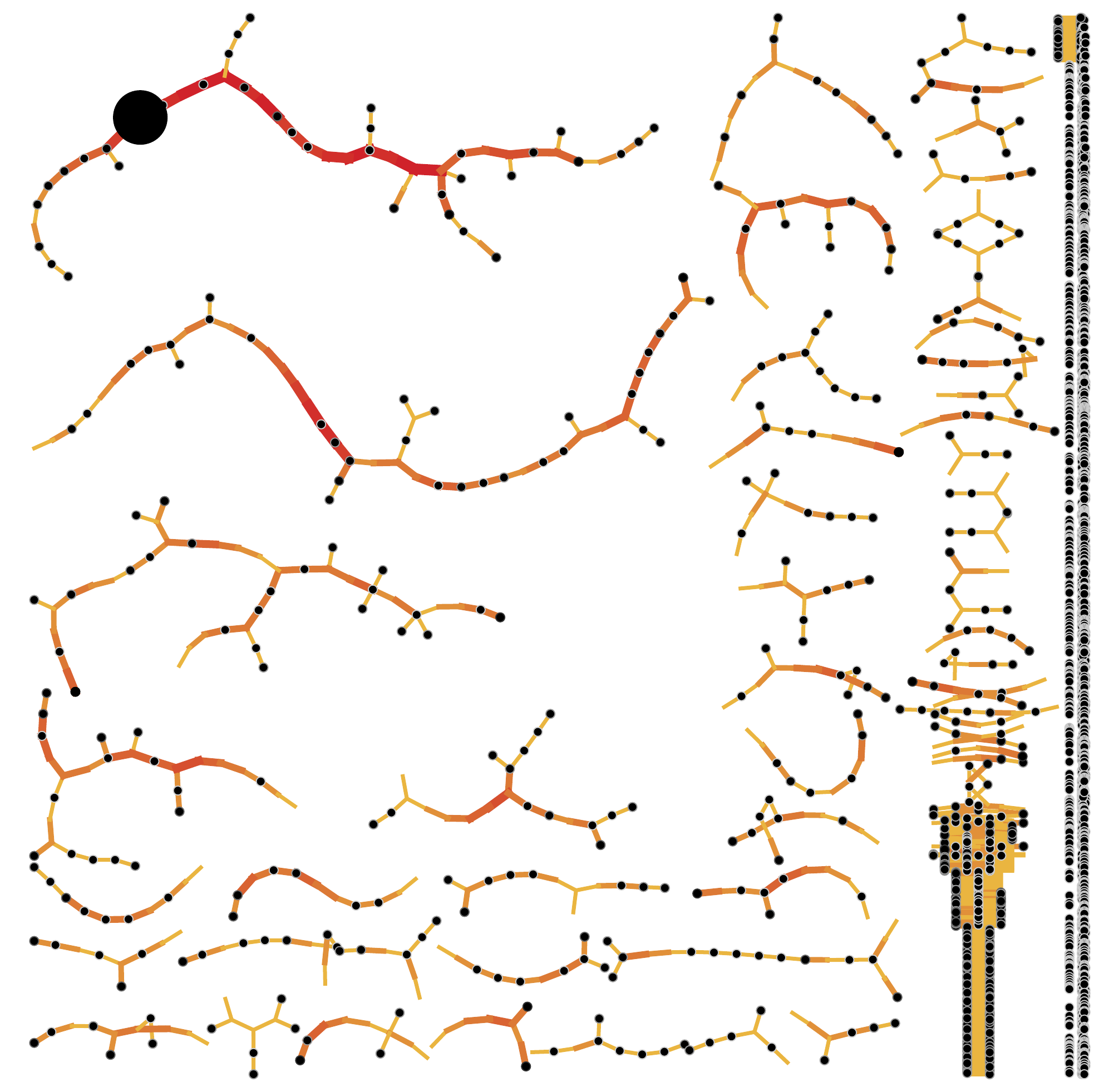}
 \caption{(Left) A simulation of a critical parking on $ T_{5000}$ with $2500$ cars. The colors and widths of the edges indicate the flux of cars going through them. The root of the tree is represented by a {black} disk. (Right) The decomposition of the same tree into its parked components. \label{fig:decomp}}
 \end{center}
 \end{figure}
 
Notice in particular that, in the critical window $m = \frac{n}{2} + O(n^{2/3})$, the flux of cars that did not manage to park in $T_{n}$ is of order $n^{1/3}$ whereas the size of the largest cluster of parked cars is of order $n^{2/3}$. See Figure \ref{fig:decomp} for a simulation of a critical parking and its decomposition into parked components. Our theorem also holds for different versions of components e.g. if we only keep the edges between parked vertices, see Section \ref{sec:proofthm1}.
\begin{remark}[Dynamical parking and coalescence] It is striking to notice that Konheim \& Weiss' parking on the line is related to the \emph{additive} coalescent (see \cite{chassaing2002phase,broutin2016new}), whereas the essence of our findings is that the parking process on random Cayley trees obeys a modified \emph{multiplicative} coalescence rule. 
\end{remark}

\paragraph{Geometry of fully parked trees and Bertoin's growth-fragmentation processes.}Theorem \ref{thm:composantes} describes the phase transition of the parking in terms of the \emph{sizes} of the parked components and outgoing flux of cars in $ {T}_{n}$. But one can wonder about the \emph{geometry} of the parked components and the flux of cars on its edges. It is not hard to see (see Proposition \ref{prop:FPT/AFPTuniform}) that except for the component of the root vertex, conditionally on their sizes $N$,  those components are (after relabeling of the vertices and cars) uniform \emph{fully parked trees}, i.e.\ random {uniform} rooted Cayley tree $T_{N}$ with $N$ vertices carrying $N$ labeled cars conditioned on the (unlikely) event that all cars successfully park  on $T_{N}$.  In what follows, we shall consider a slight variant of this model and denote by $ {P}_{N}$ a uniform \emph{nearly} parked tree of size $N \geq 1$ which is a uniform rooted Cayley trees of size $N$ carrying $N-1$ labeled cars conditioned on the event that the root $\rho$ stays void after parking, see Figure \ref{fig:NPT} and Figure \ref{fig:components}.

\begin{figure}[!h]
 \begin{center}
 \includegraphics[width=8.5cm]{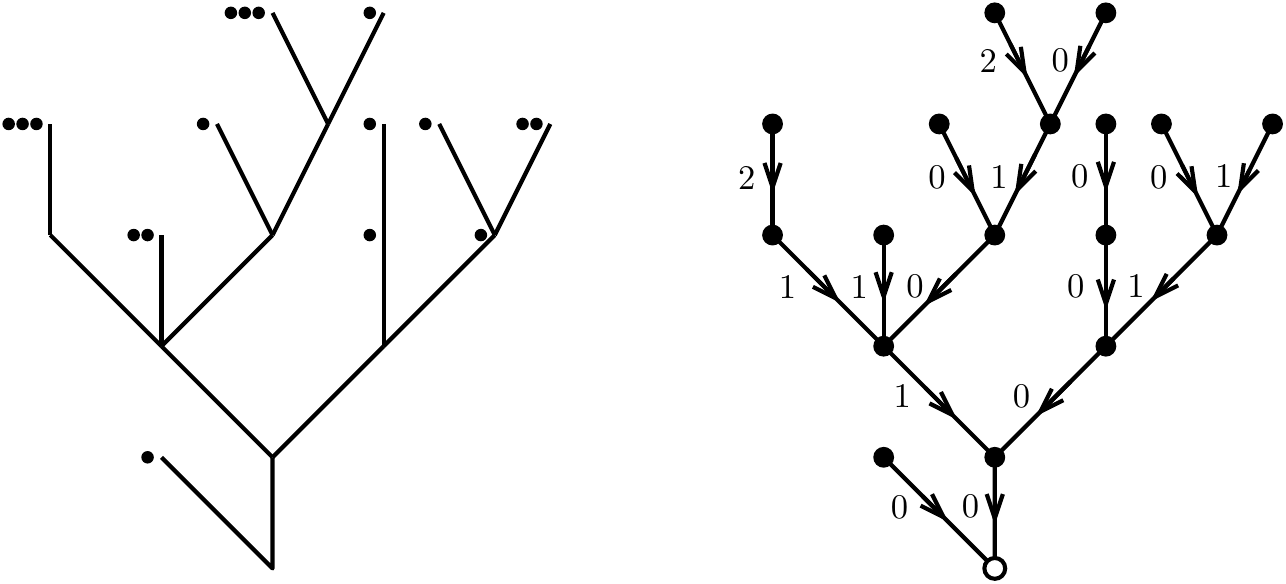} \hspace{1cm} \includegraphics[width=5.5cm]{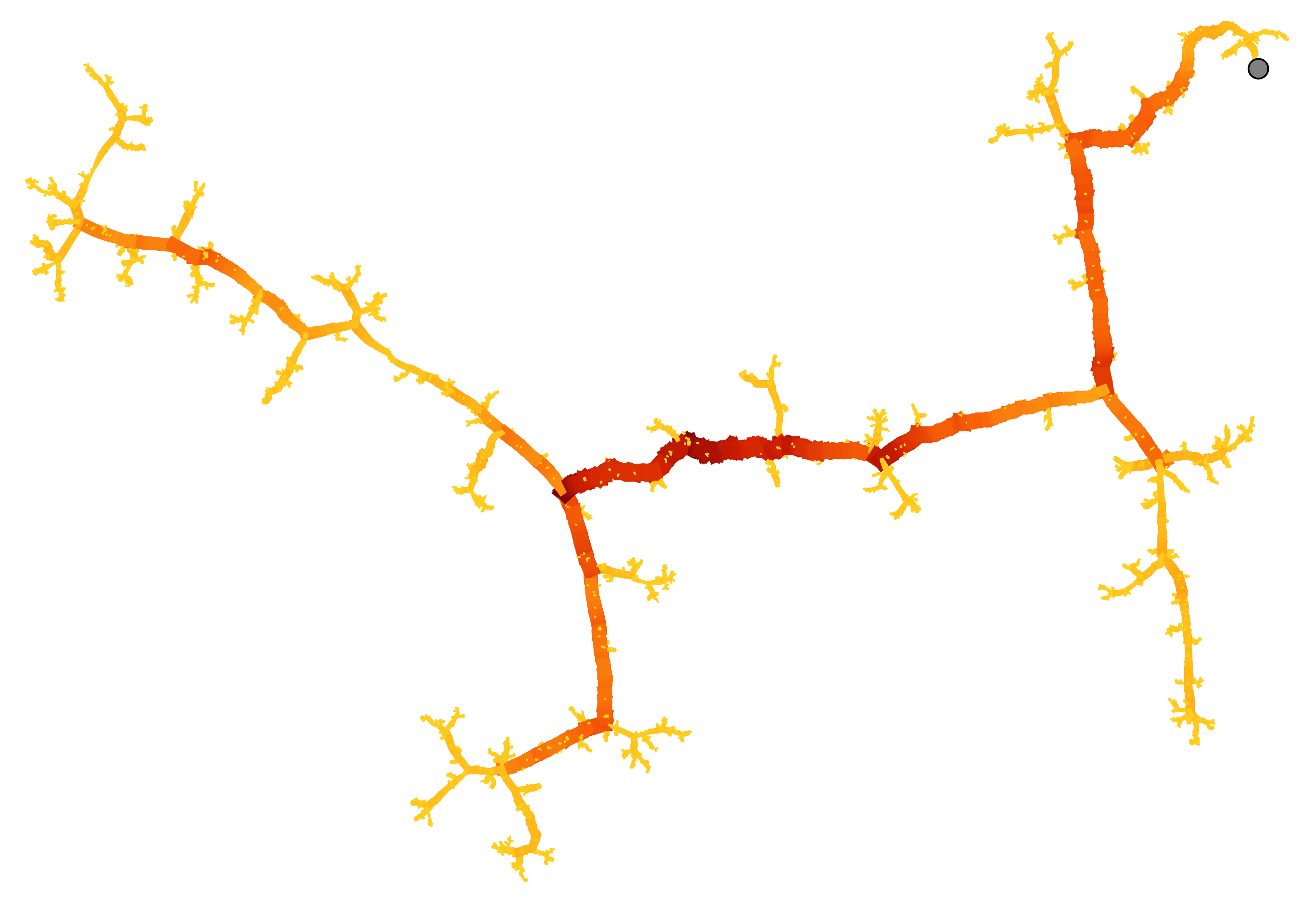}
 \caption{ \label{fig:NPT} A nearly parked tree with $18$ vertices and $17$ cars that manage to park while leaving the root empty (the labels of the vertices and cars are not displayed for the sake of clarity). Right: A simulation of a large uniform nearly parked tree of size $15000$, where the thickness and color of the edges indicate the flux of cars going through them.}
 \end{center}
 \end{figure}

The conditioning imposed on the parking configuration makes the geometry of $ {P}_{N}$ very different from that of a uniform Cayley tree $T_{N}$: heuristically they are more elongated. When restricted to nearly parked trees, our coupling  gives a construction of a  nearly parked tree $ P_{N}$ from a uniform Cayley tree $ T_{N}$ of size $N$ whose edges are labeled from $1$ up to $N-1$ and oriented randomly (see Section \ref{sec:couplingback} for details). Using this we are able to prove:

\begin{proposition}[Typical height of $P_N$] \label{prop:height} The mean height of a nearly parked tree of size $N$ is
$$  \frac{1}{N} \mathbb{E}\left[\sum_{x \in  \mathrm{Vertices}(P_{N})} \mathrm{d}_{  \mathrm{gr}}^{P_{N}}( \rho, x)\right] = \sum_{h=1}^{N-1} {N \choose h+1} \left( \frac{h+1}{N}\right)^{2} N^{1-h}  \left(\frac{1}{2}\right)_{h} \quad \underset{N \to \infty}{\sim} \quad  \frac{\Gamma (3/4)}{2^{1/4} \sqrt{\pi}}\cdot N^{3/4},$$
where $(x)_{a} = x (x+1) \cdots (x+a-1)$ is the Pochhammer symbol. 
\end{proposition}
A nearly parked tree $P_{N}$ naturally comes with a labeling $(\phi_{N}(e) : e \in \mathrm{Edges}(P_{N}))$ on its edges counting the number of cars going through that edge in the parking process, see Figure \ref{fig:NPT}. An Abelian property actually shows that this labeling does not depend on the order in which we have parked the cars. In particular, the sum $ \sum \phi_{N}(e)$, corresponding to the total distance travelled by the cars, is also invariant under relabeling of the cars. We compute the expectation of this quantity:
\begin{proposition}[Total traveled distance $N^{5/4}$] \label{prop:totaldistance}The mean total distance travelled by the cars in a uniform nearly parked tree $P_{N}$ is
$$ \mathbb{E}\left[\sum_{e \in \mathrm{Edges}(P_{N})} \phi_{N}(e)\right]=%\frac{1}{2} \sum_{i=1}^{n-1} \sum_{h=1}^{n-2} {i-1 \choose h} (h+2) n^{-h} \left(\frac{1}{2}\right)_{h} = 
\frac{1}{2}\sum_{h=1}^{N-2} {N-1 \choose h+1}  (h+2) N^{-h} \left(\frac{1}{2}\right)_{h}\quad \underset{N \to \infty}{\sim} \quad  \frac{\Gamma(1/4)}{2^{5/4} \sqrt{ \pi}} \cdot N^{5/4}.
$$
\end{proposition}

The heuristic picture suggested by the above two results is that a uniform nearly parked tree $P_{N}$ is of height $ N^{3/4}$ and that the flux of cars along ``long branches'' of $P_{N}$ is of order $ N^{1/2}$ so that $N^{3/4} \cdot N^{1/2} = N^{5/4}$ is the total distance driven by the cars. This is coherent with the fact that in the critical window, the  outgoing flux at the bottom of the root component of size $n^{2/3}$ is of order $n^{1/3} =  (n^{2/3})^{1/2} $ by Theorem \ref{thm:composantes}.

In fact, we believe that rescaled uniform nearly parked trees converge after normalization towards the growth-fragmentation trees that already appeared in the study of scaling limits of random planar maps and the Brownian sphere, see  \cite{BBCK18,BCK18,le2020growth} or \cite[Chapter 14.3.2]{CurStFlour}. Those are ``labeled continuum random trees''   describing the genealogy of the masses of individuals in a family of living cells. These cells evolve independently one from the other, and the dynamics of the mass of a typical cell is governed by (a variant of) a $3/2$-stable spectrally negative L\'evy process. Each negative jump-time for the mass is interpreted as a birth event, in the sense that it is the time at which a daughter cell is born, whose initial mass is precisely given by the absolute height of the jump (so that conservation of mass holds at birth events).  With some work, one can define a version $( \mathcal{T}, (\phi(x) : x \in \mathcal{T}))$ of those random labeled trees conditioned to start from a single cell of mass $0$ and to have a total ``volume'' of $1$, see \cite{BertoinCurienEtAl,bertoin2019conditioning} for details. We propose the following conjecture:
\begin{conjecture}  \label{conjectureGFT} We have the following convergence in distribution for some $c_{1}, c_{2}>0$
$$ \left( \left(P_{N}, \frac{\mathrm{d}_{ \mathrm{gr}}^{P_{N}}}{N^{{3/4}}}\right) ;  \left( \frac{\phi_{N}(e)}{ \sqrt{N}} : e \in \mathrm{Edges}( P_{N})\right)\right) \xrightarrow[N\to\infty]{(d)} \left(   c_{1}\cdot \mathcal{T}, (c_{2} \cdot  \varphi(x) : x \in \mathcal{T})\right),$$
see  \cite{BertoinCurienEtAl} for the topology one may want to use (which in particular implies the Gromov--Hausdorff convergence on the first coordinate).
\end{conjecture}
See Section \ref{sec:GFcomments} for more details. Apart from the above propositions, one strong support for this conjecture is the fact that fully parked trees satisfy a Tutte-like equation by splitting the flux at the root which is reminiscent of that appearing in the realm of planar maps, see Section \ref{sec:GF}. In particular, if correct, combining the above conjecture with our coupling construction would uncover a ``dynamical'' construction of growth-fragmentation trees which is similar in spirit to that of the minimal spanning tree \cite{ABBGM13} but with a redirection of the edges. %We hope to address those questions in forthcoming works.
\medskip

The paper is organized in two main parts.  The first one is purely in the discrete setting and presents the coupling construction as well as its enumerative and geometric consequences. The second one focuses on scaling limits and involves the multiplicative coalescent of Aldous as well as stable L\'evy processes. For the reader's convenience, we provide an index of the main notations at the end of the paper.

\bigskip

\textbf{Acknowledgements.} We acknowledge support from ERC 740943 \emph{GeoBrown}. We thank Lin\-xiao Chen, Armand Riera and especially Olivier H\'enard for several motivating discussions during the elaboration of this work.

\tableofcontents
\part{Discrete constructions}

This part is devoted to the discrete constructions and couplings. We consider non-necessarily connected finite \emph{multigraphs} $ \mathfrak{g}$, i.e.\ self-loops and multiple edges are allowed. The number of vertices of $ \mathfrak{g}$ will be denoted by $\| \mathfrak{g}\|_{\bullet}$, its number of edges by $\| \mathfrak{g}\|_{\edge}$ and the vertex set is often taken to be $\{1,2, \dots , \| \mathfrak{g}\|_{\bullet}\}$. The vertices of our graphs will often be colored in two colors, white (standard) or blue (frozen), and we denote by $\| \mathfrak{g}\|_{\circ}$ and $\| \mathfrak{g}\|_{\bluecirc}$ the number of vertices of each color and by $[ \mathfrak{g}]_{\circ}$ and $ [ \mathfrak{g}]_{\bluecirc}$ the graphs induced on vertices of each color. The surplus of a \emph{connected} multigraph $ \mathfrak{g}$ is defined as $\| \mathfrak{g}\|_{\edge} -  \| \mathfrak{g}\|_{\bullet}+1$ and corresponds to the number of ``cycles'' created when building $ \mathfrak{g}$. The subgraph made of the components without surplus is called the \emph{forest part} of $ \mathfrak{g}$ and denoted by $[ \mathfrak{g}]_{ \mathrm{tree}}$. 

In the rest of the paper $T_n$ is a uniform rooted Cayley tree with $n$ labeled vertices $\{1,2, \dots ,  n\}$. We shall always see the edges of $T_{n}$ as oriented towards the root vertex. For $i \geq 1$, we let $X_i,Y_i$ be i.i.d.\ uniform points of $\{1,2, \dots , n\}$ so that $ \vec{E}_i= (X_i, Y_i)$ can be seen as i.i.d.\ uniform oriented edges (self-loops are allowed). In the sequel $T_n$ will always be independent of $(X_i : i \geq 1)$ but not of $(Y_i: i \geq 1)$...

\section{Warmup}
In this section we introduce the main ingredients for the coupling of the parking process on Cayley trees with the Erd{\H{o}}s--R\'enyi random graph. We shall first describe the different notions of components in the parking process. We then present the coupling in the case of the parking on random mappings (Proposition \ref{coupling:mapping}) for which the proof is easier to understand. Note that Lackner \& Panholzer \cite{LaP16} already noticed striking similarities between parking on mappings and parking on Cayley trees.

\subsection{Components and versions of parked trees}
\label{sec:components}
\label{sec:FPTandNPT}

\medskip

Fix a random uniform rooted Cayley tree $T_{n}$ with $n$ labeled vertices $\{1,2, \dots ,  n\}$ and independently of it, let $X_i \in \{1,\dots, n\}$ be uniform i.i.d.\ car arrivals for $i \geq 1$. For $m \geq 0$, we proceed to the parking of the first $m$ cars as explained in the introduction and consider the clusters of parked cars. There are several possible notions to define those clusters and  let us go from the more restrictive to the more permissive, see Figure \ref{fig:components}:

\begin{itemize}
\item If we only keep the edges (and neighboring vertices) having a positive flux of cars (that is through which at least one car had to go), then we obtain the strong components, i.e.\ a subforest $T_ \mathrm{strong}(n,m) \subset T_n$. The components of $T_ \mathrm{strong}(n,m)$ different from the component containing the root vertex are\footnote{after an increasing relabeling of its vertices and cars} either isolated empty vertices or  \textbf{strongly parked trees} which are rooted Cayley trees of size $N$ carrying $N$ labeled cars, so that all cars manage to park (outgoing flux $0$), and such that all edges have a positive flux of cars. 
\item If we only keep the edges so that both extremities are occupied spots, then we obtain the full components $T_ \mathrm{full}(n,m)$.  The components of $T_{ \mathrm{full}}(n,m)$ different from the component containing the root vertex are either isolated empty vertices or \textbf{fully parked trees} which are rooted Cayley trees of size $N$ carrying $N$ labeled cars and so that all cars manage to park (outgoing flux $0$).
\item Finally, if we only keep the edges emanating from the occupied vertices, then we obtain the near components $ T_ \mathrm{near}(n,m)$. The components of $T_{ \mathrm{near}}(n,m)$ different from that of the root vertex are  \textbf{nearly parked trees} i.e.\ rooted Cayley trees of size $N$ carrying $N-1$ labeled cars and so that the root vertex stays empty after parking the cars. 
\end{itemize}
\begin{figure}[h!]
 \begin{center}
 \includegraphics[width=15cm]{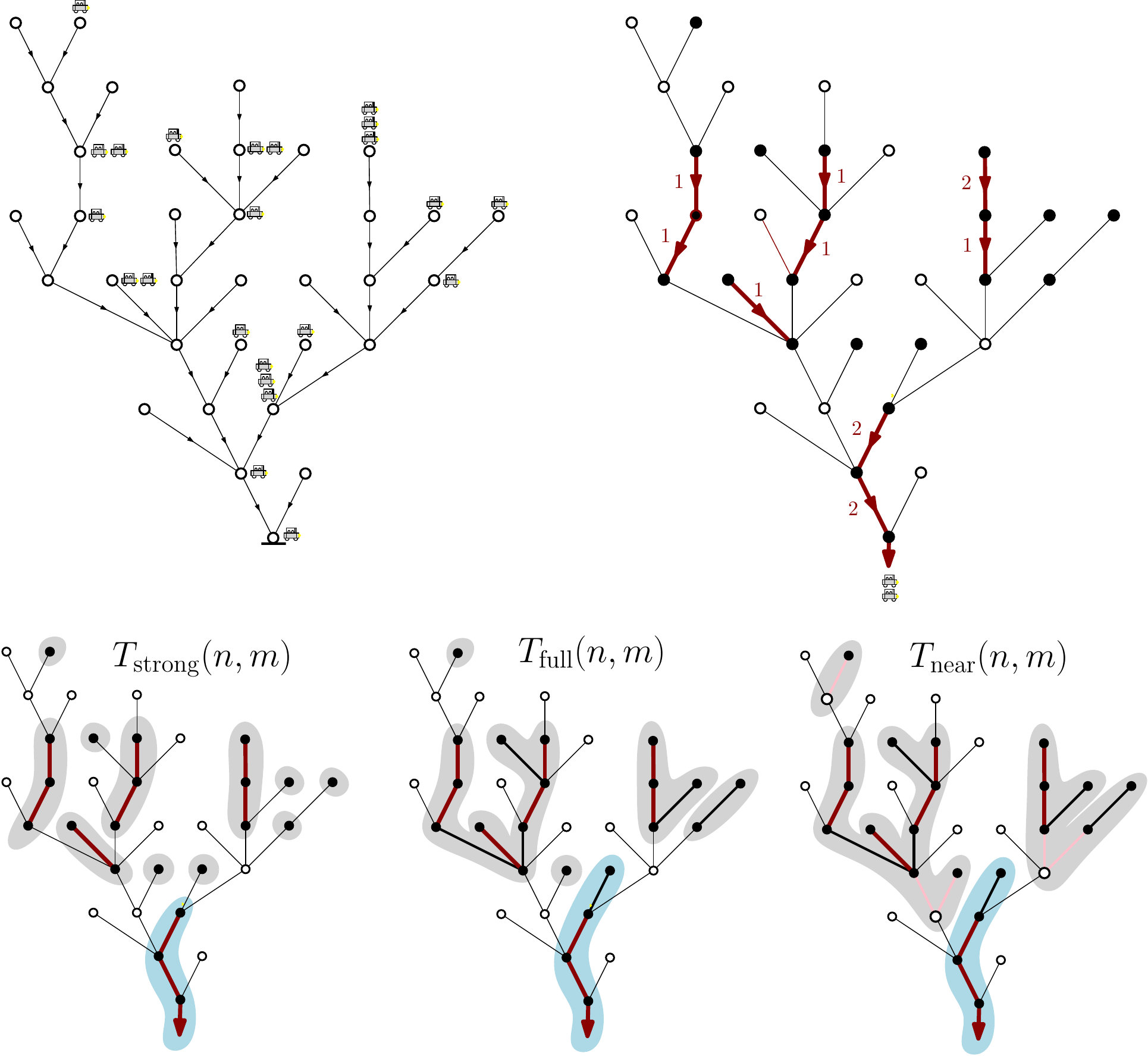}
 \caption{Illustration of the different notions of clusters of parked cars: from left to right the strong, full and near components. The labels of the vertices and of the cars are not displayed for better readability. After $m=23$ car arrivals, the black vertices contain a car and the red edges have seen at least one car going through them. The grey components are, from left to right, strongly/fully/nearly parked trees. The components of the root vertex is not of the same type and is thus colored in blue.
\label{fig:components} }
 \end{center}
 \end{figure}
 
Of course we have $ T_ \mathrm{strong}(n,m) \subset T_ \mathrm{full}(n,m) \subset T_ \mathrm{near}(n,m)$ in terms of edge sets. The component of the root vertex in those forests may not be a strong/fully/nearly parked tree since a positive flux of car may exit through the root (or in the case of near components, the root vertex may contain a car). When the component of the root \emph{is neither} a strongly/fully/nearly parked tree nor an empty vertex, we shall \emph{color it in blue}. On a high level, the starting observation of this paper is that for $n$ fixed the processes $$m\mapsto T_{ \star}(n,m) \quad \mbox{ for } \star \in \{ \mathrm{strong},\mathrm{full}, \mathrm{near}\} \quad \mbox{ are Markov processes},$$ see Proposition \ref{prop:CouplingCayleyFrozen} and Section \ref{sec:lazy}. Although the notions of strong  or full components seem more natural than the notion of near components, we shall see in the next sections that the evolution of $ m \mapsto T_ \mathrm{near}(n,m)$ is very close to the evolution of the Erd{\H{o}}s--R\'enyi random graph, which constitutes the basis of our work. One key feature is that $T_ \mathrm{near}(n,m+1)$ is obtained from $ T_ \mathrm{near}(n,m)$ by adding at most one edge (which is not the case for the two other notions of components).

\begin{remark}[Versions of parked trees] Fully parked trees have been considered by Lackner \& Panholzer in \cite{LaP16} and strongly parked trees by King  \& Yan in \cite{king2019prime}. Both works provide enumeration formulas which we shall recover in Section~\ref{sec:enumconsq}. Different versions of fully parked trees have recently been  investigated by Chen \cite{chen2021enumeration} and Panholzer \cite{panholzer2020parking}, see Section \ref{sec:linkchen} for more details. 
\end{remark}

\subsection{Frozen and Erd{\H{o}}s-R\'enyi random graphs}
\label{sec:deffrozenER}
Recall from the introduction the definition of the random graph process $(G(n,m) : m \geq 0)$ obtained by adding sequentially i.i.d.\ uniform unoriented edges $E_i=\{X_{i}, Y_{i}\}$ where the oriented edges $ \vec{E}_i=(X_{i}, Y_{i})$  have i.i.d.\ uniform endpoints over $\{1,2,\dots , n\}$.  
\begin{figure}[!h]
 \begin{center}
 \includegraphics[width=14cm]{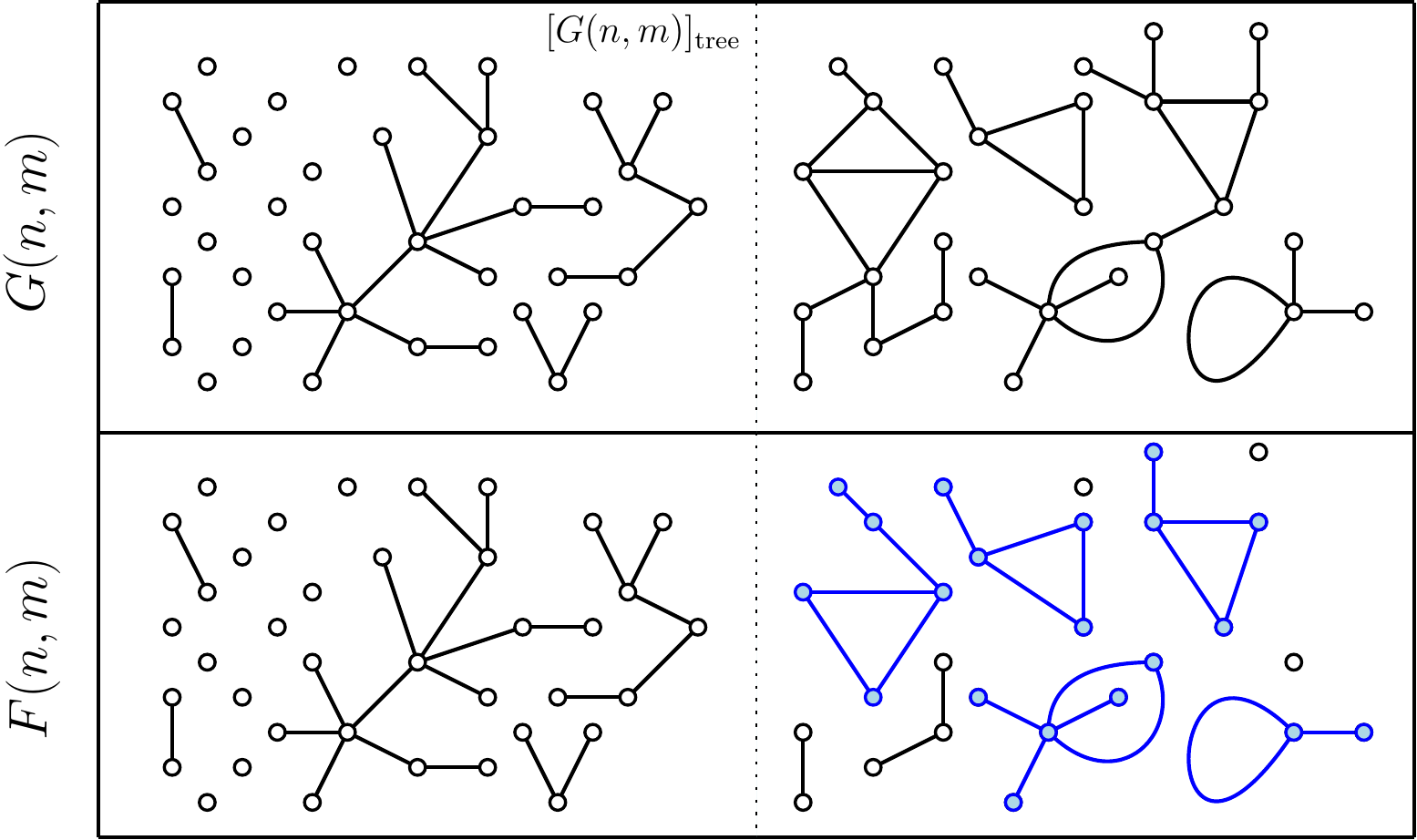}
 \caption{\label{fig:inclusion} Illustration of the inclusion of $F(n,m)$ inside $G(n,m)$. The two processes coincide on connected components of $G(n,m)$ that do not contain surplus (left column) and $F(n,m)$ is obtained by further splitting the remaining components in $ G(n,m)$.}
 \end{center}
 \end{figure}
 
 The frozen process $(F(n,m) : m \geq 0)$ is constructed by discarding certain of those edges and coloring the vertices in blue or white (see Figure \ref{fig:transitions}). In particular $\|{F}(n,m)\|_{\edge} \leq m$ and the inequality may be strict. The vertices in the frozen components of $F(n,m)$ will be colored in blue while the others stay white. 
In that construction, the process $F(n, \cdot)$ lives inside $G(n, \cdot)$ and in particular for every $n,m \geq 0$ we have   \begin{eqnarray} \label{eq:inclusionbidon}F(n,m) \subset G(n,m)  \end{eqnarray} in terms of edge set. Moreover, it is easy to see by induction on $m \geq 0$ that $F(n,m)$ and $G(n,m)$ coincide on the forest part $[G(n,m)]_{ \mathrm{tree}}$, see Figure \ref{fig:inclusion}.

\subsection{Parking on random mapping and the frozen Erd{\H{o}}s--R\'enyi}

\label{sec:couplingmapping}
A \emph{mapping} is a graph over the $n$ labeled vertices $\{1,2, \dots , n\}$ with oriented edges and so that each vertex has exactly one edge pointing away from it, see Figure \ref{fig:exmapping}. Equivalently, the oriented edges of the graph can be seen as $i \to  \sigma(i)$ where $\sigma$ is a map $\{1,2, \dots , n\} \to \{1, 2, \dots , n\}$, hence the name ``mapping''. In particular, if $M_n$ is a uniform random mapping on $\{1,2, \dots , n\}$ then the \emph{targets} $\sigma(i)$ i.e.\ the vertices to which point the edges emanating from $1,2, \dots , n$ are just i.i.d.\ uniform on $\{1,2,\dots , n\}$. 

\begin{figure}[!h]
 \begin{center}
 \includegraphics[width=12cm]{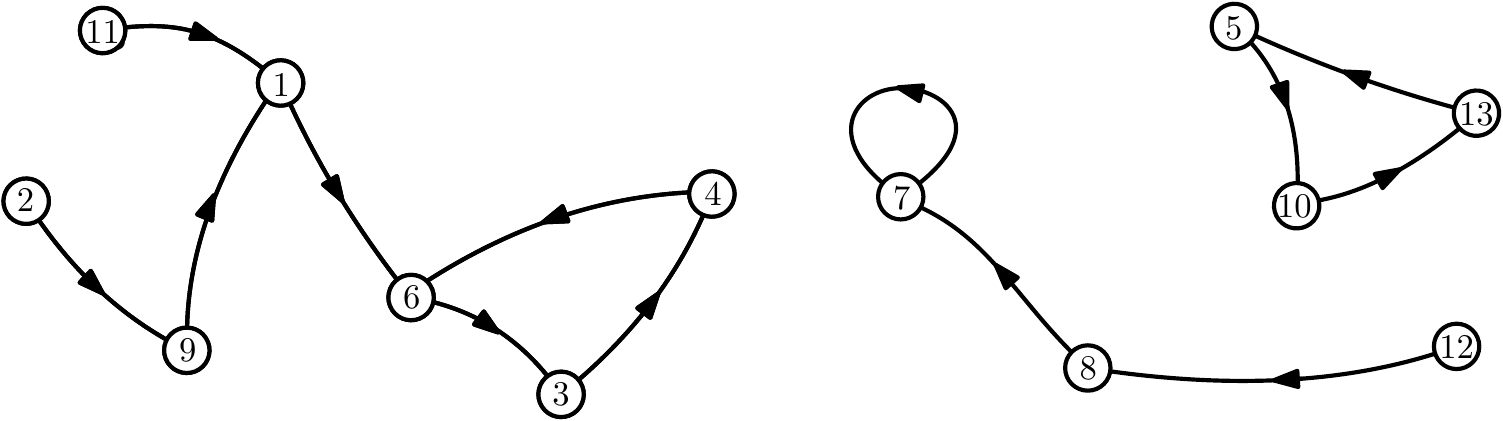}
 \caption{ \label{fig:exmapping}An example of a mapping over $\{1,2, \dots , 13\}$.}
 \end{center}
 \end{figure}

The parking process can be extended from a rooted tree to a mapping (see \cite{LaP16}): Given the random mapping  $M_{n}$, we consider independent uniform car arrivals $X_{i} \in \{1,2, \dots, n\}$. Each car tries to park on its arrival vertex and stops there if the parking spot is empty. Otherwise the car follows the oriented edges of $ M_{n}$ and takes the first available space, if there is one. If the car is caught in an endless loop, then it exits without parking. \\

As for the parking on $T_{n}$, when $m$ cars have arrived we can define submappings
$$ M_ { \mathrm{strong}}(n,m) \subset M_ \mathrm{full}(n,m) \subset M_{ \mathrm{near}}(n,m),$$ by keeping respectively the oriented edges with positive flux of cars, the oriented edges linking two occupied spots, or the oriented edges emanating from occupied spots in the parking process. In the remainder of this section we shall focus on $M_ \mathrm{near}(n,m)$. When an (oriented) cycle is discovered in $M_ \mathrm{near}(n,m)$ we shall color the entire non-oriented component in blue.  

In the above construction, the car arrivals $X_{i}$'s are independent of the uniform random mapping $M_{n}$. The main observation of this section is that one can in fact couple the oriented edges $ \vec{E}_{i} = (X_{i},Y_{i})$ from which we constructed the process $F(n,\cdot)$ with $M_{n}$ so that the above property holds true and furthermore that $F(n,m)$  has the same components as $M_{ \mathrm{near}}(n,m)$. More precisely:

\begin{proposition}[Coupling of parking on mapping with the frozen Erd{\H{o}}s--R\'enyi] \label{coupling:mapping} We can couple the uniform random mapping $M_n$ with the $Y_{i}'s$ in such a way that 
\item \textbf{(Parking on mapping)} The graph $M_{n}$ is a uniform random mapping on $\{1,2, \dots, n\}$ independent of the car arrivals $(X_{i} : i \geq 1)$,
\item \textbf{(Coupling with $F(n, \cdot)$)} For each $m \geq 0$, the subgraph $M_ \mathrm{near}(n,m)$ has the same (unoriented) connected components as  $F(n,m)$. More precisely:
\begin{itemize} \item The blue components of $F(n, m)$ correspond to components with surplus in $ M_ \mathrm{near}(n,m)$,
\item The indices of the discarded edges in $F(n,\cdot)$ correspond to the indices of the cars that do not manage to park on $M_{n}$.
\end{itemize}
\end{proposition}

\begin{proof}[Proof of Proposition \ref{coupling:mapping}] We will construct the mapping $M_{n}$ by prescribing the targets $\sigma(i)$ of its vertices using the oriented edges $ \vec{E}_{m}$'s according to the following rule:

\begin{center}
	\fbox{\begin{minipage}{15cm}
			\paragraph{From oriented edges to parking on mapping.} The starting points $(X_{i} : i \geq 1)$ of the edges $ \vec{E}_{i}$ are the i.i.d.\ arrivals of the  cars over $\{1,2, \dots ,n\}$. We use them to construct iteratively an increasing sequence of oriented graphs $(M(n,m) : m \geq 0)$ where $ M(n,0)$ is the graph over $\{1,2, \dots , n \}$ with no edge. For $m \geq 1$, we use the edges of $M(n,m-1)$ to (try to) park the $m$th car arrived on $X_{m}$. If we manage to park it, we denote by $\zeta_{m} \in \{1,2, \dots , n \}$ its parking spot, otherwise we set $\zeta_{m} = \dagger$. When $\zeta_{m} \ne \dagger$, we add the edge $\zeta_{m} \to Y_{m}$ to $M(n,m-1)$ to form $M(n,m)$, equivalently we put 
			  \begin{eqnarray} \sigma( \zeta_{m}) = Y_{m} \quad \mbox{ when } \zeta_{m} \ne \dagger.    \label{eq:rulemapping}\end{eqnarray}
			\begin{center}
			\includegraphics[width=10cm]{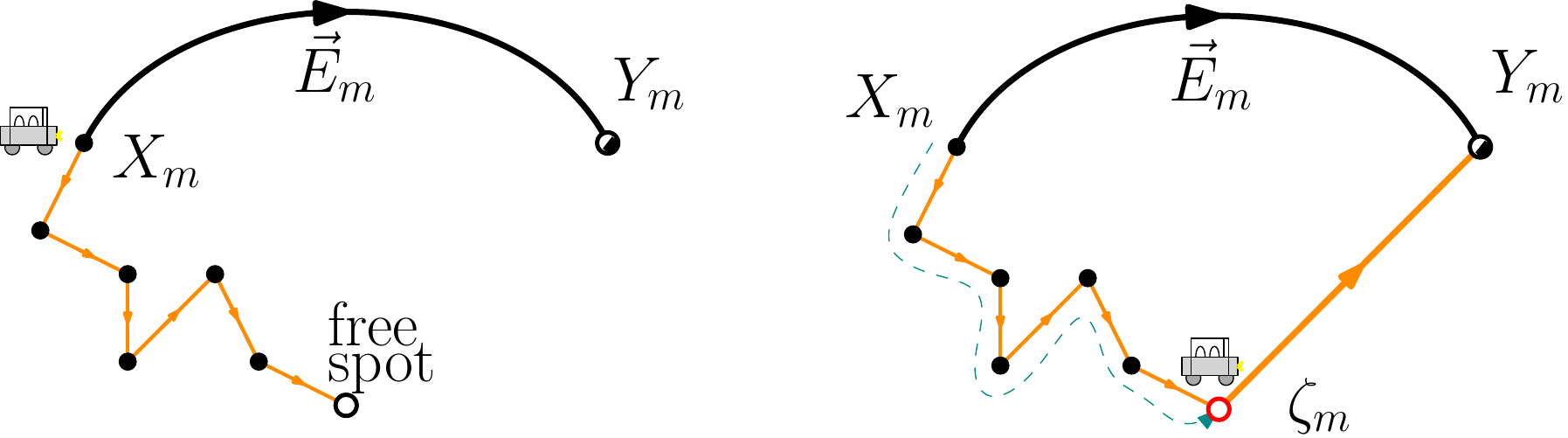}
			\end{center}

		\end{minipage}}
	\end{center}
It is perhaps not clear for the reader how the above rule serves as recipe to construct a mapping, so let us make a couple of remarks and refer to Figure \ref{fig:couplingmapping} for a step-by-step illustration. It is easy to see by induction that every vertex in  $M(n,m)$ has at most one  edge pointing away from it and that the vertices having an emanating edge are those which already accommodate a car.  If the $m$th car is trapped in an endless loop the graph does not evolve and we have $M(n,m) = M(n,m-1)$.  When $\{X_i : 1 \leq i \leq m\}$ has spanned $\{1,2, \dots , n \}$ (this is in particular the case when the $\vec{E}_i$'s are i.i.d.\ uniform oriented edges) the graph $M(n,m)$ is constant  and we \emph{define}
$$ M_{n} := \bigcup_{m \geq 1} M(n,m),$$ which is a random mapping of size $n$. With this construction, it is clear that we have $$M_{ \mathrm{near}}(n,m) = M(n,m), \quad \mbox{ for every } m \geq 0.$$
The second point of the proposition is then easy to check  by induction (see Figure \ref{fig:couplingmapping}) in particular the edges emanating from a blue component in $F(n,\cdot)$ correspond, via the coupling, to cars arriving on a component already containing an oriented loop: such cars will be trapped in an endless loop and contribute to the outgoing flux in the parking, whereas the corresponding edges are discarded in the frozen process.  The non-trivial probabilistic point consists in showing item 1, i.e.\ that this coupling reproduces the parking on a uniform random mapping or in other words, that $M_{n}$, made of the edges $$ \zeta_{i} \to Y_{i}, \quad \mbox{ when } \zeta_{i} \ne \dagger$$ forms a uniform mapping, independent of the $X_{i}$'s (but not of the $Y_{i}$'s !!!). Fix $m \geq 1$ and notice that $\zeta_{m}$ is determined by $ (\vec{E}_{i} : 1 \leq i \leq m-1)$ and $X_{m}$, and in particular is independent of $Y_{m}$. We deduce that conditionally on $ \zeta_{m} \ne \dagger$, its target $\sigma( \zeta_{m}) = Y_{m}$ is independent of the edges already constructed in $M(n,m-1)$, also independent of $(X_i : i \geq 1)$, and is uniform over $\{1,2, \dots , n \}$. Since $\{\zeta_{i} : i \geq 1\}$ spans $ \{1,2, \dots , n \}$ almost surely, conditionally on the $X_{i}$'s, the $\zeta_{i}$ (different from $\dagger$) can be seen as a way to sample the vertices of $\{1,2, \dots , n \}$ (and they all will be sampled) and at each step they are assigned an independent  random uniform target.  A moment's thought shows that the targets of all vertices are i.i.d.\ uniform over $\{1,2,\dots,n\}$ and independent of $(X_i : i \geq 1)$. \end{proof}

\begin{figure}[!h]
 \begin{center}
 \includegraphics[width=14.6cm]{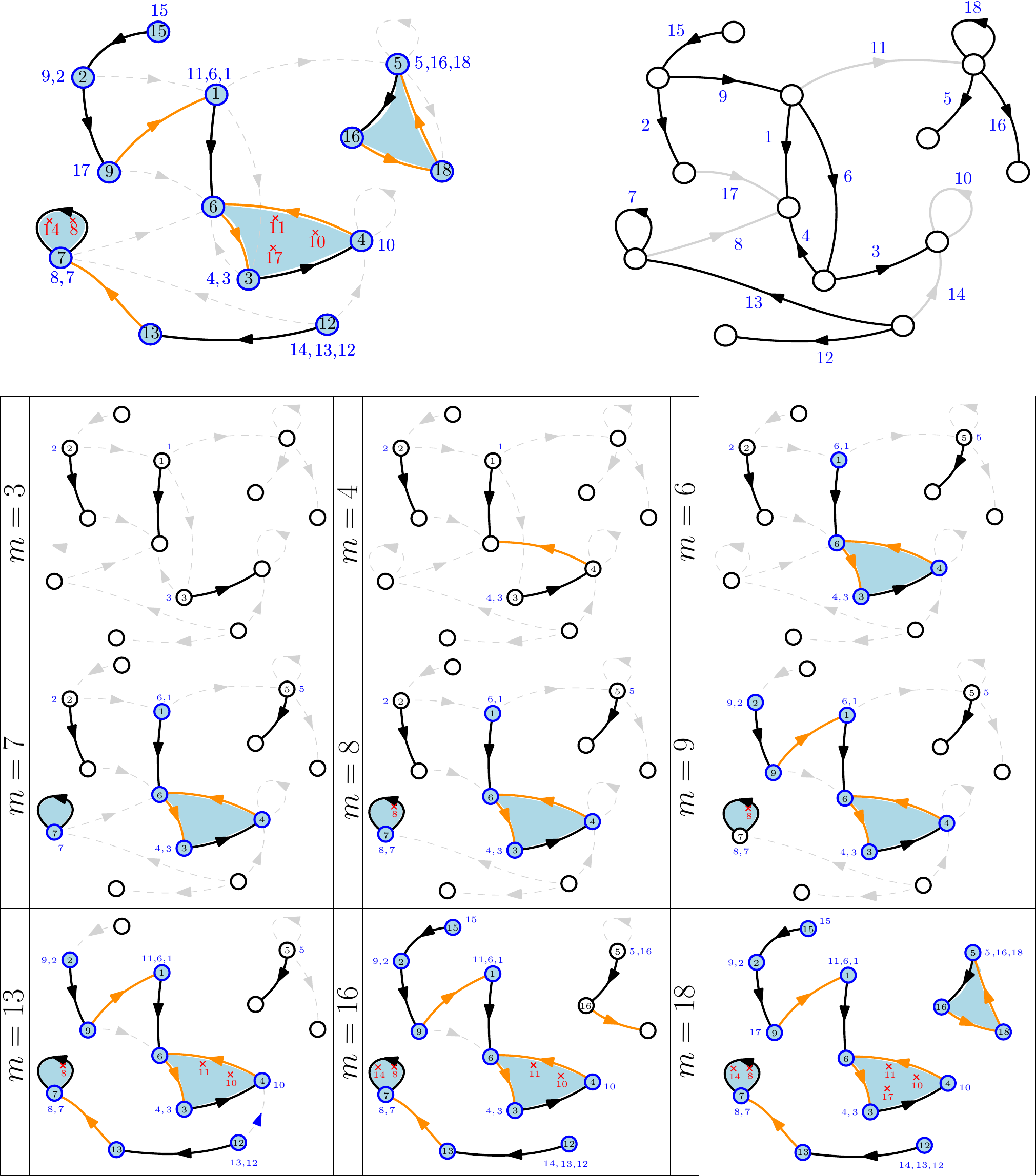}
 \caption{\label{fig:couplingmapping}Illustration of the redirections of the edges $ (\vec{E}_{i} : 0 \leq i \leq m)$ to obtain $M(n,m)$.  On the top right, a possible value for $G(13,18)$ and its corresponding $F(13,18)$ if one only keeps the black edges. On the left, the corresponding $M(13,18)$. The redirected edges which are different from those in $F(n,\cdot)$ are in orange. The labels of the vertices are not displayed for better visibility, but the labels of the cars are present (blue for their arrival vertices, black for the parking spot, and red if the car does not manage to park).  In the tabular, we represented a step-by-step construction by displaying $M(13,m)$ for $m=3,4,6,7,8,9,13,16$ and $18$.
}
 \end{center}
 \end{figure}
 
 \section{Coupling of parking on Cayley trees with the frozen Erd{\H{o}}s--R\'enyi}
  We shall now perform a similar coupling between the frozen Erd{\H{o}}s--R\'enyi process and the parking process on a uniform Cayley tree. Although the main idea (considering $T_{n}$ as unknown and revealing $T_{ \mathrm{near}}(n,m)$ step-by-step in a Markovian way) is the same, the Markovian exploration of Cayley trees is a little more complicated than in the case of random mappings and we shall need some extra randomness to perform the coupling. 

\subsection{Markovian exploration of rooted Cayley trees}
We present the Markovian explorations of uniform Cayley trees which are adapted from  \cite{contat21surprizing}. We call these ``peeling explorations''  by analogy with the peeling process of random planar maps \cite{CurStFlour}. We will later  tailor those explorations to the parking process using a specific peeling algorithm and this will yield the coupling with the frozen Erd{\H{o}}s--R\'enyi, see Proposition \ref{prop:CouplingCayleyFrozen}.
\label{sec:markovpeeling}

Recall that a rooted Cayley tree $ \mathbf{t}$ is an unordered tree over the $n$ labeled  vertices $\{1,2, \dots , n\}$  where one of its vertices has been distinguished and called the root. This root enables us to orient all edges of $ \mathbf{t}$ towards it. As in the case mappings, this allows us to speak of the target $\sigma(i)$ of each vertex $i \in \{1, \dots , n \}$ which is the vertex to which points the edge emanating from $i$. A difference with the previous section is that in the case of trees no loop can be created and  the root vertex $ \mathrm{r}$ of $ \mathbf{t}$ has no target which we write as $\sigma( \mathrm{r}) = \varnothing$. The information on $ \mathbf{t}$ is thus encoded by the $n$ ``instructions'' $$ \{ i \to \sigma(i) \} \quad \mbox{ where }\sigma(i) \in \{1,2, \dots, n \} \cup \{ \varnothing \}\mbox{ for }i \in \{1,2, \dots , n \}.$$ An \emph{exploration} of $ \mathbf{t}$ can be seen as revealing those $n$ instructions one by one by discovering the target of one vertex at a time. A set $S$ of instructions  is said to be \emph{compatible} if it corresponds to a subset of instructions of some tree. Any such set can be interpreted as a forest of rooted trees by connecting the vertices to their revealed targets, see Figure \ref{fig:explo-tree}. If the target of the root is ``revealed'' (one should probably better say that the root vertex is revealed) then we record this information by coloring the corresponding tree  in blue, the other trees being referred to as white.
\begin{figure}[!h]
 \begin{center}
 \includegraphics[width= 13cm]{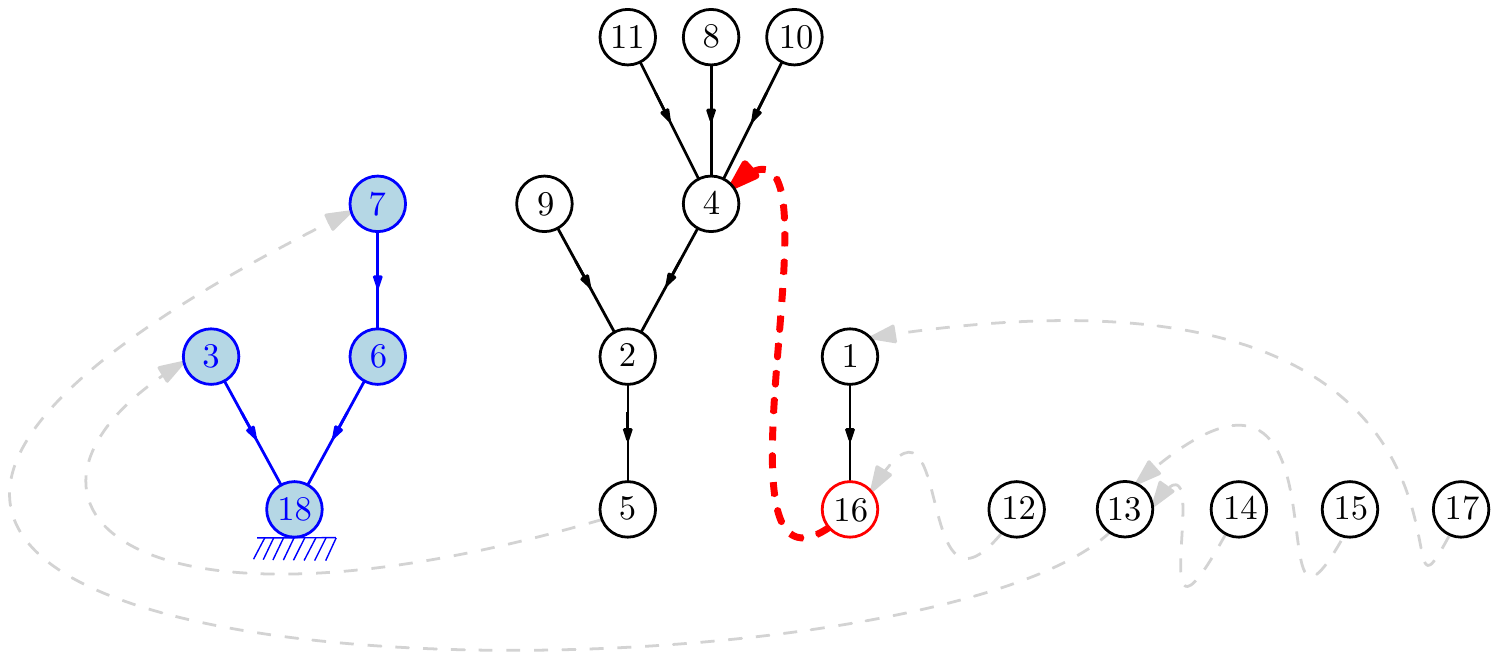}
 \caption{ Illustration of the rooted forest obtained from the explored subset $S=\{ 3 \to 18, 18 \to \varnothing, 7 \to 6, 6 \to 18, 11 \to 4, 8 \to 4, 10 \to 4, 4 \to 2, 9 \to 2, 2 \to5, 1 \to 16  \}$. Notice that the root vertex has been revealed thanks to the presence of $18 \to \varnothing$ and we have colored the corresponding tree in blue. In this example the next vertex to be peeled is $ \mathfrak{a}(S)= 16$ and its target is the vertex $4$. The remaining edges of the underlying tree are displayed in dotted gray.} 
  \label{fig:explo-tree}
 \end{center}
 \end{figure}
 
Of course, a given rooted Cayley tree $ \mathbf{t}$ with $n$ vertices can be explored in $n!$ different ways and we shall choose one using a function $\mathfrak{a}$, called the \emph{peeling algorithm}, which associates any subset $S$ of compatible instructions  (which does not yet form a tree) with a vertex $ \mathfrak{a}(S)$ whose target is not revealed yet (in particular $ \mathfrak{a}$ depends only on $S$ and not on the underlying tree, and $ \mathfrak{a}(S)$ must be a root of a standard white tree of the forest associated with $S$). The peeling of $ \mathbf{t}$ with algorithm $ \mathfrak{a}$ is then the sequence 
$$  \mathbf{S}^{ \mathfrak{a}}_{0} \subset  \mathbf{S}^{ \mathfrak{a}}_{1} \subset \cdots \subset  \mathbf{S}^{ \mathfrak{a}}_{n} = \bigcup_{i=1}^{n} \{i \to \sigma(i)\},$$ where $  \mathbf{S}^{ \mathfrak{a}}_{0}$ is the empty set and $ \mathbf{S}^{ \mathfrak{a}}_{i+1} =  \mathbf{S}^{ \mathfrak{a}}_{i} \cup \{ \mathfrak{a}( \mathbf{S}^{ \mathfrak{a}}_{i}) \to \sigma(  \mathfrak{a}(\mathbf{S}^{ \mathfrak{a}}_{i}))\}$ for all $i \leq n-1$. In other words, one can choose at each step of the exploration which vertex we want to reveal the target of. We shall call those explorations ``peeling explorations'' or ``Markovian explorations'', indeed, when the peeling algorithm $ \mathfrak{a}$ is fixed and when the underlying tree $ \mathbf{t}$ is a uniform rooted Cayley tree, this exploration is a Markov chain with explicit probability transitions:

\begin{proposition}[Markov transitions for peeling exploration of uniform Cayley trees] \label{prop:markovtree} Fix a peeling algorithm $ \mathfrak{a}$. If $T_{n}$ is a uniform rooted Cayley tree with $n$ vertices, then the exploration $(\mathbf{S}_{i}^{\mathfrak{a}})_{0 \leq i \leq n}$ of $T_{n}$ with algorithm $ \mathfrak{a}$ is a Markov chain  whose probability transitions are described as follows. Conditionally on $\mathbf{S}_{i}^{\mathfrak{a}}$ and on $\mathfrak{a} (\mathbf{S}_{i}^{\mathfrak{a}})$, in the forest representation of $\mathbf{S}_{i}^{\mathfrak{a}}$ we denote by $k \geq 1$ the number of vertices of the tree of root $\mathfrak{a} (\mathbf{S}_{i}^{\mathfrak{a}})$ and by  $\ell \geq 0$ the number of vertices of the blue tree (if any) then:
 \begin{itemize}
 \item  If $\ell = 0$, with probability $ \frac{k}{n}$ we have $ \sigma( \mathfrak{a} (\mathbf{S}_{i}^{\mathfrak{a}}))= \varnothing$ (i.e.\ the vertex we peel is the root of the underlying Cayley tree), otherwise $\sigma( \mathfrak{a} (\mathbf{S}_{i}^{\mathfrak{a}}))$ is a uniform vertex not belonging to the tree of root  $\mathfrak{a} (\mathbf{S}_{i}^{\mathfrak{a}})$. 
 \item If $\ell \geq 1$, with probability $ \frac{\ell+k}{n}$ the target $\sigma( \mathfrak{a} (\mathbf{S}_{i}^{\mathfrak{a}}))$ is a uniform vertex of the blue tree of $\mathbf{S}^{ \mathfrak{a}}_{i}$, otherwise it is a uniform vertex of the remaining trees except the tree of root $  \mathfrak{a}(\mathbf{S}^{ \mathfrak{a}}_{i})$. 
\end{itemize}
\end{proposition}

The proof is similar to that of \cite[Proposition 1]{contat21surprizing} and relies on counting formulas established in \cite[Lemma 5]{contat21surprizing} based on Pitman's approach \cite[Lemma 1]{pitman1999coalescent}. Specifically we have:
\begin{lemma} \label{lem:nbCayley}  If $ \mathbf{f}$ is a forest of white rooted trees on $\{ 1, \dots, n\}$ with $m$ edges, then the number of rooted Cayley trees containing $ \mathbf{f}$ is $n^{n-m-1}$. If $ \mathbf{f}^{\ast}$ 
is a forest of rooted trees on $\{ 1, \dots, n\}$ with $m$ edges containing a blue tree with $\ell \geq 1$ vertices, 
then the number of rooted Cayley trees containing $ \mathbf{f}^{\ast}$ with the root being the root of the blue tree is $\ell n^{n-m-2}$.
\end{lemma}
To be precise, \cite{contat21surprizing} considers Cayley trees as rooted at the vertex $n$ whereas we allow the root vertex to be any vertex  of $ \{1,2, \dots, n \}$ hence the factor $n$ difference between the numbers appearing in the above lemma and those of \cite[Lemma 5]{contat21surprizing}.

\proof[Proof of Proposition \ref{prop:markovtree}]  Given the last lemma, the proof is easy to complete. It suffices to notice that since the underlying tree $ T_{n}$ is uniform over all rooted Cayley trees with $n$ vertices, for all $i \geq 0$, conditionally on $ \mathbf{S}_i^{\mathfrak{a}}$, the tree $  T_n$ is a uniform tree among those which contain the forest associated to  $\mathbf{S}_i^{\mathfrak{a}}$ (with or without a blue tree depending whether the root vertex has been revealed or not). Hence, for every (compatible) target $v \in \{1,2, \dots , n\} \cup \{ \varnothing \}$, 
\begin{equation*}
\mathbb{P} \Big(\sigma( \mathfrak{a}( \mathbf{S}_{i}^{ \mathfrak{a}})) = v  \big| \textbf{S}_{i}^{\mathfrak{a}}, \mathfrak{a} ( \textbf{S}_{i}^{\mathfrak{a}})\Big) = \frac{\#\{ \mathbf{t}  \text{ containing the forest associated with }  \textbf{S}_{i}^{\mathfrak{a}} \cup \{ \mathfrak{a} ( \textbf{S}_{i}^{\mathfrak{a}}) \to v\} \}}{\#\{ \mathbf{t}  \text{ containing the forest associated with }  \textbf{S}_{i}^{\mathfrak{a}}\}}.
\end{equation*}
Using Lemma \ref{lem:nbCayley}, we recognize the transition probabilities given in Proposition \ref{prop:markovtree} and obtain the desired result.
\endproof

The interest of Proposition \ref{prop:markovtree} is that different peeling algorithms can be used to explore a uniform Cayley tree and this may yield to different type of information. See  \cite[Sections 3 and 4]{contat21surprizing} for applications to the greedy independent set, the Aldous--Broder or Pitman algorithms.

\subsection{The near exploration}
\label{sec:couplingtrees}

Recall the notion of near components defined in Section \ref{sec:components}. We shall see that $ T_{ \mathrm{near}}(n,\cdot)$ can be interpreted as a peeling process of $T_{n}$ using an algorithm  (called $ \mathfrak{a}_{ \mathrm{near}}$ below) tailored to the parking process. Furthermore, as in the last section, we shall make a coupling of $T_{n}$ with the oriented edges $\vec{E}_{i}'s$ so that $T_{n}$ stays independent  of the car arrivals $X_{i}$'s but in such a way that $ T_{ \mathrm{near}}(n, \cdot)$ is closely related to the frozen process $F(n,\cdot)$. 
\begin{proposition}[The main coupling]\label{prop:CouplingCayleyFrozen} We can couple  $T_{n}$ with the $Y_{i}$'s so that: \begin{itemize}
\item \textbf{(Parking on Cayley tree)} The tree $T_{n}$ is a uniform rooted Cayley tree independent of the car arrivals $(X_{i} : i \geq 1)$.
\item \textbf{(Coupling with $F(n, \cdot)$)} For each $m \geq 0$, the subforest $T_{  \mathrm{near}}(n,m)$ has the same (unoriented) connected components as $F(n,m)$ where all the frozen components have been joined. More precisely:
\begin{itemize} \item The white components of $ F(n,m)$ are the connected components of $T_{ \mathrm{near}}(n,m)$ which do not contain the root,
\item The vertices of the blue components of $F(n, m)$ correspond to the vertices of the (unique) blue component of $ T_{ \mathrm{near}}(n,m)$,
\item The indices of the discarded edges in $F(n,\cdot)$ correspond to the indices of the cars that do not manage to park on $T_{n}$.
\end{itemize}
\end{itemize}
\end{proposition}
	
\begin{proof}[Proof of Proposition \ref{prop:CouplingCayleyFrozen}] As in the proof of Proposition \ref{coupling:mapping}, we shall construct $T_{n}$ using the $\vec{E}_{i}$'s.  The main difference being that  the apparition of the first cycle in $G(n,m)$ corresponds to the detection of the root vertex in the Cayley tree and that we need an additional randomization to redirect some of the edges $ \vec{E}_{i}$ (whereas in the case of mapping, the redirection was a measurable function of the $X_{i}$ and $Y_{i}$).

\begin{center}
	\fbox{\begin{minipage}{15cm}
			\paragraph{From oriented edges to parking on trees.} 
			The starting points $(X_{i} : i \geq 1)$ of the edges $ \vec{E}_i$ are the i.i.d.\ arrivals of the  cars over $\{1,2, \dots ,n\}$.  We use them to construct iteratively an increasing sequence of compatible instructions $ (\mathbf{S}_{m} ^{ \mathrm{park}} : m \geq 0)$ or equivalently of growing forests $(T(n,m) : m \geq 0)$ with possibly one blue tree. Initially $ \mathbf{S}_{0}^{\mathrm{park}}$ is the empty set and for $m \geq 1$, we use the edges of $\mathbf{S}_{m-1} ^{\mathrm{park}}$ to (try to) park the $m$th car arrived on $X_{m}$. If we manage to park it, we denote by $\zeta_{m} \in \{1,2, \dots , n \}$ its parking spot, otherwise set $\zeta_{m} = \dagger$. If $\zeta_{m}= \dagger$ then $ \mathbf{S}_{m}^{ \mathrm{park}} = \mathbf{S}_{m-1}^{ \mathrm{park}}$. Otherwise
			\begin{itemize}
			\item if the addition of the edge $\zeta_{m} \to Y_{m}$ does not create a cycle in $T(n , m-1)$, then add it to $ \mathbf{S}_{m-1}^{ \mathrm{park}}$ to form $ \mathbf{S}_{ m }^{ \mathrm{park}}$,
			\item if the addition of the edge $\zeta_{m} \to Y_{m}$ creates a cycle in $T(n,m-1)$ then 
			\begin{itemize}
			\item If $T(n,m-1)$ has no blue tree (the root vertex is not revealed), then add $\zeta_{m} \to \varnothing$ to form $  \mathbf{S}_{m} ^{\mathrm{park}}$,
			\item Otherwise add  $\zeta_{m} \to U_{m}$ where $U_{m}$ is a uniform point over the blue tree of $ T( n , m-1)$ sampled independently of the past to form $  \mathbf{S}_{m} ^{\mathrm{park}}$, see Figure \ref{fig:redirection}.
\end{itemize}
			\end{itemize}
		\end{minipage}}
	\end{center}
	
	\begin{figure}[!h]
 \begin{center}
 \includegraphics[width=13cm]{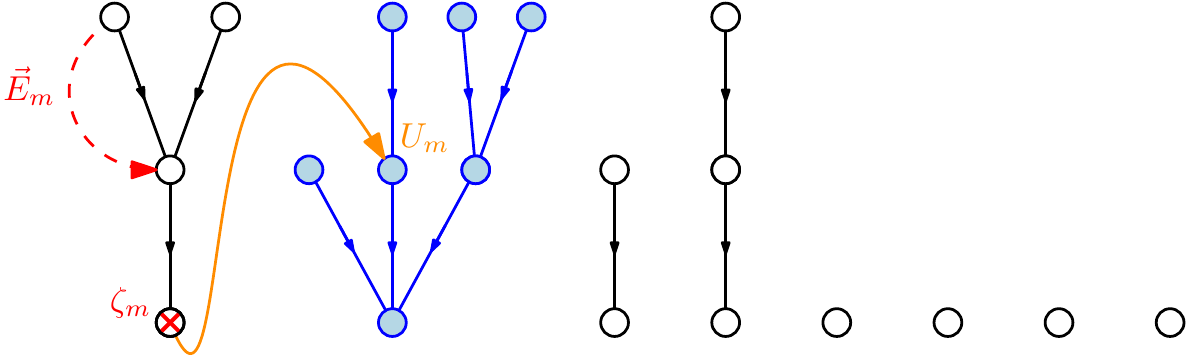}
 \caption{\label{fig:redirection}When a new cycle is created by the addition of the edge $ \vec{E}_{m}$, the target of $\zeta_{m}$ is chosen uniformly in the blue tree.}
 \end{center}
 \end{figure}

As in the proof of Proposition \ref{coupling:mapping}, the increasing forests $T(n,m)$ eventually stabilize to form a (blue) tree and we put 
	$$T_{n} := \bigcup_{m \geq 0} T(n,m).$$
	With this definition, it is clear that we have  $$T_{ \mathrm{near}}(n,m) = T(n,m) \quad \mbox{for all } m \geq 0,$$ and the deterministic properties of the coupling between the parking on $T_{n}$ and $F(n, \cdot)$ are easy to prove by induction. It thus remains to prove that $T_{n}$ is indeed a uniform rooted Cayley tree independent of $ (X_{i} : i \geq 1)$. To see this, we shall interpret the Markov chain $( \mathbf{S}_{m} ^{ \mathrm{park}} : m \geq 0)$ as a peeling exploration of a uniform Cayley tree. Specifically, given $(X_{i} : i \geq 1)$ we construct a peeling algorithm $ \mathfrak{a}_{ \mathrm{near}}$ as follows. At $m=0$, we start from the empty set $ \mathbf{S}_{0}^{\mathfrak{a}_{ \mathrm{near}}}$ and for $m \geq 1$, if  $ \mathbf{S}_{m-1}^{\mathfrak{a}_{ \mathrm{near}}}$ is the current status of the exploration, we let a car arrive on vertex $X_{m}$. The car follows the oriented edges already present in $\mathbf{S}_{m-1}^{ \mathfrak{a}_{ \mathrm{near}}}$ to find its parking spot $\zeta_{m}$. As in the case of random mapping, if the car does not park (i.e.\ exits through the root of the tree) then we put $\zeta_{m} = \dagger$ and do not trigger a peeling step, i.e.\ move to step $m+1$. In the case $\zeta_{m} \ne \dagger$ we put 
 \begin{eqnarray} \label{eq:ruletree} \mathfrak{a}_{ \mathrm{near}}( \mathbf{S}_{m-1}^{ \mathfrak{a}_{ \mathrm{near}}}) = \zeta_{m},  \end{eqnarray} that is we reveal the target $\sigma( \zeta_{m}) \in \{1,2, \dots ,n \} \cup \{ \varnothing \}$ and include $ \zeta_{m} \to \sigma( \zeta_{m})$ to form $ \mathbf{S}_{m}^{ \mathfrak{a}_{ \mathrm{near}}}$. 
The process $ (\mathbf{S}_m^{  \mathrm{park}} : m \geq 0)$ has the same law as the peeling exploration $ (\mathbf{S}_m^{ \mathfrak{a}_{ \mathrm{near}}} : m \geq 0)$ with the random algorithm $ \mathfrak{a}_{ \mathrm{near}}$: indeed the probability transitions of $ \mathbf{S}^{ \mathrm{park}}$ described above are the same as those of Proposition \ref{prop:markovtree}. Conditionally on the $X_{i}$'s, the function $  \mathfrak{a}_{ \mathrm{near}}$ can be seen as a deterministic peeling algorithm, so by Proposition \ref{prop:markovtree}, the tree $T_{n}$ constructed this way is indeed uniform. In particular, the tree $T_n$ is independent of the $(X_i : i \geq 1)$ which are themselves i.i.d.~uniform on $\{1,2, \dots , n\}$. 
 Our claim follows
\end{proof}

\begin{center} \hrulefill \textit{Convention} \hrulefill  \end{center}
In the rest of the paper we shall always suppose that the tree $T_{n}$ the car arrivals $(X_{i} : i \geq 0)$ and the frozen Erd{\H{o}}s--R\'enyi process $F(n,\cdot)$ are built from the sequence $ (\vec{E}_{i} = (X_{i},Y_{i}): i \geq 1)$ as in the proof of Proposition \ref{prop:CouplingCayleyFrozen}. \begin{center} \hrulefill  \end{center}

%\begin{remark}[Other couplings between mappings and Cayley trees.] By combining the constructions in the previous two sections, we get a coupling between a uniform mapping $M_{n}$ and  a uniform rooted Cayley tree $T_{n}$ which is different from the one \cite{aldous2004brownian} based on Joyal's bijection \cite{Joyal}.%Also, Lackner \& Panholzer \cite{LP16} already realized that the parking process on uniform mapping is intimately connected to parking process on uniform Cayley tree in terms of enumeration. Our constructions give a strong sense to this connection.%that on a Cayley tree On a high level, for any peeling algorithm that can be run on mappings and on trees, coupling the transitions probabilities of Proposition \ref{prop:markovtree} with that of a uniform mapping yields a coupling of a uniform mapping and a uniform rooted Cayley tree. One can in particular think of the coupling obtained using (variants of) Aldous--Broder or Pitman's algorithms. Passing those couplings to the scaling limits may yield to new identity in distribution.
%\end{remark}

\subsection{The strong exploration} \label{sec:lazy} We saw above in the proof of Proposition \ref{prop:CouplingCayleyFrozen} that the process $m \mapsto T_{ \mathrm{near}}(n,m)$ can be seen as a peeling exploration of the underlying tree $T_{n}$ with the algorithm $ \mathfrak{a}_{ \mathrm{near}}$ that reveals the targets of the parked vertices. In a similar vein, one can interpret $ m \mapsto T_{ \mathrm{strong}}(n,m)$ as a peeling exploration where we reveal the target of a vertex when a car is emanating from it. More precisely, we let the cars arrive one by one on the vertices $X_{i}$ and peel the vertices when the cars need to \emph{move} and find their potential parking spot (as opposed to the former near algorithm where we peeled the vertex on which the $i$th car parked). In particular, the arrival of a car may result in no peeling step (e.g.\ if the car parks on its arrival vertex) or to several peeling steps, see Figure \ref{fig:lazyparking}. 
We do not formalize further and hope it is clear for the reader.  After $m$ cars have arrived, this exploration has revealed the strong components 
$$ T_{ \mathrm{strong}}(n,m)$$ which we defined in Section \ref{sec:components}. Recall also that if the outgoing flux of cars is positive then the tree carrying the root vertex in $ {T}_{ \mathrm{strong}}(n,m)$ is seen as a blue tree (and indeed we discovered the root vertex during the peeling exploration).

\begin{figure}[!h]
 \begin{center}
 \includegraphics[width=13cm]{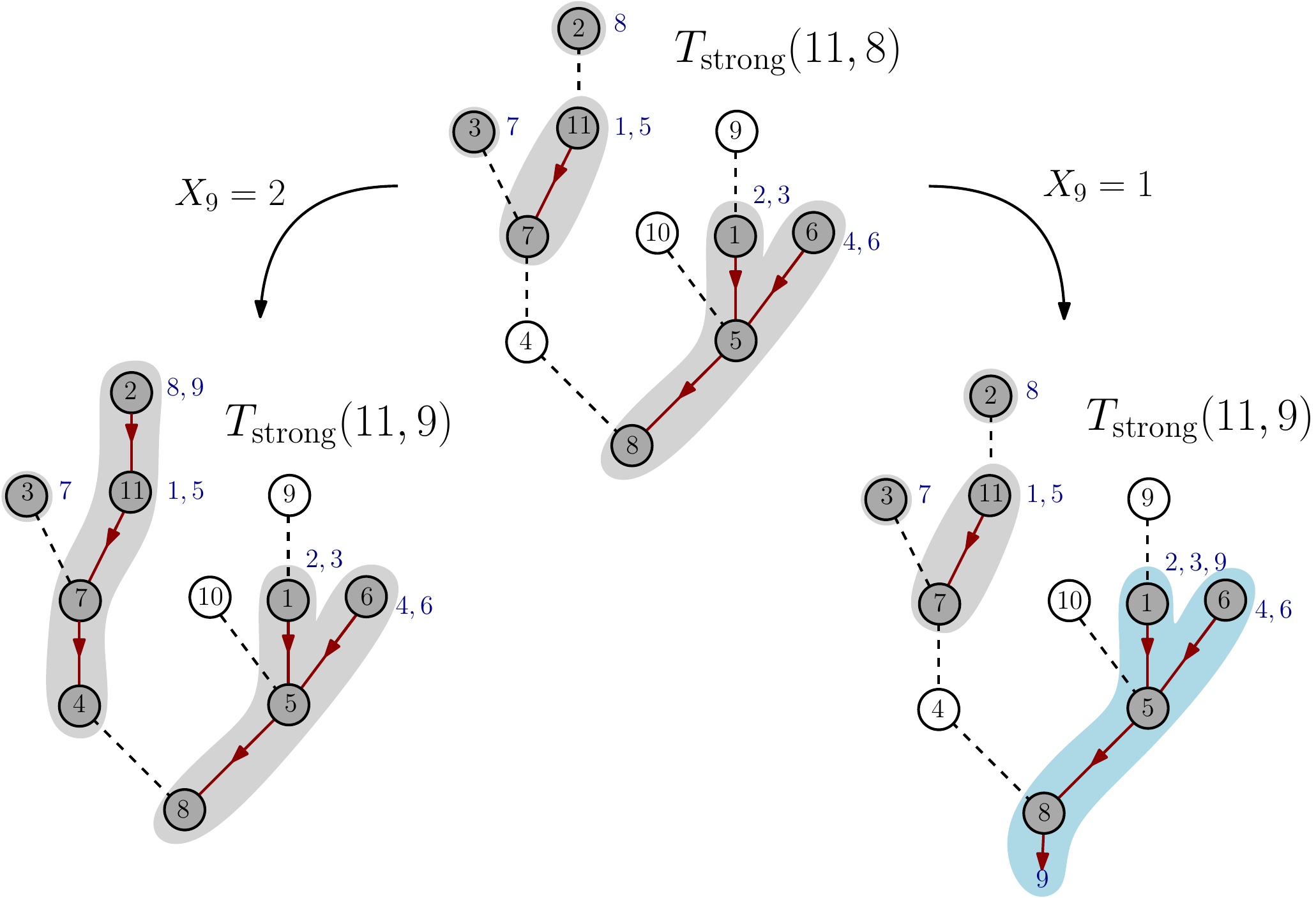}
 \caption{ \label{fig:lazyparking}Illustration of the strong parking peeling algorithm. On the {top}, the current status $\{11 \to 7, 1\to 5, 6 \to 5, 5 \to 8\}$ after $8$ cars have arrived triggering in total {$4$} peeling steps. The available spots are in white whereas the {gray} vertices already contain a car, the red edges have positive flux. If the next car arrives on vertex $2$, it triggers two peeling steps resulting in $2 \to 11$ and $7 \to 4$ before parking on vertex $4$. If the next car arrives on vertex $1$, it follows the edges, triggers the step $8 \to \varnothing$ and cannot park. The root components then becomes blue because we discovered the root of the underlying tree.}
 \end{center}
 \end{figure}

This peeling exploration enables to see $ T_{ \mathrm{strong}}(n,m)$ (together with its coloration) as a Markov chain. We shall not describe its probability transitions, but we shall use it to  relate the probability that the root of $T_{n}$ contains a car to the probability that the outgoing flux in $T_{n}$ is equal to $0$. Recall from the introduction that $D(n,m)$ is the number of cars that did not manage to park among the first $m$ cars.

\begin{lemma} \label{lem:rootflux} For $n\geq 1$ and $0 \leq m \leq n$ we have
\[ \P( \mbox{the root of } T_n \mbox{ is not occupied by one of the first $m$ cars }| D(n,m) = 0) = 1 - \frac{m}{n}. \]
\end{lemma}

\proof Let us explore the underlying tree $T_{n}$ using the strong parking peeling algorithm until we manage to park $m \leq n$ cars (notice that the number of peeling steps is between $0$ and $m-1$). On the event $\{ D(n,m)=0\}$ all peeling steps performed so far have not revealed the root vertex of $T_{n}$ (we did not need to peel the root vertex since no car was emanating from it) so the corresponding forest $ {T}_{ \mathrm{strong}}(n,m)$ is made of white rooted trees (no blue tree) containing $n-m$ isolated vertices which do not yet accommodate a car. By the proof of Proposition \ref{prop:markovtree} and Lemma \ref{lem:nbCayley}, conditionally on ${T}_{ \mathrm{strong}}(n,m)$, the probability that the root vertex of $T_{n}$ (which is yet undiscovered) is a given root of a tree $ \mathbf{t}$ of ${T}_{ \mathrm{strong}}(n,m)$ is proportional to the number of vertices of $ \mathbf{t}$. Hence, the probability that the root vertex of $T_{n}$ does not contain one of the first $m$ cars is $ \frac{n-m}{n}$ as desired.
\endproof
\begin{remark}We saw above that $m \mapsto T_{ \mathrm{strong}}(n,m)$ and $m \mapsto T_{ \mathrm{near}}(n,m)$ can be seen as peeling explorations of $T_n$. It does not seem to be the case for $ m\mapsto T_{ \mathrm{full}}(n,m)$ although it is a Markov process and might alternatively be used to prove the above proposition.
\end{remark}

\section{Free forest property}
In this section we gather several results about (random uniform) labeled (unrooted unordered) forests over $\{1,2, \dots , n\}$. We first recall their enumeration from classical results of R\'enyi and Britikov.

\subsection{Uniform (unrooted) forest} \label{sec:uniformforest}

\label{sec:countforest}

Let  $ \mathfrak{F}(n,m)$  be the set of all unrooted unordered forests over the $n$ labeled vertices $\{1,2, \dots , n\}$ with $n-m$ components  (hence $m$ edges in total). To enumerate such forests it is better to considered the trees as \emph{indexed} by $\{1,2, \dots , n-m\}$ and consider the set  of all unrooted, unordered forests of $\{1,2, \dots , n\}$ with $n-m$ components indexed by $1,2, \dots , n-m$. The number of such forests with components of sizes $(k_1, \dots, k_{n-m})$  is equal to   \begin{eqnarray} \label{eq:countingforest} \binom{n}{k_1, \dots, k_{n-m}} \prod_{i=1}^{n-m} k_i^{k_i -2}, \end{eqnarray} the binomial coefficient $\binom{n}{k_1, \dots, k_{n-m}}$ counts for the number of choices to partition the $n$ vertices in a list of $n-m$  subsets of $k_{1}, \dots , k_{n-m}$ vertices and on each subset there are $k_{i}^{k_{i}-2}$ ways to choose a spanning tree (Cayley's formula). To manipulate those numbers, let us introduce
   \begin{eqnarray} \label{eq:seriesT}\mathbf{T} (z) = \sum_{n \geq 1} \frac{n^{n-2}}{n!} z^n   \end{eqnarray}
the exponential generating function of (unrooted) Cayley trees. Summing \eqref{eq:countingforest} over all choices of $k_{1}, \dots , k_{n-m}$ and dividing by $(n-m)!$ to remove the indexation of the components we deduce that  \begin{eqnarray} \label{eq:forestexact}  \# \mathfrak{F}(n,m) = \frac{n!}{(n-m)!} [x^n] \mathbf{T}^{n-m} (x),   \end{eqnarray} where we recall the standard notation $[x^{n}] \sum_{i\geq 0} a_{i}x^{i} = a_{n}$. Based on \eqref{eq:forestexact} R\'enyi \cite{renyi1959some} showed 
  \begin{eqnarray} \label{eq:forestcount}\# \mathfrak{F}(n,m) = \frac{1}{(n-m)!} \sum_{i=0}^{n-m} \binom{n-m}{i} \left(-\frac{1}{2}\right)^{i} (n-m+i) n^{m-i-1} \frac{n!}{(m-i)!}.  \end{eqnarray}

Note that the power series $ \mathbf{T}(z)$ is convergent when $|z| \leq  \mathrm{e}^{-1}$, and  for $ z= \mathrm{e}^{{-1}}$ we have $ \mathbf{T}(  \mathrm{e}^{-1}) = \frac{1}{2} $ (it follows from \eqref{eq:lienT} below) so that $ 2\mathbf{T}(z/ \mathrm{e})$ is the generating function of a probability measure 
  \begin{eqnarray} \label{eq:defmu} \mu(k) &:=& 2\cdot \frac{k^{k-2}}{ \mathrm{e}^{k} \cdot k!} \quad \mbox{ for } k \in\{1,2, \dots\} ,  \end{eqnarray} of expectation $ 2z\partial_{z} \mathbf{T}(z)|_{z=  \mathrm{e}^{-1}} = 2$ which has furthermore a heavy tail $\mu(k) \sim  \sqrt{ \frac{2}{\pi}} \cdot k^{-5/2}$ as $k \to \infty$. The following proposition is the probabilistic translation of the above combinatorial results:

\begin{proposition} \label{prop:RWcoding} Let  $  \mathcal{C}_{1}, \dots , \mathcal{C}_{n-m}$ be the components indexed from $1$ to $n-m$ in a uniform manner of a uniform unrooted unordered forest over $\{1,2, \dots ,n\}$ with $m$ edges. The vector of the sizes $$ ( \|\mathcal{C}_{i} \|_{\bullet}: 1 \leq i \leq n-m)$$ has the same law as the increments of a random walk $(S_{i} : 0 \leq i \leq n-m)$ started from $S_{0}=0$ with i.i.d.\ increments of law $\mu$ and conditioned  on $\{S_{n-m} = n\}$.  Furthermore, conditionally on their sizes $ (\| \mathcal{C}_{i}\|_{\bullet} : 1 \leq i \leq n-m)$ the (increasing relabeling of the) trees $ \mathcal{C}_{i}$ are independent (unrooted) uniform Cayley trees.
\end{proposition}
\begin{proof} The fact that conditionally on the vertices in each component, their increasing relabeled versions are independent Cayley trees is clear already in our way to obtain \eqref{eq:countingforest}. The same property holds true if we condition on the sizes of the components only. For the first point, notice that the probability that the increments of the walk are $k_{1}, \dots , k_{n-m}$ with $ k_1 + \dots + k_{n-m} = n$ is equal to $2^{n-m} \mathrm{e}^{-n} \prod_{i=1}^{n-m} \frac{k_{i}^{k_{i}-2}}{k_{i}!}$ which is proportional to \eqref{eq:countingforest} and where the proportionality factor only depends on $n$ and $m$. This proves the proposition. Note for the record that  we have  \begin{eqnarray} \label{eq:forestwalkproba} \mathbb{P}(S_{n-m}=n) = \frac{2^{n-m}(n-m)!}{\mathrm{e}^{n} n! } \cdot \#  \mathfrak{F}(n,m).  \end{eqnarray}

\end{proof}

%ON LAISSE SUR ARXIV \begin{remark} If instead of ordering the components of our forest in a uniform random manner, we order them according to their minimal elements, then the connection with the increments of a random walk does not hold. Indeed, in this ordering, the size of the first component is size-biaised and typically larger than a uniform component of the forest. \end{remark}

\label{sec:stabledef}
The above proposition still holds if we consider a random walk with step distribution generating function given by $ z \mapsto \mathbf{T}(z_{0})^{-1} \cdot  \mathbf{T}( z \cdot z_{0})$ for any $0 < z_{0} \leq \mathrm{e}^{-1}$. However, our choice of $z_{0}= \mathrm{e}^{-1}$ is the ``correct" probabilistic choice in the critical window $m = \frac{n}{2} + O(n^{2/3})$ and yields a measure $\mu$ with a heavy tail in the domain of attraction of the $3/2$-stable law. More precisely, we shall consider the  stable L\'evy process   $( \mathscr{S}_{t})_{t \geq 0}$ with index $3/2$ and only positive jumps,  which starts from $0$ and normalized so that its L\'evy measure is $$\frac{1}{\sqrt{2\pi}} |x|^{-5/2} \mathbf{1}_{x > 0},  \quad \mbox{or equivalently} \quad  \mathbb{E}[\exp( - \ell  \mathscr{S}_{t})] = \exp(  \tfrac{2^{3/2}}{3}t \ell^{3/2})$$ for any $\ell,t \geq 0$, see \cite[Section VIII]{Ber96}. We chose this normalization so that $n^{-2/3}S_{n/2} \xrightarrow[n\to\infty]{(d)} \mathscr{S}_1$ . By standard results \cite{Zol86}, for any $t >0$ the variable $ \mathscr{S}_{t}$ --which is distributed as a $3/2$-stable totally asymmetric spectrally positive random variable-- has a density with respect to the Lebesgue measure on $ \mathbb{R}$ which we denote by $ p_{t}(x)$ for $x \in \mathbb{R}$ and $t >0$. By the scaling property of $( \mathscr{S})$ we  have $$p_{t}(x) = t^{-2/3} p_{1}( x \cdot t^{-2/3}),$$
with $$p_1(x) =  -\frac{1}{2}  \mathrm{e}^{x^3/12}\left(x  \mathrm{Ai}\left(\frac{x^2}{4}\right) + 2  \mathrm{Ai}'\left( \frac{x^2}{4}\right)\right),$$ where $ \mathrm{Ai}$ is the Airy function. In particular,  $$p_1(0) = \frac{3^{1/6} \Gamma\left( \frac{2}{3}\right)}{2 \pi }. $$ 
The function $ p_{1}(x)$ (see Figure \ref{fig:airy}) is sometimes called the (map)-Airy distribution as in \cite{banderier2001random} (in the notation of \cite[Definition 1]{banderier2001random} we have $p_{1}(-x) = c \mathcal{A}(c x)$ with $c= \frac{1}{2}$ and in the notation of  \cite{martin2018critical} we have $p_{1}(x) = c g(cx)$ with $c = 2^{2/3}$).  In particular, it is a smooth positive function tending to $0$ at $\pm \infty$ and from \cite[Eq. (3)]{banderier2001random} we have the following asymptotics 
 \begin{equation} \label{eq:asympp1}
p_1 ( \lambda) \sim
\left\lbrace 
\begin{array}{lcl}
\frac{1}{\sqrt{2 \pi}} \sqrt{| \lambda|} \exp \left( - \frac{ | \lambda|^3}{6}\right) & \mbox{if} & \lambda \to - \infty \\
\frac{1}{\sqrt{2 \pi}} | \lambda|^{-5/2}& \mbox{if} & \lambda \to + \infty. \\

\end{array}\right. 
 \end{equation}
 
 \begin{figure}[!h]
  \begin{center}
  \includegraphics[width=8cm]{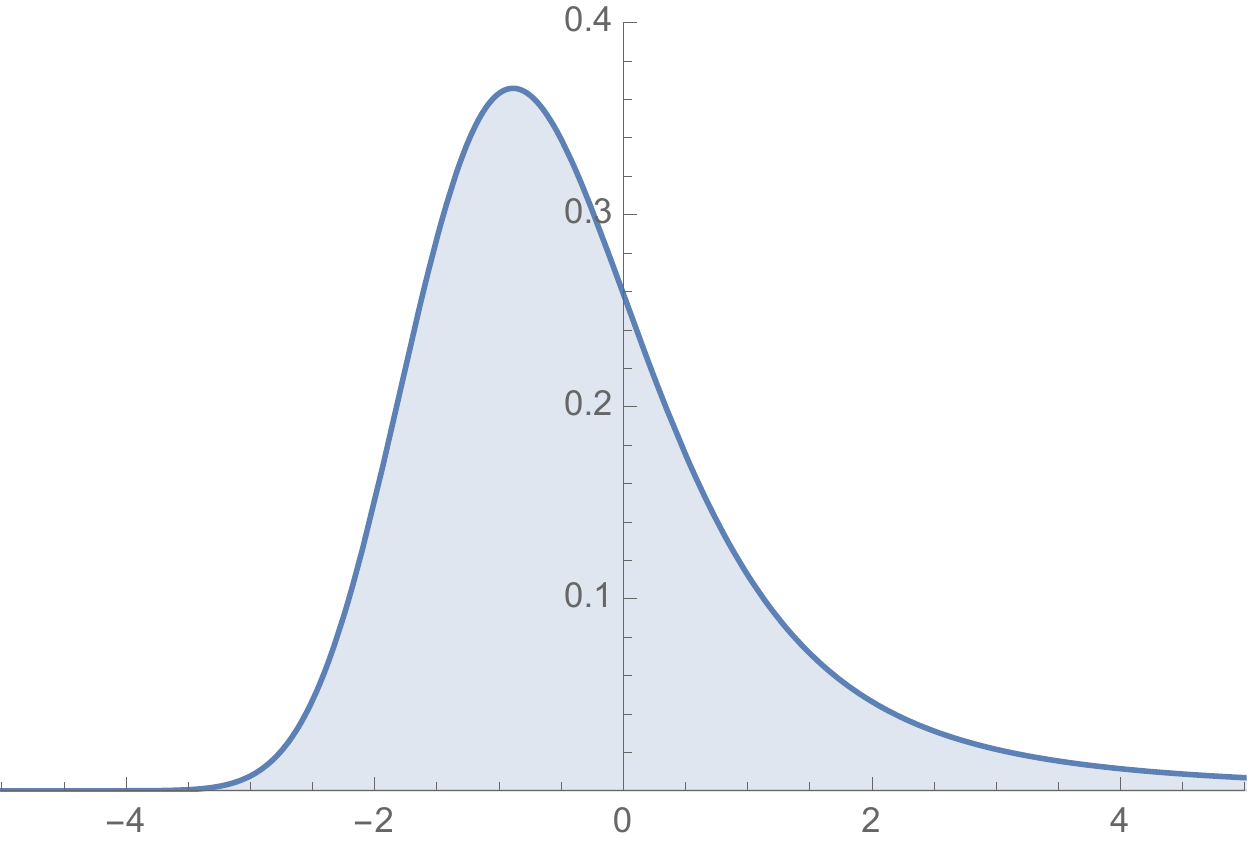}
  \caption{Plot of the density $p_{1}(\cdot)$ over $[-5,5]$. The function is rapidly decreasing to $0$ as $ x \to -\infty$ and polynomially decreasing to $0$ as $x \to \infty$. It is smooth and bimodal: increasing from $-\infty$ to $ \approx -0.886$ and then decreasing up to $\infty$. \label{fig:airy}}
  \end{center}
  \end{figure}

Using this notation and equipped with \eqref{eq:forestcount}, Britikov \cite{britikov1988asymptotic} computed the asymptotic of $\# \mathfrak{F} (n,m)$ as $n$ and $m$ go to $ \infty$. Those results are recalled here: 
\begin{lemma}[Britikov \cite{britikov1988asymptotic}] If $n, m \to \infty$, then 
\begin{equation*}
\# \mathfrak{F} (n,m) \sim
\left\lbrace 
\begin{array}{lcl}
\frac{n^{2m}}{2^m m!} \sqrt{1 - \frac{2m}{n}} & \mbox{if} & \frac{2m-n}{n-m} n^{1/3} \to - \infty \\
\frac{n^{n-1/6}}{2^{n-m}(n-m)!}  p_1( \lambda ) \sqrt{ 2 \pi } & \mbox{if} & m = \frac{n}{2} + \frac{ \lambda}{2} n^{2/3} \\ 
\frac{n^{n-2}}{2^{n-m-1} (n-m-1)!} \left( \frac{2m}{n} - 1\right)^{-5/2} & \mbox{if} & \frac{2m-n}{n-m} n^{1/3} \to + \infty.
\end{array}\right. 
 \end{equation*}

\end{lemma}

Those asymptotics are better understood on the variable $ \mathbb{P}(S_{n-m}=n)$ which is related to the number of forests by \eqref{eq:forestwalkproba}. Indeed, writing  $ m = \frac{n}{2} +  \frac{\lambda}{2}n^{2/3}$ and using the asymptotics on $p_1(\lambda)$ from  \eqref{eq:asympp1}, we see that as long\footnote{that is $\lambda\equiv \lambda_{n}$ may depend on $n$ but $n^{-1/3} \cdot \lambda_{n} \to 0$ as $n \to \infty$} as $|\lambda| \ll n^{1/3}$ then 
\begin{equation} \label{eq:asymtptoPSnm}
n^{2/3} \cdot \mathbb{P}(S_{n-m}=n) \sim p_{1}( \lambda),  \end{equation}
which can be seen as a strong form  of the local central limit for random variables in the domain of attraction of stable laws due to Gnedenko. 

\subsection{Free forest property}
In this section we establish a Markovian property of the frozen Erd{\H{o}}s--R\'enyi process $F(n,m)$. Recall that $D(n,m) = m - \|F(n,m)\|_{\edge}$ stands for the number of discarded edges up to time $m$ (i.e.\ the edges that have not been added in $F(n,m)$ because their starting point was in a frozen component) and  recall that $\|F(n,m)\|_{\bluecirc}$ is the number of vertices in the frozen (blue) components of $F(n,m)$. 

\begin{proposition}[Free forest property] \label{prop:freeforest} For any $n \geq 1, m \geq 0$, conditionally on $ \|F(n,m)\|_{\edge}$ and $\| F(n,m)\|_{\bluecirc}$ the (increasing relabeling of the) forest part $[F(n,m)]_{ \mathrm{tree}}$ is uniformly distributed over $$ \mathfrak{F}\big( n - \|F(n,m)\|_{\bluecirc}, \|F(n,m)\|_{\edge} - \|F(n,m)\|_{\bluecirc}\big).$$
\end{proposition}
 The proof of the proposition follows from two invariance properties of the law of  a uniform random forest which are described as follows. We shall write $W(n,m)$ for a uniform forest of $ \mathfrak{F}(n,m)$.
\begin{itemize}
\item \textit{Size-biased removal}. Pick $X \in \{1,2, \dots, n \}$ uniformly and independently of $W(n,m)$ and denote by $K$ the number of vertices of the tree containing the vertex $X$ in $W(n,m)$. Then conditionally on $K$, the forest obtained by removing the tree containing $X$ and relabeling the vertices in increasing order has the same law as $W(n- K, m-K+1)$.
\item \textit{Addition of one edge}. Pick $(X,Y) \in \{1,2, \dots, n \}^{2}$ uniformly and independently of $W(n,m)$ and let us add the edge $E = \{X,Y\}$ to the forest $W(n,m)$. If the addition of this edge creates a cycle, let us denote by $K$ the number of vertices of this component. Otherwise put $K=0$. Then conditionally on $K$, the forest obtained by adding $E$ to $W(n,m)$, and removing the corresponding component if this addition creates a cycle has the same law as $W(n-K, m-K+1)$ (as usual up to an order-preserving relabeling of the vertices).
\end{itemize}
The proof of these two facts is easily seen by counting arguments, see \cite[p 957 after Lemma 6.1]{martin2018critical} for a proof of the second one. In the first case, we call this operation a size-biased removal because the tree of size $K$  removed from $W(n,m)$ is not uniform over all components but biaised by its number of vertices. 
\begin{proof}[Proof of Proposition \ref{prop:freeforest}] We prove the proposition by induction on $m \geq 0$. For $m=0$ there is nothing to prove. We decompose the effect of the (tentative) addition of the edge $\vec{E}_{m+1}=(X_{m+1},Y_{m+1})$ to $F(n,m)$ in a two steps procedure. First, conditionally on $ \|F(n,m)\|_{\bluecirc}$ and $\|F(n,m)\|_{\edge}$ we decide whether:
\begin{enumerate}
\item $X_{m+1}, Y_{m+1} \in [F(n,m)]_{\bluecirc}$  with probability $n^{-2} \|F(n,m)\|_{\bluecirc}^{2}$,
\item $X_{m+1} \in [F(n,m)]_{\bluecirc}$ and  $Y_{m+1} \in [F(n,m)]_{\circ}$ with probability $n^{-2}\|F(n,m)\|_{\bluecirc}\cdot \|F(n,m)\|_{\circ}$,
\item $X_{m+1} \in [F(n,m)]_{\circ}$ and  $Y_{m+1} \in [F(n,m)]_{\bluecirc}$ with probability $n^{-2} \|F(n,m)\|_{\bluecirc}\cdot \|F(n,m)\|_{\circ}$,
\item or  $X_{m+1}, Y_{m+1} \in [F(n,m)]_{\circ}$ with probability $n^{-2} \|F(n,m)\|_{\circ}^{2}$
\end{enumerate}
In this first two cases the edge $E_{m+1}$ is not added and $F(n,m+1)=F(n,m)$. Conditionally on case $3$, the point $X_{m+1}$ is uniformly distributed over $[F(n,m)]_{\circ}$ and the addition of the edge $E_{m+1}$ will link the component of $X_{m+1}$ to a frozen component, freezing it. Since by induction, $[F(n,m)]_{\circ}$ was a uniform forest, we conclude by invariance under size-biased removal that $[F(n,m+1)]_{\circ}$ is again a uniform forest of $\mathfrak{F}\big( n - \|F(n,m+1)\|_{\bluecirc}, \|F(n,m+1)\|_{\edge} - \|F(n,m+1)\|_{\bluecirc}\big)$. Case $4$ is similar and we argue as above using the invariance property under addition of one edge. 
\end{proof}

%\begin{remark}  The above proposition is easily extended to the frozen Erd{\H{o}}s--R\'enyi process with parameter $p  \in [0,1]$, see Part \ref{sec:comments}. In particular, when $p=1$, the forest part corresponds to the tree component in the standard Erd{\H{o}}s--R\'enyi and this gives a Markovian property of the process $G(n,m)$, see Part \ref{sec:comments} for possible applications. \end{remark}

\begin{corollary}[Transitions of the size of the freezer and discarded edges] \label{cor:freezer}For every fixed $n \geq 1$, the process $$\big( \|F(n,m)\|_{\bluecirc},  \|F(n,m)\|_{\edge} : m \geq 0\big)$$ is a (inhomogeneous) Markov chain with transitions 
$$ \mathbb{P} \left(  \begin{array}{rcl}\Delta \|F(n,m)\|_{\bluecirc} &=& 0 \\ \Delta \|F(n,m)\|_{\edge} &=&0 \end{array}  \left|  \begin{array}{l}\|F(n,m)\|_{\bluecirc} \\  \|F(n,m)\|_{\edge} \end{array}\right)  \right.= \frac{ \| F(n,m)\|_{\bluecirc}}{n}, $$ and if we write $n'= \|F(n,m)\|_{\circ} = n- \|F(n,m)\|_{\bluecirc}$ and $m' =  \|F(n,m)\|_{\edge}- \|F(n,m)\|_{\bluecirc}$ 
  \begin{eqnarray*}
 \mathbb{P} \left(  \begin{array}{rcl}\Delta \|F(n,m)\|_{\bluecirc} &=& k \\ \Delta \|F(n,m)\|_{\edge} &=&1 \end{array}  \Bigg |  \begin{array}{l}\|F(n,m)\|_{\bluecirc} \\  \|F(n,m)\|_{\edge} \end{array}\right)  
&=& {n' \choose k } \frac{ k^{k-2} \# \mathfrak{F}(n' - k, m'  -k {+1})  }{\#\mathfrak{F}(n' , m')} \left( \frac{k^{2} + k \|F(n,m)\|_{\bluecirc} }{n^2}\right)\\
&=&  \mu(k)  (n'-m')  \frac{ \mathbb{P}(S_{n'-m'-1} = n'-k)}{ \mathbb{P}(S_{n'-m'}=n')}  \left( \frac{k^{2} + k \|F(n,m)\|_{\bluecirc}}{n^{2}}\right).
   \end{eqnarray*}
\end{corollary}

In particular if $ k = y n^{2/3}, \|F(n,m)\|_{\bluecirc} = x n^{2/3}$, $m = \frac{n}{2} + \frac{\lambda}{2} n^{2/3}$ for $x,y \geq0$, $\lambda \in \mathbb{R}$ and $m - \|F(n,m)\|_{\edge} = o(n^{2/3})$, using the asymptotic on the tail of $\mu$ given after \eqref{eq:defmu} together with \eqref{eq:asymtptoPSnm} we deduce that the last probability transitions are asymptotic to 
 \begin{eqnarray} \label{def:gxy}   \frac{1}{ y^{3/2}\sqrt{2\pi}} \cdot g_{x,\lambda}(y) \cdot n^{-4/3} \quad \mbox{ where } g_{x,\lambda}(y) :=  (y + x)\frac{p_{1}(\lambda-x-y)}{p_{1}(\lambda-x)},  \end{eqnarray} 
 and this asymptotic is uniform as long as $x,y, |\lambda| \ll n^{1/3}$ and $k \to \infty$. We will meet again the function $g_{x,\lambda}(y)$ in Section \ref{sec:scalingjump} when dealing with the scaling limit of the process $ \| F(n,m)\|_{\bluecirc}$ in the critical window.

\section{Enumerative consequences} \label{sec:enumconsq}
In this section, we derive enumerative consequences of the coupling between the parking process on Cayley trees and the frozen  Erd{\H{o}}s--R\'enyi. In particular we recover much of the results of \cite{LaP16}. The reader may also find a discussion about enumeration of (strongly or fully) parked trees with outgoing flux at the end of the paper (Section \ref{sec:GF}).

We denote by $ \mathrm{PF}(n,m)$ (resp. $ \mathrm{PF}_{ \mathrm{root}}(n,m)$) the number of configurations made of  $m$ labeled cars arriving on the vertices of a Cayley tree over $\{1,2,\dots ,n\}$ so that all cars can park i.e.\ no outgoing flux (resp.~so that the root vertex does not contain a car after the parking process). These numbers thus count the parking functions on Cayley trees \cite{LaP16,panholzer2020parking}. In particular, the number of fully parked trees of size $n$ is $ \mathrm{PF}(n,n)$ whereas the number of nearly parked trees of size $n$ is $ \mathrm{PF}_{ \mathrm{root}}(n,n-1)$. Also for $m \leq n-1$ we have that $ \mathrm{PF}_{ \mathrm{root}}(n,m) \leq { \mathrm{PF}}(n,m)$ and actually Lemma  \ref{lem:rootflux} shows that 
  \begin{eqnarray} \label{eq:relationPFPFtilde} \mathrm{PF}_{ \mathrm{root}}(n,m) = \left(1- \frac{m}{n}\right) \cdot \mathrm{PF}(n,m).  \end{eqnarray}

\subsection{Exact counting and asymptotics for parking functions}

We start by proving Proposition \ref{prop:acyclicity} stated in the Introduction: By Proposition \ref{prop:CouplingCayleyFrozen}, the probability that the root of a uniform Cayley tree of size $n$ is not parked after $m$ i.i.d~uniform car arrivals is the probability that $F (n,m)$ (hence $G(n,m)$) contains no cycle (i.e.\ no frozen blue component). In that case, the graph $G(n,m)$ must be an unrooted forest. Therefore we have 
\begin{eqnarray} 
 \mathbb{P}(\mbox{the root of } T_n \mbox{ is not occupied by one of the first $m$ cars})= \mathbb{P}( G(n,m) \mbox{ has no cycle}),\label{eq:onreprend}
  \end{eqnarray}
  and so Proposition \ref{prop:acyclicity} follows by combining the last display with \eqref{eq:relationPFPFtilde}. We can actually go further and give a formula
for the number $ \mathrm{PF}_{ \mathrm{root}}(n,m)$:
 \begin{proposition}  \label{prop:countFP}For $0 \leq m \leq n-1$ we have
\begin{equation}  \mathrm{PF}_{ \mathrm{root}}(n,m) = \frac{n^{n-m} m!n!}{(n-m)!} \sum_{i = 0}^{n-m} \binom{n-m}{i} \frac{(-1)^{i}2^{m-i} n^{m-i-1} (n+i-m)}{(m-i)!}.
\end{equation}
In particular, when $m = n-1$, the number of nearly parked trees of size $n$ (see Section \ref{sec:FPTandNPT}) is equal to
\begin{equation}  \mathrm{PF}_{ \mathrm{root}}(n,n-1) = 2^{n-1} (n-1)! n^{n-2}.
\end{equation}
\end{proposition}

\begin{proof} Equation \eqref{eq:onreprend} can be rewritten as 
  \begin{eqnarray*}  \mathrm{PF}_{ \mathrm{root}}(n,m) & =& n^{n-1}\cdot n^{m} \mathbb{P}( G(n,m) \mbox{ has no cycle}),\nonumber \\
&=&  n^{n+m-1} \sum_{f \in \mathfrak{F} (n,m)} \frac{1}{n^{2m}} m! 2^m \\
&=& n^{n-m-1} m! 2^m \cdot \# \mathfrak{F}(n,m)
\end{eqnarray*}
and the result follows after plugging in \eqref{eq:forestcount}.
\end{proof}

By combining the above proposition with \eqref{eq:relationPFPFtilde} we find an exact expression of $ \mathrm{PF}(n,m)$ for $m \leq n-1$, see \cite[Theorem 4.5]{LaP16}  for a different\footnote{Obviously the two expressions coincide numerically, but we have not been able to transform one into the other} expression. Plugging the asymptotics of Section \ref{sec:countforest} we also recover (and extend) the asymptotics of \cite[Theorem 4.6]{LaP16} namely
\begin{equation*}
 \left(1- \frac{m}{n}\right)\cdot \mathrm{PF}(n,m)  \sim
\left\lbrace 
\begin{array}{lcl}
\displaystyle n^{n+m-1} \sqrt{1 - \frac{2m}{n}} & \mbox{if} & \frac{2m-n}{n-m} n^{1/3} \to - \infty \\

\displaystyle \frac{2^{2m-n}m!}{(n-m)!} n^{2n-m-1} p_1( \lambda ) \sqrt{ 2 \pi }  n^{-1/6}& \mbox{if} & m = \frac{n}{2} + \frac{ \lambda}{2} n^{2/3} \\ 
\displaystyle \frac{2^{2m-n+1}n^{2n-m-3}m!}{ (n-m-1)!} \left( \frac{2m}{n} - 1\right)^{-5/2} & \mbox{if} & \frac{2m-n}{n-m} n^{1/3} \to + \infty.
\end{array}\right. 
 \end{equation*}

\subsection{Enumeration of parked trees}
 In the case $m=n$, Equation \eqref{eq:relationPFPFtilde} is meaningless and does not enable us to compute $ \mathrm{PF}(n,n)$. To do so, we shall use a decomposition at the root of a nearly parked tree and recover  \cite[Theorem 3.2]{LaP16}. We introduce the (exponential) generating functions for nearly parked trees, fully parked trees and strongly parked trees
   \begin{eqnarray}
    \mathbf{N}(x) = \sum_{n \geq 1} \frac{ \mathrm{PF}_{  \mathrm{root}}(n,n-1)}{n!(n-1)!} x^{n}, \quad       \mathbf{F}(x) = \sum_{n \geq 1} \frac{\mathrm{PF}(n,n)}{(n!)^{2}} x^{n}, \quad 
                \mathbf{S}(x) = \sum_{n \geq 1} \frac{ \mathrm{SP}(n,n)}{(n!)^{2}} x^{n},
                \end{eqnarray} where $ \mathrm{SP}(n,n)$ is the number of strongly parked tree of size $n$. By Proposition \ref{prop:countFP} we have $ \mathbf{N}(x) = \frac{1}{2} \mathbf{T}(2x)$.
\begin{lemma} \label{lem:FPTnn} The number of fully parked trees with $n$ vertices and $n$ cars is 
$$ \mathrm{PF}(n,n)=((n-1)!)^2 \sum_{j = 0}^{n-1} \frac{ (n-j) \cdot (2n)^j}{j!}.$$
\end{lemma}
\begin{proof} Performing a decomposition at the root of a nearly parked tree (the trees attached to the root of a nearly parked tree are fully parked trees up to an order-preserving relabeling of the vertices and of the cars) we obtain for $n \geq 1$: 
\begin{eqnarray*}  
\mathrm{PF}_{ \mathrm{root}}(n,n-1) &=& \sum_{k \geq 0} \frac{1}{k!} \mathop{\sum_{ \sum_{i = 1}^{k} n_i = n-1}}_{n_i \geq 1}   \binom{n-1}{n_1 , \dots , n_k} \binom{n}{1,n_1 , \dots , n_k}\prod_{i = 1}^{k}  \mathrm{PF}(n_{i},n_{i})  \\
&=& \sum_{k \geq 0} \frac{1}{k!} \mathop{\sum_{ \sum_{i = 1}^{k} n_i = n-1}}_{n_i \geq 1} n! (n-1)! \prod_{i = 1}^{k} \frac{ \mathrm{PF}(n_{i},n_{i})}{(n_i!)^2} 
\end{eqnarray*}

Here $k$ denotes the number of subtrees attached to the root and the factor $1/k!$ corresponds to the $k!$ reorderings of the subtrees that represent the same tree. This is equivalent to the following equation on generating functions 
 \begin{eqnarray}
\mathbf{N}(x) = x \cdot  \mathrm{exp}( \mathbf{F}(x)). \label{eq:N=xexp(F)}
  \end{eqnarray}
Recalling that $ \frac{1}{2} \mathbf{T}(2x) = \mathbf{N}(x) $ and using the classical relations (see for instance \cite{moon1970counting})
\begin{equation} \label{eq:lienT} \mathbf{T}(x) = x \mathbf{T}'(x) - \frac{1}{2} (x \mathbf{T}'(x))^2 \quad \mbox{ and } \quad  \mathbf{T}'(x) =  \exp( x \mathbf{T}'(x)),
\end{equation} 
we deduce that 
$$ \mathbf{F}(x) = 2x \mathbf{T}'(2x) + \ln \left( 1 - x \mathbf{T}'(2x)\right).$$
This relation is the same as that obtained by Lackner \& Panholzer in \cite[Equation 5]{LaP16}.  Using Lagrange inversion formula on $x \mathbf{T}'(x)$ via \eqref{eq:lienT} (right) we obtain
$$ [z^n] \ln \left( 1 - x \mathbf{T}'(2x)\right) =- \frac{1}{n} \sum_{j=0}^{n-1} \frac{(2n)^j}{j!},$$ and straightforward calculations yield the result.
\end{proof}

In turn the enumeration of fully parked  trees can be used to count strongly parked trees by a simple substitution operation. This was already done by King \& Yan in \cite{king2019prime} but we recall it to prepare the reader to similar decompositions in Section \ref{sec:componentsLLN}.
 
\begin{proposition}[King \& Yan \cite{king2019prime}] \label{prop:kingyan} For $n \geq 1$ we have $ \mathrm{SP}(n,n)=(2n-2)!$. \end{proposition}
\proof A fully parked tree can be decomposed into the strong component of the root vertex (which can be reduced to a single vertex) on which fully parked trees are attached. This decomposition translates into the following equation for $n \geq 1$,
$$ \mathrm{PF}(n,n) = \sum_{n_0=1}^{n} \mathrm{SP}(n_0,n_0) \sum_{\begin{subarray}{c} k_1, k_2, \dots , k_{n_0} \geq 0 \\ \sum k_i = K
\end{subarray}} \prod_{j=1}^{n_0}\frac{1}{k_j!} \sum_{ \begin{subarray}{c}n_1, \dots , n_K \geq 1\\
\sum n_i = n-n_0 \end{subarray}} \binom{n}{n_0,n_1,\dots ,n_K} \prod_{j=1}^K \mathrm{PF}(n_j,n_j).$$

Summing over $n \geq 1$, we obtain 
\begin{eqnarray} 
\label{eq:Fu=St(xExp(Fu))}
\mathbf{F}(x) = \mathbf{S}( x \cdot \exp( \mathbf{F}(x))),  \end{eqnarray} see \cite[Section 3]{king2019prime}. Solving the above equation (see \cite{king2019prime}), we obtain $ \mathbf{S}(x) =1 - \ln(2) - \sqrt{1 - 4 x} + \ln\left(1 + \sqrt{1 - 4 x}\right)$, whose derivative is simply the usual generating function of the Catalan numbers $\binom{2n}{n}/(n+1)$ i.e.\ $ \mathbf{S}'(x) = (1 - \sqrt{1-4x})/(2x)$, hence $ \mathrm{SP}(n,n) = (2n-2)!$.
\endproof
In Section \ref{sec:GF} we show how the above results can be extended to enumerate exactly and asymptotically fully/strongly parked trees with a positive outgoing flux at the root. In particular, those problems are very similar to the enumeration of random planar maps with a boundary.

\section{Geometry of parked trees}
In this section we study the \emph{geometry} of the components, specifically the near components, in the parking process. We prove that uniform nearly parked trees of size $N$ have height of order $N^{3/4}$ and total flux of order $ N^{5/4}$.  We expect that those large scale properties are shared by the fully or strongly parked trees (and that versions of Conjecture \ref{conjectureGFT} hold for them). 
We start by describing the decomposition of a uniform nearly parked of size $N$ into strongly/fully parked components.
\subsection{Law of large numbers for components} \label{sec:componentsLLN}

Recall from Section \ref{sec:FPTandNPT} (see Figure \ref{fig:components}) the definition of nearly/fully/strongly parked trees as the components (different from the root component and possibly from isolated vertices)  of the subforests  $$T_{ \mathrm{strong}}(n,m) \subset \mathrm{T}_{ \mathrm{full}}(n,m) \subset  T_{  \mathrm{near}}(n,m).$$ 
We saw in the proof of proof of Lemma \ref{lem:FPTnn} that a nearly parked tree can be decomposed at the root into a forest of fully parked trees. Going further, we can decompose each fully parked tree into a forest of strongly parked trees after removing the edges without flux, see Figure \ref{fig:strong-bitype}. In this decomposition, each nearly parked tree $\mathfrak{n}$ is associated with a bitype rooted tree $ \mathrm{ Bitype} ( \mathfrak{n})$ such that the vertices at even generation are disks $\circ/\bullet$ and those at odd generations are squares~$\square$:
\begin{itemize}
\item Each parked vertex of $ \mathfrak{n}$ corresponds to a disk vertex $\bullet$ in $\mathrm{ Bitype} ( \mathfrak{n})$, and the empty root vertex of $ \mathfrak{n}$ corresponds to the root vertex $\circ$ of $\mathrm{ Bitype} ( \mathfrak{n})$.
\item The children of each disk vertex in $\mathrm{ Bitype} ( \mathfrak{n})$  are square vertices which correspond to the strongly parked components of $ \mathfrak{n}$ linked to this vertex by edges with zero flux. 
\item The children of each square vertex in $\mathrm{Bitype} ( \mathfrak{n})$ correspond to the vertices of the strongly parked component of $ \mathfrak{n}$ above the corresponding square.
\end{itemize}

\begin{figure}[!h]
 \begin{center}
 \includegraphics[width=12cm]{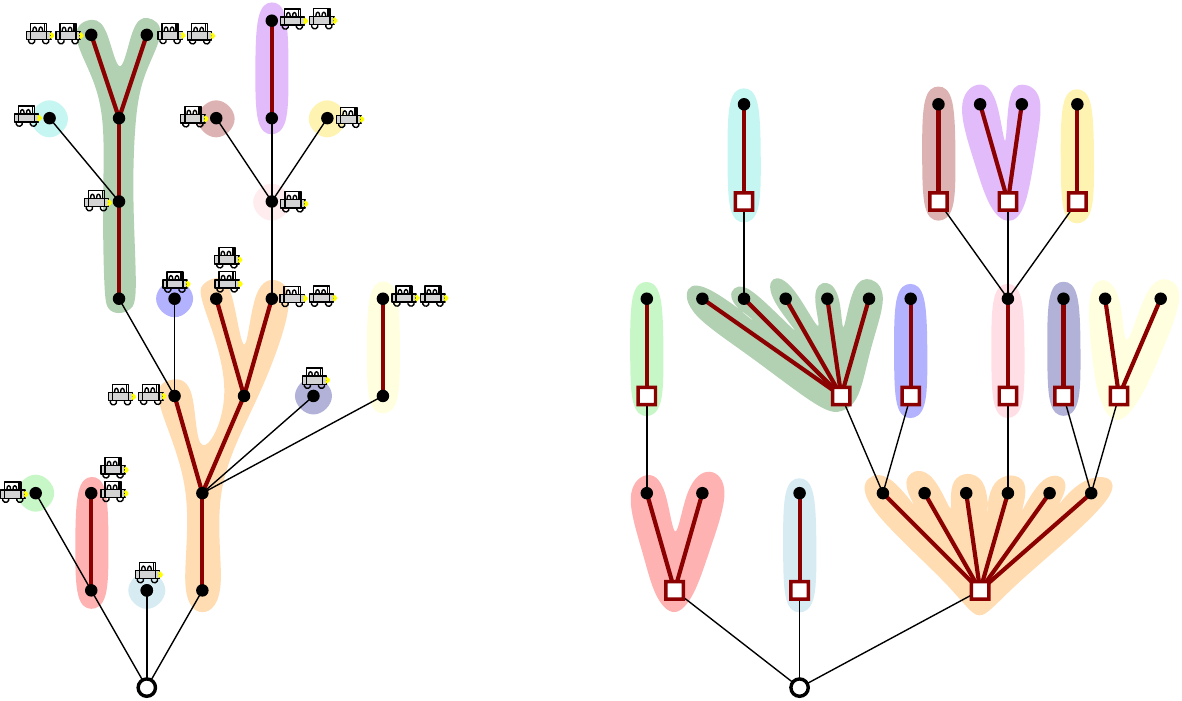}
 \caption{\label{fig:strong-bitype}(Left) A nearly parked tree $ \mathfrak{n}$ with $26$ vertices.  (Right) Its decomposition $ \mathrm{Bitype}( \mathfrak{n})$ into the tree of strong components. Notice that the square vertices (at odd generations in $ \mathrm{Bitype}( \mathfrak{n})$) correspond to the strong components and their degrees are the sizes of the components. Whereas $ \mathfrak{n}$ is unordered, the tree $ \mathrm{Bitype}( \mathfrak{n})$ is ordered for convenience. }
 \end{center}
 \end{figure}
For convenience, the tree $\mathrm{ Bitype} ( \mathfrak{n})$ is given a plane orientation by fixing independently for each vertex an order on its children. Recall that by Proposition \ref{prop:countFP} we have $ \mathbf{N}(x) = \frac{1}{2} \mathbf{T} (2x)$ and combining it with \eqref{eq:N=xexp(F)} and \eqref{eq:Fu=St(xExp(Fu))}, we get
\begin{eqnarray}\label{eq:Ncritic} \mathbf{N}\left(\frac{1}{2 \mathrm{e}}\right) = 1/4 < \infty \quad \mbox{and} \quad  \mathbf{F}\left(\frac{1}{2 \mathrm{e}}\right)= \mathbf{S}\left(\frac{1}{4}\right) = 1 - \ln(2) < \infty. \end{eqnarray} 
Therefore, we can define a random nearly parked tree $P$  under the critical Boltzmann distribution, i.e.\ with law
  \begin{eqnarray} \mathbb{P} \left( P = \mathfrak{n}\right) = \frac{4}{ (\|\mathfrak{n}\|_{\bullet}-1)! \|\mathfrak{n}\|_{\bullet}!} \left( \frac{1}{2 \mathrm{e}}  \right)^{ \|\mathfrak{n}\|_{\bullet}}.    \label{eq:BoltzcritP}\end{eqnarray}

\begin{lemma}\label{lem:nearlybitype} The tree $ \mathrm{Bitype} ( P)$  has the law of a bitype alternating Bienaym\'e--Galton--Watson (BGW) tree where (disk) vertices at even generations have  Poisson offspring distribution $\nu_{\bullet}$ with mean $ \mathbf{F}((2 \mathrm{e})^{-1})$ and where (square) vertices at odd height have offspring distribution $\nu_{\square}$ given by   $$ \nu_{\square} (n) = \frac{[x^n]  \mathbf{S}(x/4)}{ \mathbf{S} (1/4)} =  \frac{4^{-n}}{ 1- \ln (2)} \frac{(2n-2)!}{ (n!)^{2}}, \quad \mbox{ for } n \geq1.
$$

\end{lemma}
\proof

Let $ \mathfrak{t}$ be a fixed bitype alternating plane rooted tree starting at a disk vertex and let us denote by $ (\square_{i})_{1\leq i \leq n_{\square}}$ its square vertices, by $(k_{\square_i})_{1\leq i \leq n_{\square}}$ their respective number of children all of which should be positive, by $ (\bullet_{i})_{1\leq i \leq n_{\bullet}}$ its disk vertices and by $(k_{\bullet_i})_{1\leq i \leq n_{\bullet}}$ their respective number of children. Notice that   \begin{eqnarray} \label{eq:sumvertices} \sum_{i = 1 }^{n_{\bullet}} k_{\bullet_i} = n_{\square}  \quad  \mbox{and} \quad  \sum_{i = 1 }^{n_{\square}}k_{\square_i} = n_{\bullet} - 1.  \end{eqnarray} The probability that the BGW tree described in the lemma equals $ \mathfrak{t}$ is  $$ \prod_{i = 1}^{n_{\bullet}}
  \mathrm{e}^{- \mathbf{F}((2 \mathrm{e})^{-1})} \frac{ \mathbf{F}((2 \mathrm{e})^{-1})^{k_{\bullet_i}}}{k_{\bullet_i}!}
\prod_{i = 1}^{n_{\square}}
\frac{ (1/4)^{k_{\square_i}} \mathrm{SP}(k_{\square_i},k_{\square_i})}{ \mathbf{S}( 1/4) (k_{\square_i}!)^2}.$$
By counting the number of ways to partition the vertices $\{1,2, \dots , n_{\bullet}\}$ and assign a strongly parked tree to each square vertex of $ \mathfrak{t}$, recalling \eqref{eq:BoltzcritP}, we deduce that the probability that  ${ \mathrm{Bitype}}(P) = \mathfrak{t}$ is equal to

   $$ \frac{4 \cdot (2 \mathrm{e})^{- n_{\bullet}} }{(n_{\bullet}-1)!n_{\bullet}!}  \times \binom{n_{\bullet}}{1, k_{\square_1}, \dots, k_{\square_{n_{\square}}}}\binom{n_{\bullet}-1}{k_{\square_1}, \dots, k_{\square_{n_{\square}}}} \prod_{i = 1}^{n_{\square}} \mathrm{SP}(k_{\square_i},k_{\square_i}) (k_{\square_i})! 
 \times \prod_{i=1}^{n_{\square}} \frac{1}{k_{\square_i}!} \prod_{i=1}^{n_\bullet} \frac{1}{k_{\bullet_i}!}.$$
Using  \eqref{eq:sumvertices}, \eqref{eq:Fu=St(xExp(Fu))} and \eqref{eq:Ncritic} it can be easily checked that the above two probabilities are the same and we get the desired result.
\endproof

This decomposition is used in the following lemma which states that inside a large uniform nearly parked tree of size $N$, there is an essential unique fully parked tree of size $N - O_{ \mathbb{P}}(1)$ containing  an essentially unique strongly parked tree  of size $N/2 + o_{  \mathbb{P}}(N)$. The proof is based on a condensation phenomenon for conditioned subcritical Bienaym\'e--Galton--Watson  trees and is similar to the approach of Addario-Berry to block size in random planar maps  \cite{addario2019probabilistic}. This will be used in the proof of Theorem \ref{thm:composantes} when dealing with full components.

\begin{proposition}\label{prop:nearlytofully} Let $P_N$ be a uniform nearly parked tree of size $N$ and consider $MF(P_N)$ the fully parked tree of maximal size above its root and $MS(P_N)$ the strongly parked component of maximal size included in $P_N$. Then we have 
$$ \frac{\|MF(P_N)\|_{\bullet}}{N}  \xrightarrow[N\to\infty]{( \mathbb{P})}1, \qquad \frac{\|MS(P_N)\|_{\bullet}}{N} \xrightarrow[N\to\infty]{( \mathbb{P})} \frac{1}{2},$$
furthermore the second largest fully parked tree is of size $O_{ \mathbb{P}}(1)$ and the second largest strongly parked tree is of size $ O_{ \mathbb{P}}( N^{2/3})$.

\end{proposition}

During the proof we shall need a well-know ``big-jump'' lemma for which we provide some details for the reader's convenience. See \cite[Lemma 2.5]{AS03} or \cite[Lemma 3.3]{CKdissections} for similar results and \cite{armendariz2011conditional,asmussen2003asymptotics} for generalizations.
\begin{lemma}[Single big-jump in a random sum] \label{lem:singlejump} Let $Z_{1}, \dots , Z_{i}, \dots$ be i.i.d.~random variables of law $\nu$ having a heavy tail $\nu(n) \sim c n^{-\alpha}$ for some $c>0$ and $\alpha >1$. We let $K$ be a random variable independent of the $Z_{i}$'s and having some exponential moment $ \mathbb{E}[ \mathrm{e}^{\delta K}]< \infty$ for some $\delta >0$. We consider the random sum 
$$ \mathfrak{S} = \sum_{i=1}^{K} Z_{i}.$$
Then conditionally on $\{ \mathfrak{S}=N\}$, if we remove the largest term $Z_{i}$ for $i \in \{1,2, \dots , K\}$ from $(Z_1, \dots , Z_K)$, then the remaining random vector converges in law towards $(Z_{1}, \dots , Z_{\overline{K}-1})$ where $\overline{K}$ is the size-biased variable $K$ independent of the $Z_i$'s. In particular $N - \max _{1 \leq i \leq K} Z_{i}$ converges in law as $N \to \infty$.
\end{lemma}
\begin{proof} Since $\nu$ is a regular polynomial tail and $K$ has exponential moments, it  follows from  \cite[Theorem 3 (i)]{asmussen2003asymptotics} that 
 \begin{eqnarray} \lim_{N \to \infty}\frac{\mathbb{P}(  \mathfrak{S}=N)}{ \mathbb{P}(Z_{1}=N)} = \mathbb{E}[K],   \label{eq:tail}\end{eqnarray} (in our cases of applications, this can directly be checked by a calculation using  generating functions). Now, fix $k\geq 1$, fix values $n_{1}, \dots , n_{k-1}$ and denote by $\tilde{X}_{1}, \dots , \tilde{X}_{K-1}$ the re-indexed variables $\{X_{i} : 1 \leq i \leq K \mbox{ with } i \ne \mathrm{argmax}_{1 \leq i \leq K} Z_{i}\}$. Then for $N$ large we have
 \begin{eqnarray*} && \mathbb{P}(K= k \mbox{ and } \tilde{X}_{j}=n_{j} \mbox{ for } 1\leq j \leq k-1 \mid  \mathfrak{S}=N) \\
 &=&  \frac{1}{ \mathbb{P}(  \mathfrak{S}=N)}\sum_{k\geq 1} \mathbb{P}(K=k)  \sum_{i =1}^{k}\nu(n_{1})\cdots \nu \left(N-\sum_{j \ne i} n_{j}\right) \cdots \nu(n_{k-1})\\
 & \xrightarrow[N\to\infty]{\eqref{eq:tail}}& \frac{1}{ \mathbb{E}[K]} \sum_{k \geq 1} k \mathbb{P}(K=k) \nu(n_{1})\cdots \nu(n_{k-1}). \end{eqnarray*}
Since the above probabilities sum to $1$, this implies the desired convergence in law. \end{proof}

\proof Let us start with the case of the fully parked tree of maximal size. By the decomposition of nearly parked trees at the root vertex (proof of Lemma \ref{lem:FPTnn}), the size of the critical Boltzmann nearly parked tree $P$ can be written as $ 1 + \sum_{i=1}^{K} Z_{i}$ where $K$ is a Poisson random variable of mean $1- \ln(2)$ independent of $Z_{1}, Z_{2}, \dots , Z_{i}, \dots$ which are the sizes of i.i.d.~critical Boltzmann fully parked trees, i.e. with $\mathbb{P}(Z_{i} = n) = [x^{n}] \mathbf{F}( x/(2 \mathrm{e}))/ (1- \ln (2)) \sim  \sqrt{\frac{2}{\pi}} \frac{1}{1- \ln 2} n^{	-5/2}$ as $n \to \infty$. We can thus directly apply  Lemma \ref{lem:singlejump} and deduce that when we condition $1 + \sum_{i=1}^{K} Z_{i}$ to be equal to $N$, then as $N \to \infty$ with high probability one of the $Z_{i}$ is of order $N- O_{ \mathbb{P}}(1)$. This translates into the desired result on $ \|MF(P_{N})\|_{\bullet}$.

Let us now move to the case of strongly parked tree. By Lemma \ref{lem:nearlybitype} the variable $\|MS(P_N)\|_{\bullet}$ is equal in law to the maximal degree of a square vertex in the alternating bitype BGW tree with offspring distribution $(\nu_{\bullet}, \nu_{\square})$ conditioned to have $N$ disk vertices in total. We shall first consider the monotype Bienaym\'e--Galton--Watson tree obtained by ``skipping" the odd generations i.e. with offspring distribution $\xi$ given by $\sum_{i=1}^{K} S_{i}$ where $K$ has Poisson distribution with mean $1-\ln(2)$ independent of  the $S_i$'s which are i.i.d.~with  distribution $ \nu_{\square}$. This BGW tree is subcritical since $$ \mathbb{E}[K] \cdot \mathbb{E}[S_1] =  x \mathbf{S}'(x)\big|_{x=1/4} =  \frac{1}{2},$$ and furthermore it has a regular varying heavy tail $ \mathbb{P}(\xi=n) \sim \frac{1-\ln (2)}{4  \sqrt{\pi}} n^{-5/2}$ as $n \to \infty$. Here also a ``big-jump'' or ``condensation'' phenomenon appears \cite{Jan12b,kortchemski2015limit} and it is known that the maximal degree of such a tree is of order $N/2$, whereas the second largest is of order $N^{2/3}$ with high probability. We then condition on the value of $D$ and remember that this degree has been obtained as $D = \sum_{i=1}^K S_i$. We can thus apply Lemma \ref{lem:singlejump} and deduce that when $D$ is large, the largest degree of the square vertices contributing to $D$ is $D - O_ \mathbb{P}(1)$. \endproof 

After all these combinatorial decompositions, the following should come as no surprise:

 \begin{proposition} \label{prop:FPT/AFPTuniform} Conditionally on their component sizes and after relabeling, the non-blue strong (resp.\ full, resp.\ near) components in $T_{n}$ after cars $X_{1}, \dots , X_{m}$ have parked, are  independent uniform strongly (resp.\ fully, resp.\ nearly) parked trees. \end{proposition}
 \begin{proof}[Proof (sketch)]To fix ideas, let us consider the case of the full components. Fix $n\geq 1$ and $m \geq 0$ and let us condition on everything except the internal structure of the fully parked trees (obtained after relabeling of the vertices and cars as usual) of $T_{ \mathrm{full}}(n,m)$. That is, we reveal the partition of $\{1,2, \dots , n \}$ into the full components,  the induced partition of the cars $\{1,2, \dots , m\}$, the edges of $T_{n}$ between empty vertices, as well as the possible blue tree of $T_{ \mathrm{full}}(n,m)$ containing the root. It should be clear then that any fully parked of the proper size can appear in each component, so that the probability of seeing a given configuration is proportional to $ \prod_{i=1}^{k} \mathrm{PF}(n_{i},n_{i})$. The result follows. The case of near or strong components is similar.  \end{proof}

 \subsection{Height and Flux}
In this section we use our coupling construction of Section \ref{sec:couplingtrees} specified to the case of nearly parked trees to deduce some geometric information on the latter. We use the letter $N$ to denote the size of the nearly parked tree not to confuse with the size $n$ of the underlying Cayley tree.
\subsubsection{``Coupling construction'' of nearly parked trees}

\label{sec:couplingback}

Fix $N \geq 1$ and denote by $ P_{N}$ a uniform random nearly parked tree with $N$ vertices (chosen among the $2^{N-1} (N-1)! N^{N-2}$ possibilities, according to Proposition \ref{prop:countFP}).  This is a random rooted Cayley tree over $\{1,2, \dots , N \}$ which carries $N-1$ cars arrivals $X_{i} \in \{1,2, \dots , N \}$ so that after the parking process, all cars are parked and the root of the tree is free. We can obtain such a random tree by applying the coupling construction of Section \ref{sec:couplingtrees}  with the oriented edges $( \vec{E}_{i} : 1 \leq i \leq N-1)$ on the event when the unoriented edges $({E}_{i} : 1 \leq i \leq   N-1)$ do not create any cycle. In such case, the graph $G(N, N-1)$ is simply a uniform (unrooted) Cayley tree $ T_{N}^{u}$ ($u$ stands for unrooted) and the edges $ (\vec{E}_{i}:1 \leq i \leq N-1)$ can be obtained by labeling the edges of $T_{N}^{u}$ by $\{1,2, \dots ,N-1\}$ uniformly at random and given random independent orientations. We shall denote by $\vec{e}_{1}, \vec{e}_{2}, \dots , \vec{e}_{N}$ the labeled oriented edges  of $T_{N}^{u}$ (they correspond to $ \vec{E}_{1}, \dots , \vec{E}_{N}$ on the appropriate event) and by $ \vec{r}_{1}, \dots, \vec{r}_{N}$ their redirections which form the nearly parked tree $P_{N}$ (where the oriented edges are directed towards its root). In this special case, the coupling presented in the proof of Proposition \ref{prop:CouplingCayleyFrozen} (or Proposition \ref{coupling:mapping}) is very simple: With the same notation, the $i$th car arriving on vertex $X_{i}$ will \emph{always} find a parking spot $\zeta_{i}$ and we redirect the edge $ \vec{e}_{i}=(X_{i},Y_{i})$ into $ \vec{r}_{i} = (\zeta_{i}, Y_{i})$. Since we never encounter loops, we never have $\zeta_{i} = \dagger$ and never create any ``blue'' component. See Figure \ref{fig:AFPT}.

\begin{figure}[!h]
 \begin{center}
 \includegraphics[width=13cm]{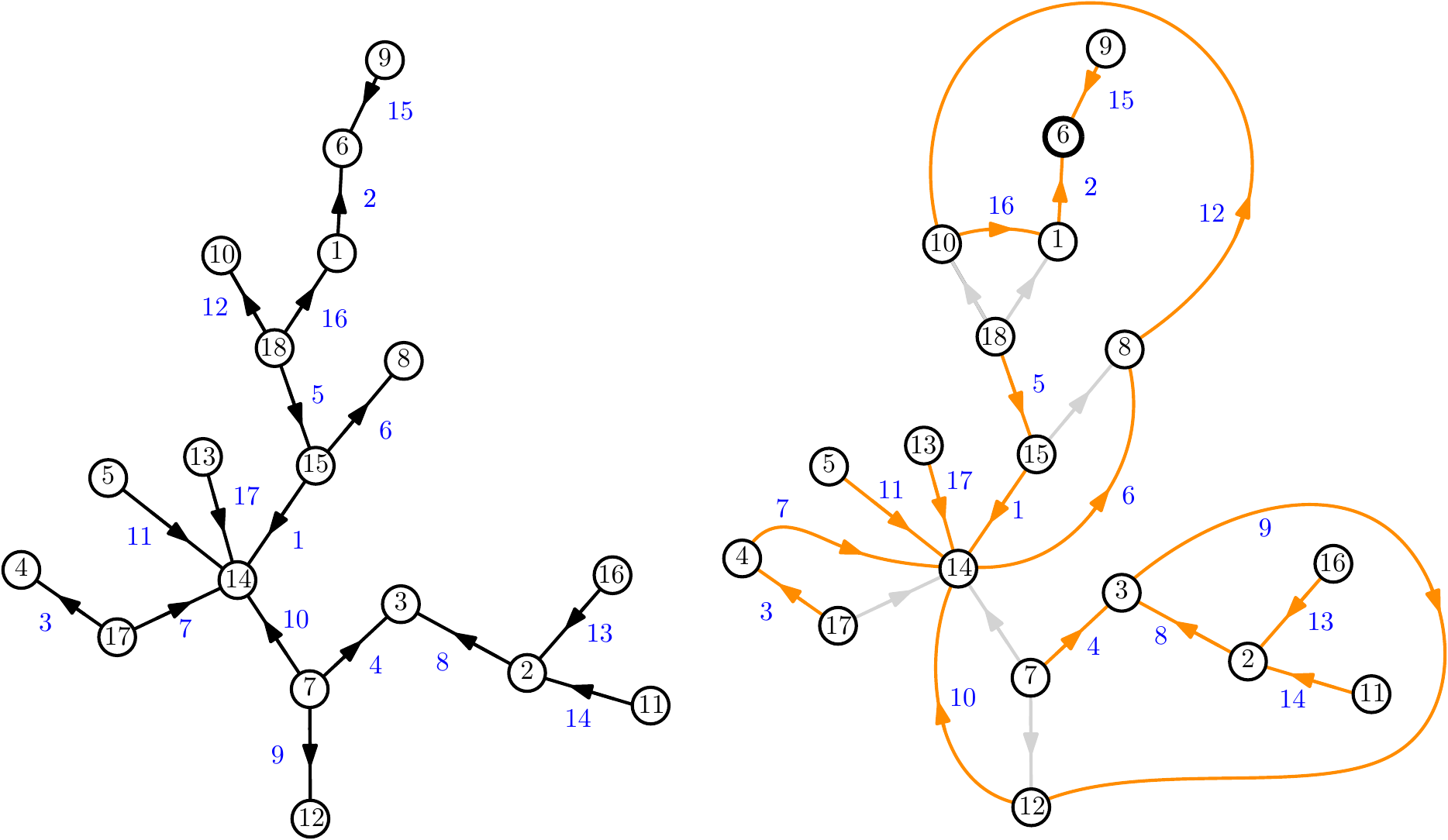}
 \caption{Illustration with $N=18$ of the construction of  a nearly parked tree  from a uniform unrooted Cayley tree whose edges are uniformly labeled and oriented. The black edges represent the $ \vec{e}_{i}$'s and the orange edges are their redirections $\vec{r}_{i}$'s. The root of $P_{18}$ is here the vertex $6$. \label{fig:AFPT}}
 \end{center}
 \end{figure}
 
 In the above construction, for $j >i$ let us describe the event on which the car number $j>i$ needs to go through the redirection $ \vec{r}_{i}$ of the edge $ \vec{e}_{i}$ to find its parking spot. To do this we introduce $i=\ell_1 , \dots , \ell_{s} =j$ the labels of the edges on the path between the edge $ \vec{e}_{i}$  and $ \vec{e}_{j}$ (both included) in $T_{N}^{u}$, see Figure \ref{fig:height}. We consider the record times $1= \tau_{1} < \tau_{2}< \cdots  < \tau_{k}$ associated with the strict ascending records $i=b_{1}<b_{2} < \cdots < b_{k-1}< b_{k}$ of the process $ \ell_{1}, \ell_{2}, \dots , \ell_{s-1}, \ell_{s}$. That is we put $\tau_{1} =1$, set  $b_{1}= \ell_{\tau_{1}}=i$ and recursively define
$$ b_{i+1} = \ell_{\tau_{i+1}} \mbox{ where } \tau_{i+1} = \inf\{ \tau_{i} < t \leq s : \ell_{t} > b_{i}\}.$$
\begin{lemma}\label{lem:cnsheight} \label{lem:passthrough} With the above notation, the $j$th car goes through the redirection $ \vec{r}_{i}$ of the edge $ \vec{e}_{i}$ in $P_{N}$ if and only if $b_{k}=j$ (that is $j$ is the maximal record) and all edges with labels  $ b_{1}, b_{2}, \dots , b_{k-1}$ on the path from $ \vec{e}_{i}$ to $ \vec{e}_{j}$ in $T_{N}^{u}$ point  away from $ \vec{e}_{j}$, and furthermore $\vec{e}_{j}$ points away from $ \vec{e}_{i}$.
 \end{lemma}

\begin{figure}[!h]
 \begin{center}
 \includegraphics[width=12cm]{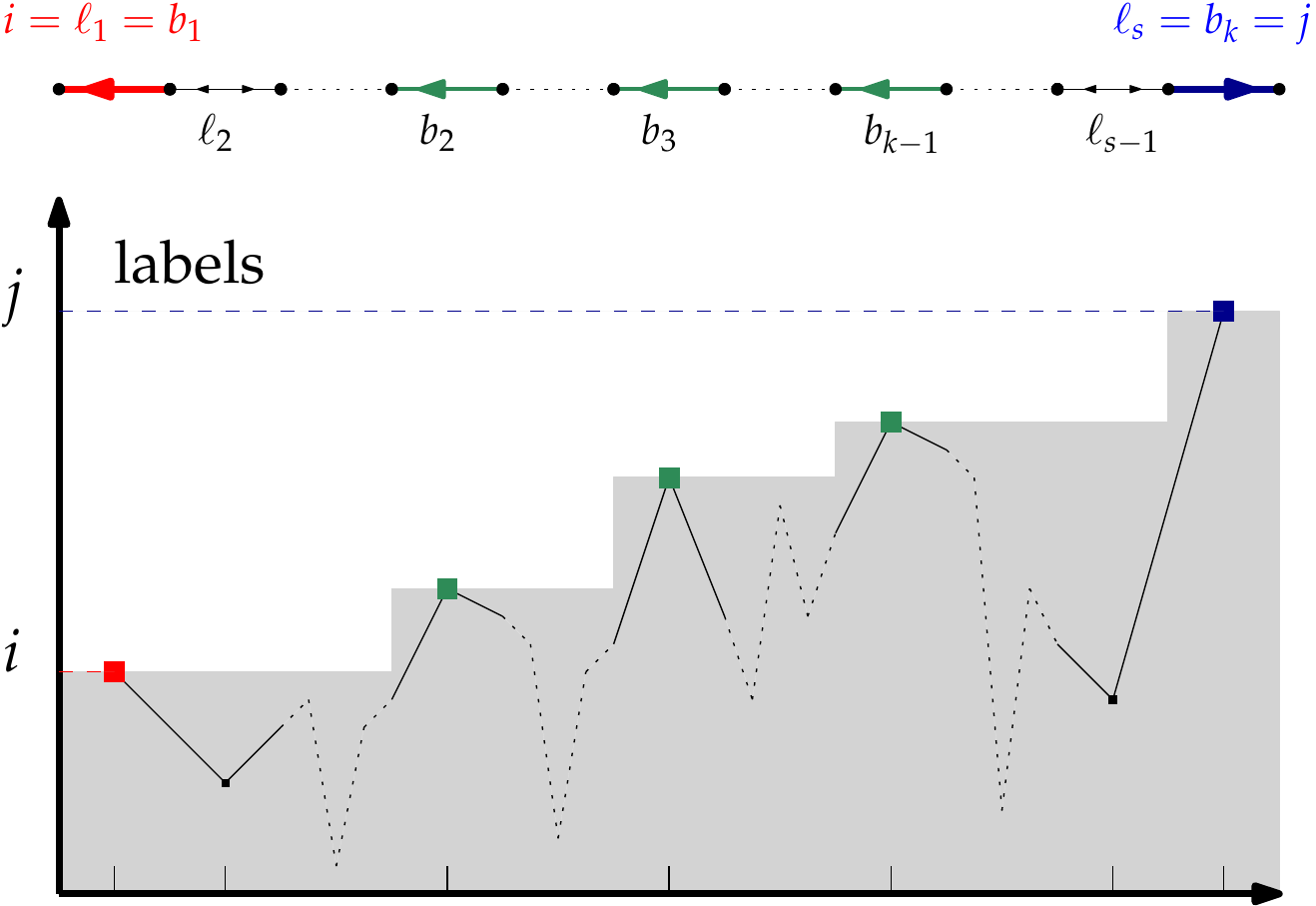}
 \caption{Labeling of the edges on the branch from $ \vec{e}_{i}$ to $ \vec{e}_j$ in $T_{N}^{u}$. The $j$th car goes through the redirection of $\vec{e}_{i}$ in $P_{N}$ if and only if $j$ is a record on the branch from $ \vec{e}_{i}$ to $\vec{e}_{j}$ and the edges of records are oriented accordingly. \label{fig:height}}
 \end{center}
 \end{figure}
 
\begin{proof} Let us show the proposition by induction on $N$, see Figure \ref{fig:lemma-furet}. Fix $i <j$ and consider the path going from $ \vec{e}_{i}$ to $\vec{e}_{j}$ (included) in $T_{N}^{u}$. We write $v_{i}$ and $v_{j}$ for the vertices at the extremities of this branch, the point $v_{i}$ being closer to $\vec{e}_{i}$ and $v_{j}$ to $\vec{e}_{j}$. Imagine that we build the tree $P_{N}$ by re-orienting the edges $ \vec{e}_{1}, \dots , \vec{e}_{n}$ one after the other and first contemplate the situation when we examine the edge $ \vec{e}_{b_{k}}$. This edge connects two nearly parked trees (which may be reduced to single free spots) made of some of the edges $\vec{r}_{1}, \dots , \vec{r}_{b_{k}-1}$ that we already re-oriented. We denote those nearly parked trees ${ \mathfrak{P}}_{i}$ and $ { \mathfrak{P}}_{j}$ where $ { \mathfrak{P}}_{i}$ contains $v_{i}$ and $ { \mathfrak{P}}_{j}$ contains $v_{j}$. Clearly if the edge $ \vec{e}_{ b_{k}}$ separates $\vec{e}_{i}$ from $ \vec{e}_{j}$ then the $j$th car is already parked in $ {\mathfrak{P}}_{j}$ and did not go through $   \vec{r}_{i}$. So the interesting case is when $b_{k} =j$. In this case we must further have that $ \vec{e}_{b_{k}}$ is oriented from $  { \mathfrak{P}}_{i}$ to $ {\mathfrak{P}}_{j}$ for otherwise the $j$th car arrives on $ {\mathfrak{P}}_{j}$ and parks at its root without going through $ \vec{r}_{i}$. Let us now go backward in time and examine the situation when we constructed $\vec{r}_{1}, \dots , \vec{r}_{b_{k-1}-1}$ and were about to re-orient $ \vec{e}_{b_{k-1}}$. Similarly as above, at that time, the edge $\vec{e}_{b_{k-1}}$ connects two nearly parked trees $ \tilde{\mathfrak{P}}_{i}$ containing the vertex $v_{i}$ and another one $ \tilde{\mathfrak{P}}'$ which may not contain $v_{j}$. A reasoning similar to the one above shows that it is necessary for the $j$th car to go through $ \vec{r}_{i}$ that $ \vec{e}_{b_{k-1}}$ points towards $ \tilde{\mathfrak{P}}_{i}$. In this case, when the $j$th car arrives, it lands on some vertex of $ \tilde{\mathfrak{P}}'$, follows the oriented edges to its root and then go through $ \vec{r}_{b_{k-1}}$ to reach the target of $ \vec{e}_{b_{k-1}}$ in $ \tilde{\mathfrak{P}}_{i}$. Asking whether that car goes through $ \vec{r}_{i}$ is equivalent to asking whether a car arrival corresponding to the edge $\overleftarrow{e}_{b_{k-1}}$ (with reversed orientation) would go through $  \vec{r}_{i}$ inside $ \tilde{ \mathfrak{P}}_{i}$. Since the size of the system strictly decreased, we can then apply the induction hypothesis and deduce the condition presented in the lemma.
 \begin{figure}[!h]
 \begin{center}
 \includegraphics[width=16cm]{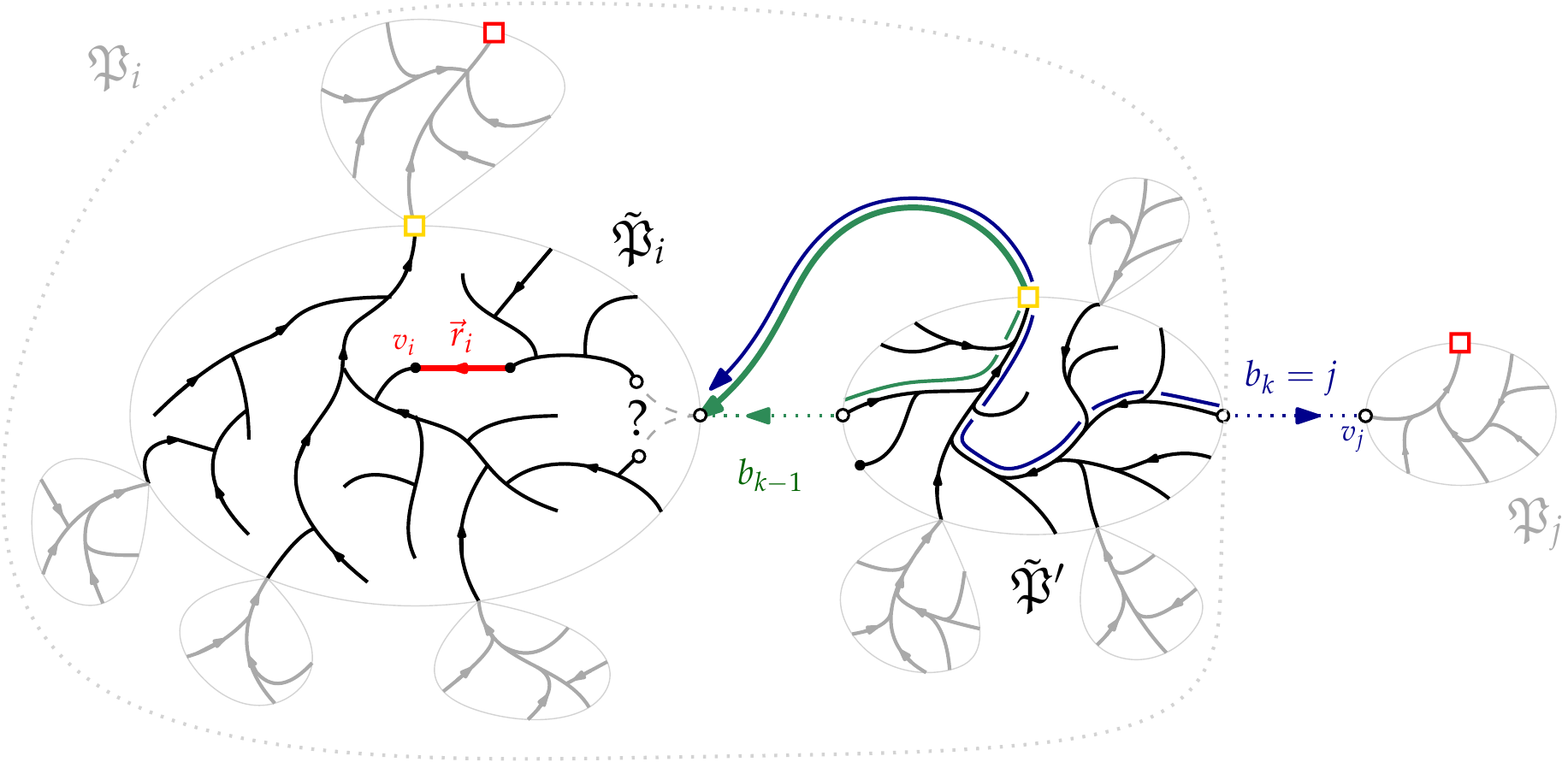}
 \caption{\label{fig:lemma-furet} Illustration of the proof of Lemma \ref{lem:passthrough}. The edges $\vec{r}_{1}, \dots , \vec{r}_{b_{k-1}-1}$ are in black. The edge $\vec{r}_{ b_{k-1}}$ is in green as well as the journey of the $b_{k-1}$th car. The edges $\vec{r}_{b_{k-1}+1}, \dots , \vec{r}_{b_{k}-1}$ are in gray. The root of the trees $\tilde{ \mathfrak{P}}_{i}$ and $\tilde{ \mathfrak{P}}'$ are in yellow, and those of $ \mathfrak{P}_{i}$ and $ \mathfrak{P}_{j}$ are in red. The car of index $b_{k}=j$ goes through $ \vec{r}_{i}$ if and only if $\vec{e}_{b_{k}}$ and $ \vec{e}_{b_{k-1}}$ respectively points away from and towards $v_{i}$ and if a car landing on the target of $ \vec{e}_{b_{k-1}}$ would go through $ \vec{r}_{i}$ inside $ \tilde{\mathfrak{P}}_{i}$. The beginning of the journey of the $j$th car is in blue.}
 \end{center}
 \end{figure} \end{proof}

\subsection{Typical height}

We can now prove Proposition \ref{prop:height}.

\begin{proof}[Proof of Proposition \ref{prop:height}.] We suppose that $P_{N}$ is constructed from $T_{N}^{u}$ as in the preceding section. Independently of $T_{N}^{u}$,  we let $V \in \{1,\dots , N\}$ be a uniform point and $I \in \{1,2,\dots , N-1\}$ be an independent (label of an) oriented edge. Then we have 
$$\frac{1}{N} \mathbb{E}\left[\sum_{x \in P_{N}} \mathrm{d}_{  \mathrm{gr}}^{P_{N}}( \rho, x)\right] = (N-1) \mathbb{P} \left (\vec{r}_I \mbox{ contributes to the height of }V\mbox{ in } P_{N}\right).$$
Notice that the number $ H_{N}$ of edges on the path between $V$ and $\vec{e}_I$ (included) in $T_{N}^{u}$ has the same law as the length of the branch of $T_{N}^{u}$ between two uniform distinct points.  By \cite[Theorem 7.8 p76]{moon1970counting} we have for $1 \leq h \leq N-1$,
$$ \mathbb{P} (H_{N} = h) =  \frac{h+1}{N-1} \frac{N (N-1)\cdots (N-h)}{N^{h+1}},$$ 
To see whether the $I$th car will contribute to the height of $V$ in $P_{N}$ we can graft an imaginary oriented edge $\vec{e}_{N+1}$ on $V$ oriented away from $\vec{e}_{I}$ and apply Lemma \ref{lem:passthrough} with $j=N+1$ to ask whether that fictive $N+1$th car would go through $ \vec{e}_{I}$. We deduce that the necessary and sufficient condition is that all oriented edges on the path from $ \vec{e}_{I}$ to $V$ in $T_{N}^{u}$ corresponding to strict ascending record for their labels are oriented away from $V$. Since conditionally on $H_{N}$, the order preserving relabeling  of the edges on the branch is uniform, we deduce that the number of such records is equal in law to the number of cycles with disjoint support of a uniform permutation $\sigma_{H_{N}}$ over $H_{N}$ elements, see \cite[Example II.16 p140]{Flajolet:analytic}. By \cite[Example III.2 p155]{Flajolet:analytic} we have
 \begin{eqnarray} \label{eq:numbercycles} \mathbb{E} \left[ \left(\frac{1}{2}\right)^{\# \mathrm{Cycles}( \sigma_{h})} \right] = \prod_{j=0}^{h-1} \frac{1/2 + j}{j+1} = \frac{(h+1)}{(h+1)!} \left( \frac{1}{2} \right)_{h}.   \end{eqnarray}
Combining these lines, we obtain 
\begin{eqnarray*} \frac{1}{N} \mathbb{E}\left[\sum_{x \in P_{N}} \mathrm{d}^{ P_{N}}_{ \mathrm{gr}}( \rho, x)\right] &=& (N-1) \sum_{h=1}^{N-1}  \mathbb{P} (H_{N} = h) \mathbb{E} \left[ \left(\frac{1}{2}\right)^{\#  \mathrm{Cycles}(\sigma_{h})}\right] \\
&=& \sum_{h=1}^{N-1}  \frac{N (N-1) \cdots  (N-h)}{N^{h+1}} \cdot \frac{(h+1)}{h!} \left( \frac{1}{2} \right)_{h} \end{eqnarray*}
and we get the desired result. The asymptotic of this sum is done by standard estimates: the main contribution appears for $h \approx x\sqrt{N}$ with $x \geq 0$ for which the terms are of order $$ \binom{2h}{h} \frac{h}{4^h} \prod_{i=1}^{h} \left( 1 - \frac{i}{N}\right) \sim   N^{1/4} \cdot  \sqrt{ \frac{x}{\pi}} \mathrm{e}^{- \frac{x^{2}}{2}},$$ and a series-integral comparison yields the asymptotic $ N^{3/4} \int_{0}^{\infty} \mathrm{d}x \sqrt{ \frac{x}{\pi}} \mathrm{e}^{- \frac{x^{2}}{2}} = N^{3/4}  \frac{ \Gamma(3/4)}{2^{1/4} \sqrt{\pi}}$.
\end{proof}

\subsubsection{Total traveled distance}

\begin{proof}[Proof Proposition \ref{prop:totaldistance}] The proof is similar to that of Proposition \ref{prop:height}. If $I,J$ are two distinct uniform edge labels of $\{1,2, \dots, N-1\}$ then we have 
$$\mathbb{E}\left[ \mathrm{Total\ Distance \ Traveled  \ in\  }P_{N}\right] = (N-1)(N-2) \mathbb{P}( J\mbox{th car goes through }I\mbox{th edge}),$$ where $(N-1)(N-2)$ is the number of distinct pairs of edges. Since choosing $2$ different edges in a tree is the same as choosing two vertices at distance at least $2$, by \cite[Theorem 7.8]{moon1970counting} again, the length $\tilde{H}_{N}$ of the branch from $\vec{e}_{I}$ to $\vec{e}_{J}$ in $T_{N}^{u}$ is distributed as 
$$ \mathbb{P}( \tilde{H}_{N} =h) =  \frac{N}{N-2} \cdot  \frac{h+1}{N-1} \frac{N(N-1) \cdots (N-h)}{N^{h+1}}.$$ 
Since conditionally on $\tilde{H}_{N}$ the increasing reordering of the labels on the branch is uniform, by Lemma \ref{lem:passthrough} and using \eqref{eq:numbercycles} again, conditionally on $\tilde{H}_{N}=h$, the probability that the $J$th car passes through $ \vec{r}_{I}$ is equal to $ \frac{1}{h}$ (the probability that $J$ is a record) time $$ \frac{1}{2}\frac{h}{h!} \left( \frac{1}{2} \right)_{h-1},$$ where the additional $1/2$ comes from requiring the good orientation for $\vec{e}_{J}$.
 Combining those lines gives the desired result. The asymptotic of the sum is done as in the preceding proof and is left to the reader. \end{proof}

\part{Scaling limits}
This part is devoted to scaling limits in the critical regime $ m = \frac{n}{2} + O( n^{2/3})$. We first use known results on the (standard augmented) multiplicative coalescent to show the convergence of the component sizes in the frozen Erd{\H{o}}s--R\'enyi process (Theorem \ref{thm:FMC}). Thanks to our coupling construction (Section \ref{sec:couplingtrees}) these translate into results on the parking process on Cayley trees (Theorem \ref{thm:composantes}).  We then take another point of view on the limiting processes, and in particular on the total mass of the frozen components, using the Markovian properties of $F(n, \cdot)$. On the way we describe the scaling limit of component sizes of a critical random forest using conditioned stable L\'evy processes thus giving an alternative (and shorter) approach to the results of Martin \& Yeo \cite{martin2018critical}. Recall convention \eqref{eq:notationcriticalwindow}.

\section{The frozen multiplicative coalescent} \label{sec:FMC}
In this section we establish a scaling limit for the component sizes in the frozen Erd{\H{o}}s--R\'enyi $ \mathrm{F}_n( \cdot)$ in the critical window. This is deduced from known results on the multiplicative coalescent but requires some care because the inclusion of the frozen process in the Erd{\H{o}}s--R\'enyi process is ``non-monotonous''. We use cutoffs and controls which are similar to those of \cite{rossignol2021scaling}. 

For $q\geq 1$, we let $$\ell^{q}_{\downarrow} := \left\{ (x_{1}, x_{2}, \cdots ) : x_{1} \geq x_{2} \geq \dots \geq 0 \mbox{ and } \sum_{i \geq 1} x_i^q < \infty \right\},$$ be the space of non-increasing $\ell^{q}$ sequences. It has a natural norm inherited from the $\ell^{q}$ space and is a closed subspace of $\ell^{q}$.  In the following we denote by $ \mathcal{E} = \ell^{1}_{\downarrow} \times \ell^{2}_{\downarrow}$ which is a Polish space when endowed with the distance $ \mathrm{d}_{ \mathcal{E}}$ defined by $$ \mathrm{d}_{ \mathcal{E}}\left( (  \mathbf{x}, \mathbf{y}) , ( \mathbf{x}', \mathbf{y}')\right) = \sum_{i \geq 1} |x_{i}-x'_{i}| + \left( \sum_{i \geq 1} |y_{i} - y'_{i}|^{2}\right)^{1/2}.$$
An element $ (\mathbf{x}, \mathbf{y})$ of $ \mathcal{E}$ will be interpreted as the masses of the particles of a system, the particles whose masses are $ x_{1}, x_{2}, \dots$ will be called the \emph{frozen} or \emph{blue} particles and their total mass is finite, whereas the particles whose masses are $y_{1}, y_{2}, \dots$ will be called the standard or white particles and their total mass may be infinite. With this interpretation in mind, and in accordance with the notation for graphs, we put for $ \mathbf{z}=( \mathbf{x}, \mathbf{y}) \in \mathcal{E}$ 
$$ [\mathbf{z}]_{\bluecirc} = \mathbf{x} \quad \mbox{ and } \quad [\mathbf{z}]_{\circ} = \mathbf{y}, \quad \mbox{ and }\quad \| \mathbf{z}\|_{ \bluecirc} = \sum_{i\geq 1}x_{i}.$$  Recall from the Introduction the definition of the frozen Erd{\H{o}}s--R\'enyi random graph $(F(n,m) : m \geq 0) $ and its continuous time counterpart $ (\mathrm{F}_{n}(\lambda) : \lambda \in \mathbb{R})$. We shall denote by 
$$ \mathbb{F}_{n}(\lambda) \in \mathcal{E}$$  the decreasing sequence of the sizes of the frozen blue components (completed with zeros) of $ \mathrm{F}_{n}(\lambda)$ renormalized by $n^{-2/3}$, followed by the decreasing sequence of sizes of the white components also renormalized by $n^{-2/3}$ (also completed with zeros). If $I \subset \mathbb{R}$ is an interval and $ \mathrm{Pol}$ some Polish space, we denote by $ \mathbb{C} \mathrm{adlag}(I, \mathrm{Pol})$ the set functions $ f: I \to \mathrm{Pol}$ which are right-continuous with left limits at every point, endowed with the Skorokhod $J_{1}$ topology on every compact interval of $I$.  The main theorem of this section is:

\begin{theorem}[Scaling limit  for component sizes of the frozen Erd{\H{o}}s--R\'enyi] \label{thm:FMC}  We have the following convergence in distribution for the Skorokhod topology on $ \mathbb{C} \mathrm{adlag}( \mathbb{R} , \mathcal{E})$
 \begin{eqnarray} \label{eq:convdiscreteFMC} \left( \mathbb{F}_{n}\left(\lambda \right) \right)_{\lambda \in \mathbb{R}} \xrightarrow[n\to\infty]{(d)}  \left( \FM( \lambda)\right)_{\lambda \in 	\mathbb{R}}.  \end{eqnarray}
 The process $ \FM$ is called the \emph{frozen multiplicative coalescent}.
\end{theorem}

\begin{remark} 
It will follow from the proof that  $ \FM$  can be built from the (augmented) multiplicative coalescent of Aldous \cite{aldous1997brownian} by taking an appropriate cutoff procedure. %It may also be possible to prove that $ \FM$ is a nice Markov process (e.g.~with  a Feller type property) but we refrain from doing so in this paper.
\end{remark}

The proof of Theorem \ref{thm:FMC} occupies the rest of this section. To fix ideas, we shall restrict to a fixed compact time interval and prove the convergence \eqref{eq:convdiscreteFMC} for $ \lambda \in [-1,0]$. The general case is \emph{mutatis mutandis} the same. We first prove a convergence in distribution using the (weaker) supremum norm 
  \begin{eqnarray} \label{eq:defsupnorm} \mathrm{d}_{\sup}\left( (  \mathbf{x}, \mathbf{y}) , ( \mathbf{x}', \mathbf{y}')\right) = \sup_{i \geq 1} |x_{i}-x'_{i}| + \sup_{i \geq 1} |y_{i} - y'_{i}|, \end{eqnarray}
and then lift it for the $ \mathrm{d}_{ \mathcal{E}}$ distance by proving the required tightness (see Proposition \ref{lem:tightness}). Recall from the Introduction that the dynamics between standard white particles in the frozen multiplicative coalescent is the same as in the multiplicative coalescent, but the interaction between standard and frozen particles is different. Our first difficulty in this program is that the frozen part is always present, i.e.\ $ \|\FM(\lambda)\|_{\bluecirc} >0$ for all $\lambda \in \mathbb{R}$. In the next section, we shall prove however that we can neglect the effect of the frozen part that  is ``old enough'' in the $\ell^{1}$-sense in the time-window $\lambda \in [-1,0]$. Then, we approximate the remaining frozen process by a process on finitely many ``particles'' for which the convergence in distribution is obvious, see Figure \ref{fig:cutoff}. These cutoff procedures are of course reminiscent of the original construction of Aldous \cite{aldous1997brownian} and of the more recent work of Rossignol on dynamical percolation \cite{rossignol2021scaling}. We first present deterministically the two cutoff procedures in the following section and then prove the necessary estimates  using the relations with the  multiplicative coalescent process (Proposition \ref{prop:controls}). 

\begin{figure}[h!]
 \begin{center}
 \includegraphics[width=12.3cm]{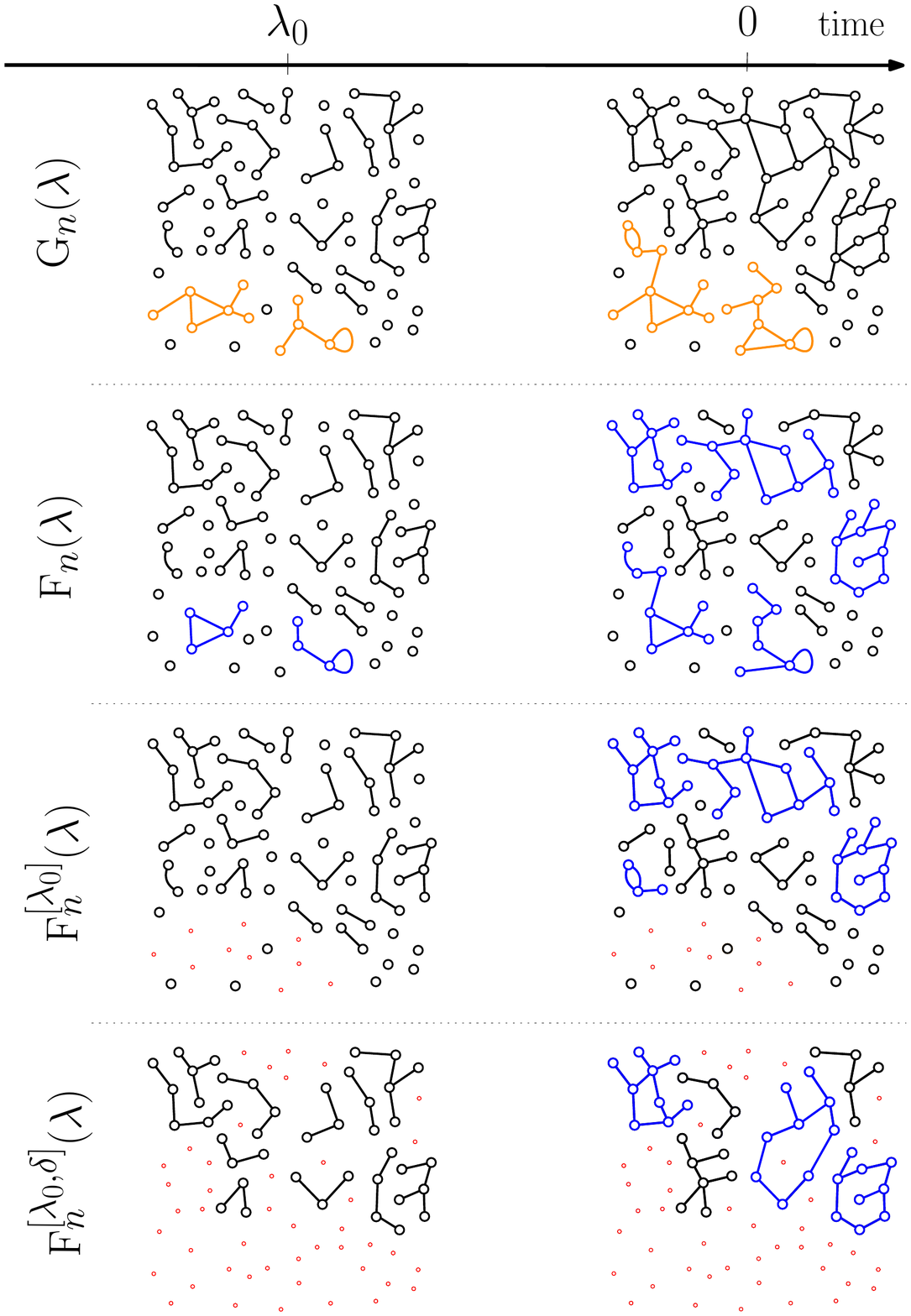}
 \caption{\label{fig:cutoff}The two cutoff procedures: Comparisons of the processes $ \mathrm{F}_{n}$ with $ \mathrm{F}_{n}^{[\lambda_{0}]}$ and $ \mathrm{F}_{n}^{[\lambda_{0}, \delta]}$. The orange part on the top right corner is $ \mathrm{O}_n( \lambda_{0})$. Third line: The process $ \mathrm{F}_{n}^{[\lambda_{0}]}$ is obtained by starting from $G(n, \left\lfloor \frac{n}{2} + \frac{\lambda_{0}}{2} n^{2/3}\right\rfloor)$, removing the components having a surplus, and then applying the rules of the construction of the frozen Erd{\H{o}}s--R\'enyi with the remaining edges. Fourth line: The process $ \mathrm{F}_{n}^{[\lambda_{0}, \delta]}$ is further obtained by restricting to components of size at least $ \delta n^{{2/3}}$ at time $\lambda_{0}$. }
 \end{center}
 \end{figure}

\subsection{Getting rid of old cycles} \label{sec:gettingridoldcycles}
We first start with a control that enables us to get rid of the frozen part that is ``old enough''. For $\lambda_{0}<-1$, let us denote by  \begin{eqnarray} \label{def:Onlambda} \mathrm{O}_n( \lambda_{0})  \end{eqnarray} the union  of the components of $ \mathrm{G}_{n}(0)=G(n, \left\lfloor\frac{n}{2}\right\rfloor)$ which have a surplus that appeared before time $\lambda_{0}$ i.e.\ before $ \left\lfloor\frac{n}{2} + \frac{\lambda_{0} }{2} n^{2/3}\right\rfloor$ edges have been added, see Figure \ref{fig:cutoff}. We will see in Proposition~\ref{prop:controls} that $n^{-2/3} \cdot \|\mathrm{O}_n(\lambda_{0})\|_{\bullet}$ is small  provided that $\lambda_{0}$ is negative enough. We shall compare our usual frozen process $ \mathrm{F}_{n}( \lambda)$ for $\lambda \in [\lambda_{0},0]$ to the process $$ \mathrm{F}^{[\lambda_{0}]}_{n}( \lambda) : \lambda \in [\lambda_{0},0]$$ which is started from time $\lambda_{0}$ without any frozen part and obtained as follows. Let us consider the graph $ \mathrm{G}_n( \lambda_{0})$ and remove from it the components with surplus to get its forest part $[\mathrm{G}_n( \lambda_{0})]_{ \mathrm{tree}}$. We then let the remaining edges arrive as in the $G(n,m)$ process and only examine those  that connect points of $  [\mathrm{G}_n( \lambda_{0})]_{ \mathrm{tree}}$ and apply the rule of the frozen process (Figure \ref{fig:transitions}) to get $\mathrm{F}^{[\lambda_{0}]}_{n}( \lambda)$ for $\lambda \in [\lambda_{0},0]$, see Figure \ref{fig:cutoff} (third line).

Of course, the process $ \mathrm{F}_{n}$ and $ \mathrm{F}_{n}^{[\lambda_{0}]}$ are not identical, but it should be clear that their possible differences are only located on $ \mathrm{O}_n( \lambda_{0})$. Since the supremum distance $ \mathrm{d}_{\sup}$ defined in \eqref{eq:defsupnorm} decreases under the non-increasing re-arrangement of both parts, it follows from the above remark  that for all  $\lambda \in [\lambda_{0},0]$
 \begin{eqnarray}  \mathrm{d}_{ \sup}\left(  \mathbb{F}_{n}( \lambda) ;  \mathbb{F}_{n}^{[\lambda_{0}]}(  \lambda)\right) &\leq& 
n^{-2/3}\cdot \| \mathrm{O}_n( \lambda_{0})\|_{\bullet},   \label{eq:firstcutoff}\end{eqnarray}
 where $  \mathbb{F}_{n}^{[\lambda_{0}]}(\lambda)$ is the pair of renormalized non-increasing sizes of frozen components followed by the renormalized sizes of the standard components of $ \mathrm{F}_{n}^{[\lambda_{0}]}(\lambda)$.

\subsection{Approximation by the $\eta$-skeleton} \label{sec:skeleton}
Our goal now is to approximate the process $ (\mathrm{F}_{n}^{[\lambda_{0}]}( \lambda)  : \lambda \in [\lambda_{0},0])$ by a process with a number of particles that stays bounded as $n \to \infty$. More precisely, in the following, we shall call the components of $  [\mathrm{G}_n( \lambda_{0})]_{ \mathrm{tree}}=  \mathrm{F}^{[\lambda_{0}]}_{n}(\lambda_{0})$ the ``\textbf{specks}'' and say that their \textbf{masses} are given by their number of vertices renormalized by $n^{-2/3}$. \medskip

For every $ \delta >0$, we denote by $(\mathrm{F}_{n}^{[\lambda_{0}, \delta]}( \lambda)  : \lambda \in [\lambda_{0},0])$ the frozen process started from $\mathrm{F}_{n}^{[\lambda_{0}]}( \lambda_{0})$ and obtained by only examining those edges between or inside specks of mass at least $ \delta$, see Figure \ref{fig:cutoff} fourth line. The fact that we discarded  some edges may affect the colors and the connections of the vertices due to the non-monotonicity of the frozen dynamics. However, we shall see that if $\delta$ is small enough the dynamics are coherent (Lemma \ref{lem:coherencecutoff}) on  a large part of the graph. Following the above notation rule, we write $  \mathrm{G}_{n}^{[\lambda_{0}]}(\cdot)$ the Erd{\H{o}}s--R\'enyi process started from $ [ \mathrm{G}_n( \lambda_{0})]_{ \mathrm{tree}}$  at time $\lambda_{0}$ and keeping only the remaining edges that belong to $[ \mathrm{G}_n( \lambda_{0})]_{ \mathrm{tree}}$. We now establish deterministic inclusions between all these processes.

\paragraph{The $\eta$-skeleton.}
Consider the graph $\mathrm{G}^{[\lambda_{0}]}_n( 0)$. This graph has a certain number of non-trivial cycles (including self-loops and multiple edges) involving certain specks (recall that the specks are the initial components of $\mathrm{F}_{n}^{[\lambda_{0}]}( \lambda_{0}) = \mathrm{G}^{[\lambda_{0}]}_n( \lambda_{0}) = [ \mathrm{G}_n( \lambda_{0})]_{ \mathrm{tree}}$). We fix $ \eta>0$ and define the 
\begin{center}
$ \eta$-\emph{skeleton}
\end{center}
as the set of all specks  in $\mathrm{G}^{[\lambda_{0}]}_n( 0)$ that belong to a non-backtracking path whose extremities are either a speck of a cycle of $\mathrm{G}^{[\lambda_{0}]}_n( 0)$ or a speck of mass at least $ \eta>0$, see Figure \ref{fig:skeleton}. In particular, all specks on (non-backtracking) paths in $\mathrm{G}^{[\lambda_{0}]}_n( 0)$ between specks of the $ \eta$-skeleton actually belong to the $  \eta$-skeleton\footnote{in other words, two vertices of the $\eta$-skeleton which are connected in $\mathrm{G}^{[\lambda_{0}]}_n( 0)$ are connected within the $\eta$-skeleton and all non-trivial cycles of $\mathrm{G}^{[\lambda_{0}]}_n( 0)$ are inside the $\eta$-skeleton}. We then denote by 
 \begin{eqnarray} \label{eq:defeta1} \gamma = \gamma_n( \lambda_{0}, \eta)  \end{eqnarray} the minimal weight of a speck on the $\eta$-skeleton.

\begin{figure}[h!]
 \begin{center}
\includegraphics[width=15cm]{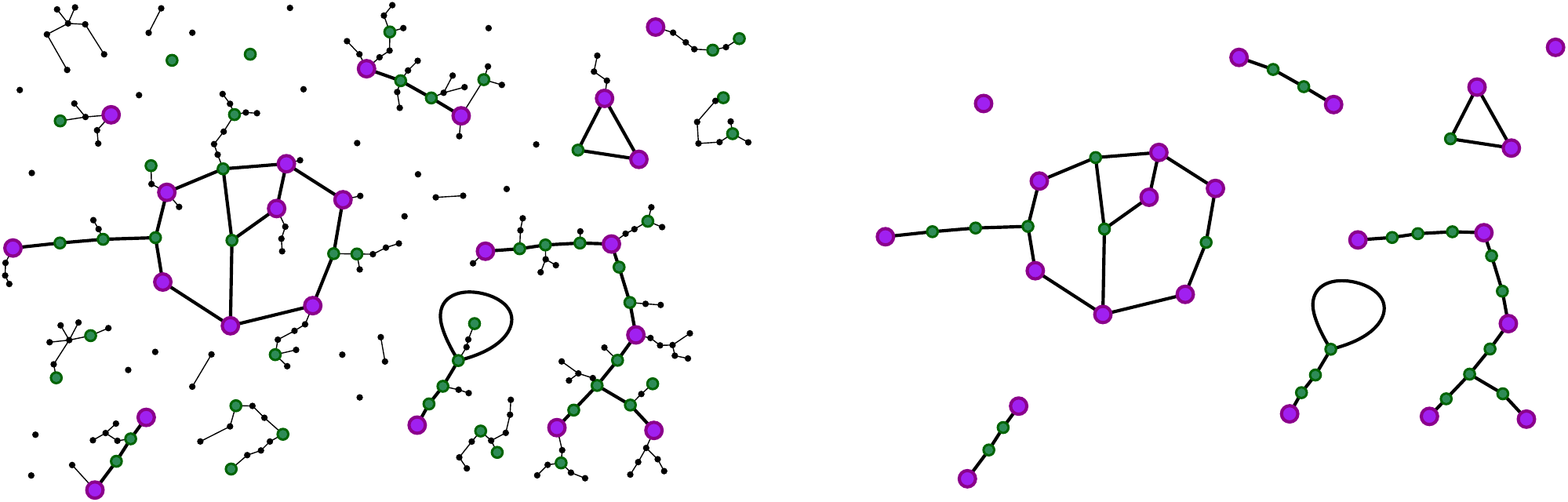}
 \caption{ \label{fig:skeleton} On the left the graph $ \mathrm{G}_{n}^{[\lambda_{0}]}(0)$ and its $\eta$-skeleton on the right. The points represent the specks, i.e.\ the components of $ [\mathrm{G}_{n}(\lambda_{0})]_{ \mathrm{tree}}$. Specks of mass larger than $\eta$ (i.e.\ with more than $\geq \eta n^{2/3}$ vertices) are displayed in purple, those of mass in-between $\gamma$ and $ \eta$ are in green, and those of mass smaller than $ \gamma$ are in black.}
 \end{center}
 \end{figure}
 
The key is to show that as soon as  $ \delta \leq \gamma$ the induced frozen Erd{\H{o}}s--R\'enyi process $ \mathrm{F}_n^{[\lambda_0, \delta]}$ is constant on the $ \eta$-skeleton. More precisely:

\begin{lemma}  \label{lem:coherencecutoff} With the above notation for  any $\lambda \in [\lambda_{0},0]$  and any $\delta \in [0, \gamma)$, the edges between vertices of the $ \eta$-skeleton and the color of these vertices are the same in  $ \mathrm{F}_{n}^{[\lambda_0, \delta]}( \lambda)$ and in $ \mathrm{F}_{n}^{[\lambda_{0}]}(\lambda)$.
\end{lemma}
\proof We prove the lemma by induction, adding the edges one by one. Fix $ \delta < \gamma$ and consider the status of the edges and of the vertices of the $ \eta$-skeleton in $\mathrm{F}_{n}^{[\lambda_0, \delta]}( \lambda)$ and $\mathrm{F}_{n}^{[\lambda_0]}( \lambda)$. Clearly, at time $\lambda=\lambda_{0}$ they match up. By induction, suppose that at some time $\lambda \in [\lambda_{0},0]$ we examine the status of an edge $E_{i}$ between two specks  and that at time $\lambda^{-}$ the induced graph on the $ \eta$-skeleton is the same in $\mathrm{F}_{n}^{[\lambda_0, \delta]}( \lambda^{-})$ and in $\mathrm{F}_{n}^{[\lambda_0]}( \lambda^{-})$:
\begin{itemize}
\item Suppose first that $E_{i}$ is an edge of the $ \eta$-skeleton.  In particular, both  endpoints are located on a speck of  mass at least $ \gamma$ and this edge gets examined  both in $\mathrm{F}_{n}^{[\lambda_0, \delta]}( \lambda^{-})$ and in $\mathrm{F}_{n}^{[\lambda_0]}( \lambda^{-})$ (since $\delta \leq \gamma$). Since the colors and the connections of the vertices in the $ \eta$-skeleton are the same in both processes at time $\lambda^{-}$, applying the rules of the construction of the frozen process yields  the same transformations for the vertices, colors and edges of the $ \eta$-skeleton in both cases. 
\item If $E_{i}$ is not an edge of the $ \eta$-skeleton, one may think that we do not care whether we add it or not in the frozen processes: Indeed, we saw above that the non-backtracking paths between specks of the $\eta$-skeleton stay within the skeleton so that adding that edge does not change the connections between specks of the $\eta$-skeleton. However its addition might change the color of some vertices of the $ \eta$-skeleton. 
 But this can only happen if this edge creates a cycle or relates a white component carrying a vertex of the $ \eta$-skeleton to a frozen blue component. Since the frozen blue components necessarily contain a cycle one can check that $E_{i}$  must belong to the $ \eta$-skeleton and so we are back to the previous item. 
\end{itemize}
\endproof

Let us use the above lemma and more generally the relations between $ \mathrm{G}_{n}(\cdot), \mathrm{G}_{n}^{[\lambda_{0}]}(\cdot)$ and $\mathrm{F}_{n}(\cdot),  \mathrm{F}_{n}^{[\lambda_{0}]}(\cdot), \mathrm{F}_{n}^{[\lambda_{0}, \delta]}(\cdot)$ to derive bounds on the $ \mathrm{d}_{ \mathcal{E}}$ distance.  
For a given vertex $x~\in~\{1,2, \dots , n\}$, the cluster of $x$ in $ \mathrm{F}_{n}^{[\lambda_{0}, \delta]}(\lambda)$ and in $\mathrm{F}_{n}^{[\lambda_{0}]}(\lambda)$ may differ\footnote{we do not have deterministic inclusion of one inside the other},  but when $\delta < \gamma$, by Lemma \ref{lem:coherencecutoff} they contain the same specks of the $\eta$-skeleton, so their difference is in particular supported by vertices of specks of mass $\leq \eta$ and belonging to the component of $x$ in $ \mathrm{G}^{[\lambda_{0}]}_{n}(0)$. If we denote by $ \Delta_{n}(\lambda_{0}, \eta)$ the maximal difference (number of vertices) between a component in $ \mathrm{G}^{[\lambda_{0}]}_{n}(0)$ and its subcomponent made of (vertices of) specks of mass $\geq \eta$ then we have for $\delta \leq \gamma$ and $\lambda \in [\lambda_{0},0]$ with an obvious notation
 \begin{eqnarray}  &&\mathrm{d}_{ \sup}\left( \mathbb{F}_{n}^{[\lambda_{0}, \delta]}( \lambda) ; \mathbb{F}_{n}^{[\lambda_{0}]}( \lambda) \right)\  \leq n^{-2/3}\cdot \Delta_{n}(\lambda_{0}, \eta).   \label{eq:secondcutoff}
\end{eqnarray}

\subsection{Estimates via the augmented multiplicative coalescent}
\label{sec:controls}
Recall from the previous section the definition of the variables $ O_{n}(\lambda_{0}), \gamma_{n}(\lambda_{0},\eta)$ and $\Delta_{n}(\lambda_{0}, \eta)$. We provide the necessary (asymptotic) controls on those variables to apply the above cutoffs. These are derived using known estimates on the (augmented) multiplicative coalescent \cite{aldous1997brownian,bhamidi2014augmented,rossignol2021scaling}.

\begin{proposition} \label{prop:controls} For any $\varepsilon >0$ one can find 
$$\lambda_{0} < -1, \quad\eta \in (0,1),\quad \delta \in (0,\eta), 
$$ and $n_{0} \geq 1$ so that for all $n \geq n_{0}$ with probability at least $ 1-\varepsilon$ we have 
 \begin{eqnarray} n^{-2/3} \cdot \| \mathrm{O}_n( \lambda_{0})\|_{\bullet} &\leq& \varepsilon, \\
n^{-2/3} \cdot \Delta_n( \lambda_{0}, \eta) &\leq& \varepsilon,\\
\gamma_n ( \lambda_{0}, \eta) &\geq& \delta.
\end{eqnarray}
\end{proposition}

\begin{proof} The proof will follow from the convergence of the component sizes and surplus of $G(n,m)$ towards the augmented multiplicative coalescent and some of its basic properties. To help the reader, let us sketch in which order the variables will be chosen
$$  \varepsilon \longrightarrow \begin{array}{c} \lambda_{0}<0\\
 \mathrm{ so \ that \ w.h.p.}\\
  n^{{-2/3}} \cdot \| \mathrm{O}_{n}(\lambda_{0})\|_{\bullet} \leq \varepsilon  \end{array} 
 \longrightarrow  
 \begin{array}{c} \eta>0\\
 \mathrm{ so \ that \ w.h.p.}\\
  n^{{-2/3}} \cdot \Delta_{n}(\lambda_{0},\eta) \leq \varepsilon  \end{array} \longrightarrow 
  \begin{array}{c} \delta >0\\
 \mathrm{ so \ that \ w.h.p.}\\
   \delta \leq \gamma_{n}(\lambda_{0}, \eta),\end{array} $$
   where w.h.p.~indicates with high probability.

We first recall the construction of the standard augmented multiplicative coalescent following Broutin \& Marckert \cite{broutin2016new}. Let $(B_{t} : t \geq 0)$ be a linear Brownian motion and $\Xi$ be an independent Poisson point process on $ \mathbb{R}_{+} \times \mathbb{R}_{+}$ with unit intensity. For $\lambda \in \mathbb{R}$, we consider the process $ B^{(\lambda)}$ obtained by reflecting $t \mapsto B_{t} + \lambda t -  \tfrac{t^{2}}{2}$ above its running infimum (i.e.\ by subtracting the running infimum process).  Each excursion of $B^{(\lambda)}$ is then seen as a particle of mass given by its length and the surplus of this particle is the number of atoms of $\Xi$ that fall under this excursion (i.e.\ such that $B_{\tau}^{(\lambda)} \geq y$ if the atom lies at $(\tau,y)$). After ranking the particles in decreasing mass and recording their surplus, we get an element $ (  (\mathscr{M}_{i}(\lambda))_{i \geq 1}, (\mathscr{Sp}_{i}( \lambda))_{i \geq 1})$ of $$  \mathbb{U}^{\downarrow} = \left\{ (x,s) \in \ell^{2}_{\downarrow} \times \mathbb{Z}_{\geq 0}^{\infty} : \sum_{i \geq 1} x_{i}s_{i} < \infty \mbox{ and } s_{i} =0 \mbox{ whenever }x_{i}=0\right\}.$$ The process $\lambda \mapsto (  \mathscr{M}( \lambda), \mathscr{Sp}(\lambda))$ is the augmented multiplicative coalescent introduced in \cite{aldous1997brownian,bhamidi2014augmented} and appears as the scaling limit\footnote{Actually, we deal with a slightly different version of the $G(n,p)$ model since we have a fixed number $m$ of  edges and we allow self-loops and multiple edges, but this model is considered in \cite{bhamidi2014augmented,limic2017playful} and the result applies.} of the renormalized component sizes and surplus in $G_{n}(\lambda)$, see \cite{bhamidi2014augmented,broutin2016new}.

 This convergence holds for the Skorokhod topology {on $ \mathbb{C}\mathrm{adlag}( \mathbb{R}, \mathbb{U}^{\downarrow})$ where $\mathbb{U}^{\downarrow}$ is endowed} with the metric 
$$ \mathrm{d}_{ \mathbb{U}^\downarrow}\left((x,s),(x',s')	\right) = \left( \sum_{i \geq 1}|x_{i}-x_{i}'|^{2}	\right)^{1/2} + \sum_{i \geq 1} |x_{i}s_{i} -x_{i}'s_{i}'|.$$ Since $(x,s) \in  \mathbb{U}^{\downarrow} \mapsto \sum x_{i} \mathbf{1}_{s_{i}>0}$ is continuous for this topology, the previous convergence  implies the convergence of the total renormalized size $ n^{-2/3} \cdot \| \mathrm{O}_n(0)\|_{\bullet}$ of all components at time $\lambda=0$ carrying a surplus at time $0$, towards its continuous counterpart   \begin{eqnarray} \label{eq:cvOninfty}  n^{-2/3} \cdot \| \mathrm{O}_n(0)\|_{\bullet} &\xrightarrow[n\to\infty]{(d)}& \sum_{i \geq 1}  \mathscr{M}_{i}(0)\mathbf{1}_{ \mathscr{Sp}_{i}(0)>0}.  \end{eqnarray} Furthermore, we have seen above that every atom $(\tau, y)$ of $ \Xi$ corresponds, in the augmented coalescent to a surplus of one in a particle and we can define the time of appearance of this surplus as the smallest $\lambda \in \mathbb{R}$ so that $B_{\tau}^{(\lambda)} \leq y$, notice that $\lambda > -\infty$ almost surely for each atom. It follows from these observations, and the proof of Theorem 4 in \cite[Section 7.2]{broutin2016new} that we have the convergence in law for each $\lambda_0<0$ fixed
 \begin{eqnarray} \label{eq:cvOnlambda} n^{-2/3} \cdot \| \mathrm{O}_n( \lambda_0)\|_{\bullet} \xrightarrow[n\to\infty]{(d)} | \mathscr{O}( \lambda_{0})|,  \end{eqnarray} where $ |\mathscr{O}( \lambda_0)|$ is the total mass of all particles of $ \mathscr{M}(0)$ which have a surplus appeared before time $\lambda_{0}$. Also, $ |\mathscr{O}( \lambda_0)| \to 0$ as $\lambda_0 \to -\infty$ almost surely by dominated convergence. Together with the last display, this proves the first point of the proposition and gives the existence of $\lambda_0$ and $n_0$.

For the second and third item, notice that once $\lambda_0$ has been fixed, the convergence to the augmented multiplicative coalescent \cite{bhamidi2018multiplicative} implies that the masses of the specks (i.e. the renormalized sizes of the components of  $  \mathrm{G}_n^{[\lambda_0]}(\lambda_{0}) = [ \mathrm{G}_{n}(\lambda_{0})]_{ \mathrm{tree}}$) converge in distribution in the $\ell^2$ sense to  \begin{eqnarray} \label{eq:convergencestarting} \left(  \mathscr{M}_{i}( \lambda_{0}) \mathbf{1}_{ \mathscr{Sp}_{i}( \lambda_{0})=0} : i \geq 1 \right)^\downarrow.  \end{eqnarray} After an inoffensive Poissonization of the time (i.e.\ by letting the edges arrive according to a Poisson process instead of discrete time steps) the second point is a consequence of the Feller property of the multiplicative coalescent \cite[Proposition 5]{aldous1997brownian} together with the last display: in words the renormalized component sizes of $ \mathrm{G}^{[\lambda_{0}]}_{n}(0)$ are well approximated by restricting to specks of mass $\geq \eta$ in the $\ell^{2}$ sense (uniformly in $n$) hence in the $\ell^{\infty}$ sense so that $\sup_{n \geq 1} n^{-2/3} \Delta_{n}(\lambda_{0}, \eta) \to 0$ in probability as $\eta \to0$. The third convergence is similar and follows from the Feller property of the augmented coalescent \cite[Theorem 3.1]{bhamidi2018multiplicative}. See also \cite[Corollary 5.6]{rossignol2021scaling}.
\end{proof}

\subsection{Proof of Theorem \ref{thm:FMC}}
We now gather the deterministic controls established in Sections \ref{sec:gettingridoldcycles} and \ref{sec:skeleton} together with the probabilistic estimates of Proposition \ref{prop:controls} to prove Theorem \ref{thm:FMC}.  As announced, we start with a weaker convergence for the supremum norm.

\paragraph{Convergence for the supremum norm.}  We consider $ \mathcal{E}_{0} = \ell_{\downarrow,0}^{\infty} \times \ell_{\downarrow,0}^{\infty}$ where $\ell_{\downarrow,0}^{\infty}$ is the space of  non-increasing sequences tending to $0$ endowed with $ \mathrm{d}_{ \sup}$ (see \eqref{eq:defsupnorm}) which is a  Polish space.  Clearly $ \mathcal{E} \subset \mathcal{E}_{0}$ and the convergence for the $ \mathrm{d}_{ \mathcal{E}}$ distance is stronger than for $ \mathrm{d}_{\sup}$.

 Fix $ \varepsilon>0$ and find $\lambda_{0}, \eta, \delta>0$ and $n_{0} \geq 1$ as in Proposition \ref{prop:controls}. On the event described in this proposition of probability at least $1- \varepsilon$ for $n$ large enough we have for every $\lambda \in [\lambda_{0},0]$
 \begin{eqnarray} & & \mathrm{d}_{ \sup}\left( \mathbb{F}_{n}( \lambda) ; \mathbb{F}_{n}^{[\lambda_{0}, \delta]}( \lambda) \right)\nonumber \\
  &\underset{ \mathrm{trig.~ineg}}{\leq}& \mathrm{d}_{ \sup}\left( \mathbb{F}_{n}( \lambda) ; \mathbb{F}_{n}^{[\lambda_{0}]}( \lambda)\right) +  \mathrm{d}_{ \sup}\left(\mathbb{F}_{n}^{[\lambda_{0}]}( \lambda) ; \mathbb{F}_{n}^{[\lambda_{0}, \delta]}( \lambda) \right) \nonumber \\
& \underset{ \eqref{eq:firstcutoff}, \eqref{eq:secondcutoff}}{\leq} & n^{-2/3}\left(\| \mathrm{O}_n( \lambda_{0})\|_{\bullet} + \Delta_n( \lambda_{0}, \eta)\right) \nonumber \\
& \underset{  \mathrm{Prop. \ } \ref{prop:controls}}{\leq} & 2 \varepsilon. \label{eq:control2}  \end{eqnarray}
On the other hand, by \eqref{eq:convergencestarting}, the starting configuration of $ \mathbb{F}_{n}^{[\lambda_{0}, \delta]}$ converges in law towards some vector having a finite number of non-zero components. Since there are only finitely many particles to take care of, applying the dynamics of the frozen coalescent it should be clear that for fixed $ \lambda_{0}<0$  and $\delta>0$
$$ \left(\mathbb{F}_{n}^{[\lambda_{0}, \delta]}( \lambda) : \lambda \in [\lambda_{0},0] \right)$$
converges in distribution for the Skorokhod topology on $ \mathbb{C} \mathrm{adlag}( [\lambda_{0},0], \mathcal{E}_{0})$. If $ \mathrm{d}_{LP}$ denotes the L\'evy--Prokhorov distance associated to the convergence in law for the Skorokhod topology on $ \mathbb{C} \mathrm{adlag}( [\lambda_{0},0], \mathcal{E}_{0})$ then restricting \eqref{eq:control2} to $\lambda \in [-1,0]$ we deduce 
$$ \mathrm{d}_{LP}\left( \left(\mathbb{F}_{n}( \lambda)\right)_{\lambda \in [-1,0]};  \left(\mathbb{F}_{n}^{[\lambda_{0}, \delta]}( \lambda)\right)_{\lambda \in [-1,0]} \right) \leq 2\varepsilon,$$ for all $n \geq n_{0}$ (this actually holds for the supremum norm which is stronger than the Skorokhod distance). Since $(\mathbb{F}_{n}^{[\lambda_{0}, \delta]}( \lambda): \lambda \in [-1,0])$ is converging in law as $n \to \infty$, we can combine this with the last display to deduce that $\left(\mathbb{F}_{n}( \lambda)\right)_{\lambda \in [-1,0]}$ is  Cauchy for $ \mathrm{d}_{LP}$ and so converges as desired. Its limit is obtained by first letting $n \to \infty$, then $\delta \to 0$ and finally $\lambda_{0} \to -\infty$ in the process $ n^{-2/3}\cdot\mathrm{F}_{n}^{[\lambda_{0}, \delta]}$. \qed

\paragraph{Convergence in $ \ell^{1}_{\downarrow} \times \ell^{2}_{\downarrow}$.} To upgrade the previous convergence for the $ \mathrm{d}_{\sup}$ distance to a convergence for the distance $ \mathrm{d}_{ \mathcal{E}}$, we need to prove tightness i.e.\ to control uniformly over $\lambda \in [-1,0]$ the cumulative effect of the small component sizes in our frozen coalescent processes. More precisely, for any $\xi\geq 0$ and $ \mathbf{z}=( \mathbf{x}, \mathbf{y}) \in \mathcal{E}$ we denote by 
$$ R_{\bluecirc,\xi} ( \mathbf{z}) = \sum_{i \geq 1} x_{i}\mathbf{1}_{x_{i} \leq \xi} \quad \mbox{ and  }\quad  R_{\circ,\xi} (  \mathbf{z}) = \sum_{i \geq 1} y_{i}^{2} \mathbf{1}_{y_{i} \leq \xi},$$
respectively the sum of the masses of the blue particles and sum of the squares of the masses of the white particles of  mass smaller than $ \xi$. We also put $ R_{\xi}( \mathbf{z}) = R_{\xi, \bluecirc}( \mathbf{z}) + R_{\xi, \circ}( \mathbf{z})$. We then have : 

\begin{lemma}[Towards tightness of $\mathbb{F}_{n}( \cdot)$] \label{lem:tightness} For any $ \varepsilon>0$ there exists $\xi>0$ and $n_0 \geq 1$ such that for all $n \geq n_0$ we have with probability at least $1- \varepsilon$
$$ \sup_{\lambda \in [-1,0]} \left( R_{\xi}( \mathbb{F}_{n}(\lambda)) \right) \leq \varepsilon.$$

\end{lemma}
\begin{proof}  Let us begin with the $\ell^2$-part. By the inclusion of the frozen exploration process in the standard Erd{\H{o}}s--R\'enyi process, all the components of $ \mathrm{F}_n(\lambda)$ for $\lambda \in (-\infty,0]$ -frozen or not- are contained in $ \mathrm{G}_n(0)$. Next, if $0 \leq f_1, \dots , f_k \leq \xi$ and $f_1+\dots + f_k \leq y$ then we have 
  \begin{eqnarray}\label{eq:sumcarres} f_1^2+ \cdots + f_k^2 \leq (f_1+ \cdots + f_k) \cdot (\xi \wedge y) \leq \left(\xi \cdot y\right) \wedge y^2.  \end{eqnarray} We apply this inequality when $f_{1}, \dots , f_{k}$ are the renormalized sizes of the components in $ \mathrm{F}_n(\lambda)$ which are included in the same component of $ \mathrm{G}_n(0)$ with renormalized size $y$ and this for each component of $ \mathrm{G}_n(0)$ which is made of small components of $ \mathrm{F}_n(\lambda)$: if we denote by $ \boldsymbol{Y^{(n)}}=(Y^{(n)}_i~:~i~\geq~1)$ the decreasing sizes of the components of $ \mathrm{G}_n(0)$ renormalized by $n^{-2/3}$ then  we have 
$$\sup_{\lambda < 0}  R_{\circ,\xi}( \mathbb{F}_n( \lambda)) \leq \sum_{i\geq1} \left(\xi \cdot Y^{(n)}_i\right) \wedge \left(Y_i^{(n)}\right)^2.$$
By the result of Aldous \cite{aldous1997brownian}, the sequence $(Y_{i}^{(n)} : i \geq 1)$ converges in distribution for the $\ell^{2}_{\downarrow}$ distance towards the multiplicative coalescent $ \mathscr{M}(0)$ at time $0$. Since $ \psi_{\xi} : \ell^{2}_{\downarrow} \to \mathbb{R}_{+}$ defined by $\psi_{\xi}( (y_{i} : i \geq 1)) = \sum_{i \geq 1} y_{i} \cdot (y_{i} \wedge \xi)$ is continuous for the $\ell^{2}_{\downarrow}$-distance we deduce from the previous convergence that $ \psi_{\xi}( \boldsymbol{Y^{(n)} })$ converges in law towards $\psi_{\xi}(  \mathscr{M}(0))$. Furthermore, by dominated convergence we have $ \psi_{\xi}( \mathscr{M}(0)) \to 0$ a.s. as $\xi \to 0$.
We deduce that 
$$ \forall \varepsilon >0,  \quad \sup_{n \geq 1}\mathbb{P}\left( \sum_{i\geq1} \left(\xi \cdot Y^{(n)}_i\right) \wedge \left(Y_i^{(n)}\right)^2  \geq \varepsilon\right) \xrightarrow[\xi\to 0]{} 0,$$
and this takes care of the $R_{\circ,\cdot}$ part of the lemma. 

The $\ell^1$-part is a bit trickier. Recall from Section \ref{sec:gettingridoldcycles} that for any $\lambda \in (-\infty,0]$, the frozen components of $ \mathrm{F}_n(\lambda)$ are included in $ \mathrm{O}_n(0)$, the union of the components of $ \mathrm{G}_n(0)$ that have a surplus. 
Notice that if $k \geq 1$ frozen components of $ \mathrm{F}_n( \lambda)$ belong to the same component of $ \mathrm{G}_n(0)$, then this component must have surplus at least $k$ (recall that each frozen component contains exactly one cycle). Hence if $ \boldsymbol{X^{(n)}}=(X_{i}^{(n)} : i \geq 1)$ are the decreasing sizes of the components of $ \mathrm{O}_{n}(0)$ renormalized by $n^{-2/3}$ and if  $K_n$ is the maximum surplus of a component in $ \mathrm{G}_n(0)$ then we have for all $\lambda \in (-\infty,0]$ $$ R_{\bluecirc, \xi}( \mathbb{F}_n(\lambda)) \leq K_n \cdot \sum_{i \geq 1} \left(X_i^{(n)} \wedge \xi \right).$$By \cite[Theorem 1]{luczak1994structure}, the sequence $(K_n : n \geq 1)$ is tight. 
From the discussion just before \eqref{eq:cvOninfty} we get that $\boldsymbol{X^{(n)}}$ converges in law for the $\ell^{1}_{\downarrow}$-topology towards the masses $\left( \mathscr{M}_{i}(0) \mathbf{1}_{ \mathscr{Sp}_{i}(0)>0} \right)^{\downarrow}$ of the particles in the augmented multiplicative coalescent at time $0$ that carry a surplus. By the same argument as above we deduce that for every $ \varepsilon>0$, there exists $\xi >0$ such that
$ \mathbb{P}\left( K_n \cdot \sum_{i\geq1} \left(\xi \cdot X^{(n)}_i\right)  \geq \varepsilon\right) \leq  \varepsilon$ for all $n \geq 1$ and this finishes the proof of the lemma.  \end{proof}

We can now finish the proof of Theorem \ref{thm:FMC}: {Recall that  $ \mathbb{F}_{n}(\lambda)$ is the renormalized sizes of the frozen and standard components in $ \mathrm{F}_{n}(\lambda)$. In a nutshell, the tightness of $  \left( \mathbb{F}_{n}\left(\lambda \right) : \lambda \in [-1,0] \right)$ for the Skorokhod topology with values in $ \mathcal{E}_{0}$ together with the last lemma establishes the tightness of $ \left( \mathbb{F}_{n}\left(\lambda \right) : \lambda \in [-1,0] \right)$ for the Skorokhod topology with values in $ \mathcal{E}$. Since the convergence in $ \mathcal{E}_{0}$ determines the law, we are done. Let us provide some details.  Recall that we already proved that  
 \begin{eqnarray} \label{eq:convdiscretedsup} \left( \mathbb{F}_{n}\left(\lambda \right): \lambda \in [-1,0] \right) \xrightarrow[n\to\infty]{(d)}  \left( \FM( \lambda) : \lambda \in [-1,0]\right),  \end{eqnarray}
for the Shokorhod topology on the space $ \mathbb{C} \mathrm{adlag}( \mathbb{R} , \mathcal{E}_0)$ of c\`adl\`ag functions with values in $ \mathcal{E}_0$ endowed with $ \mathrm{d}_{\sup}$. By Skrokorhod representation theorem, we can then assume that for each $n \geq1$, the processes $ \mathbb{F}_{n}$ and $ \FM$ are coupled in such a way that the Skorokhod distance between $  \mathbb{F}_{n}(\cdot)$ and  $\FM$ converges almost surely to $0$ for $ \mathrm{d}_{ \mathrm{sup}}$ as $n \to \infty$. This means that we can find increasing time shifts $\psi_{n} : [-1,0] \to [-1,0]$ with $\| \psi_{n}- \mathrm{Id}\| \to 0$ a.s. and such that 
 \begin{eqnarray} \label{eq:skoexplained} \sup_{\lambda \in [-1,0]}  \mathrm{d}_{ \mathrm{sup}}(  \mathbb{F}_{n}(\lambda) , \FM(\psi_{n}(\lambda)) \xrightarrow[n\to\infty]{a.s.} 0.  \end{eqnarray}
Recalling the notation introduced before Lemma \ref{lem:tightness} and using this lemma, up-to rebuilding a new coupling, one can furthermore suppose that we have for every $\xi >0$
$$ \limsup_{ \xi \downarrow 0} \sup_{n \geq 1} \sup_{\lambda \in [-1,0]}  R_{\xi}( \mathbb{F}_n( \lambda)) \xrightarrow[n\to\infty]{a.s.} 0.$$
In particular, by Fatou's lemma, this implies a similar estimate for $ \FM$ namely,
$$ \limsup_{ \xi \downarrow 0} \sup_{\lambda \in [-1,0]} R_{\xi}(  \FM( \lambda))\xrightarrow[n\to\infty]{a.s.} 0.$$
One can now use our coupling to evaluate the $ \mathcal{E}$-distance between $ \mathbb{F}_{n}$ and $ \FM$, namely 
$$ \mathrm{d}_{ \mathcal{E}}( \mathbb{F}_{n}(\lambda), \FM( \psi_{n}( \lambda))) \leq  2\sup_{\lambda \in [-1,0]}( R_{\xi}( \FM(\lambda)) + R_{\xi}( \mathbb{F}_{n}(\lambda))) +   \mathrm{d}_{ \mathcal{E}}(  \mathbb{F}_{n}^{\{\xi\}}(\lambda) , \FM^{\{\xi\}}(\psi_{n}(\lambda))).$$

Using the second and third to last displays,  the first term on the right-hand side can be made small uniformly in $n$ and $\lambda \in [-1,0]$ by choosing $\xi$ small enough. The second term also tends to $0$ thanks to \eqref{eq:skoexplained}: since $ \FM$ is c\`adl\`ag with values in $ \mathcal{E}_{0}$, the maximal size of a particle in $ \FM$ (and in $ \mathbb{F}_{n}(\cdot)$) and the maximal number of particles of mass $> \xi$ is bounded over $[-1,0]$. We have indeed proved that in this coupling we have $ \mathbb{F}_{n} \to \FM$ almost surely for the $  \mathrm{d}_{ \mathcal{E}}$ metric. This implies the desired result.}

\section{Markovian properties of the freezer and the flux}
Since the mapping $ \mathbf{z} \in \mathcal{E} \to \| \mathbf{z}\|_{ \bluecirc}$  is continuous for the topology induced by $ \mathrm{d}_{ \mathcal{E}}$, our Theorem \ref{thm:FMC} implies that
 \begin{eqnarray} \label{eq:convfrozenfirst} ( n^{-2/3} \cdot \|  \mathrm{F}_{n}( \lambda)\|_{\bluecirc} : \lambda \in \mathbb{R}) \xrightarrow[n\to\infty]{(d)} ( \| \FM(\lambda)\|_{\bluecirc} : \lambda \in \mathbb{R}),  \end{eqnarray}
for the Skorokhod topology where we recall that $\| \FM(\lambda)\|_{\bluecirc}$ is the total mass of the frozen particles in the frozen multiplicative coalescent. In this section, we use the Markov properties of the process $\|F(n,\cdot)\|_{\bluecirc}$ --or more precisely of $(\|F(n, \cdot)\|_{\bluecirc},\|F(n,\cdot)\|_{\edge})$-- given in Proposition \ref{prop:freeforest} to prove (Proposition \ref{prop:flux}) the joint convergence  of the number of discarded edges $ D(n,\cdot)$ in the scale $n^{1/3}$ which, thanks to our coupling construction, will give us the flux of outgoing cars in the parking process on Cayley trees.

Using this Markovian point of view, we also give in Proposition \ref{prop:fellerX} a new and perhaps more concrete construction of the process $\lambda \mapsto \|\FM(\lambda)\|_{\bluecirc}$ as a pure-jump Feller process with a time-inhomogeneous infinite jump measure 
 \begin{eqnarray} \label{def:jumpkernel} \mathbf{n}_{\lambda}(x, \mathrm{d}y) =\frac{1}{2} \cdot  \frac{ \mathrm{d}y}{ \sqrt{2\pi y^{3}}}   g_{x, \lambda}(y),  \end{eqnarray} where $g_{x, \lambda}(y)$ was defined in \eqref{def:gxy}. We complete this alternative Markovian description of the frozen multiplicative coalescent by computing the law of $ [\FM(\lambda)]_{\circ}$ given $ \| \FM(\lambda)\|_{\bluecirc}$ (see Proposition \ref{prop:martinyeo}): By passing Proposition \ref{prop:freeforest} to the scaling limit, conditionally on $ \| \FM(\lambda) \|_{ \bluecirc}$, the $\ell^{2}$-part of $ \FM(\lambda)$ has the law of the scaling limit of component sizes in a  critical random forest. This law has been described in  \cite{martin2018critical} using excursion lengths of a time inhomogeneous diffusion with a reflection term. We give here an alternative (and quicker) description of this law using conditioned $3/2$-stable L\'evy processes. The last two results are not used for the parking process but we include them to motivate further the study of the frozen Erd{\H{o}}s--R\'enyi process, see Part \ref{sec:comments} for perspectives.

\subsection{Scaling limits for the freezer and flux}
The main result of this section is the joint convergence of the renormalized number of discarded edges $D(n,m)$ together  with \eqref{eq:convfrozenfirst}, see Figure \ref{fig:freezer+flux}.

\begin{proposition}[Joint convergence of discarded edges] \label{prop:flux} Jointly with the convergence of Theorem~\ref{thm:FMC} we have the following convergence in distribution  for the uniform topology on $ \mathcal{C}( \mathbb{R},  \mathbb{R})$ 
$$ \left(n^{-1/3} \cdot \mathrm{D}_{n}(\lambda) \right)_{ \lambda \in \mathbb{R}} \xrightarrow[n\to\infty]{(d)} \left(\frac{1}{2}\int_{-\infty}^{\lambda} \mathrm{d}s \, \| \FM( s)\|_{\bluecirc} \right)_{\lambda \in \mathbb{R}.}$$ 

\end{proposition}

\begin{figure}[!h]
 \begin{center}
 \includegraphics[width=10cm]{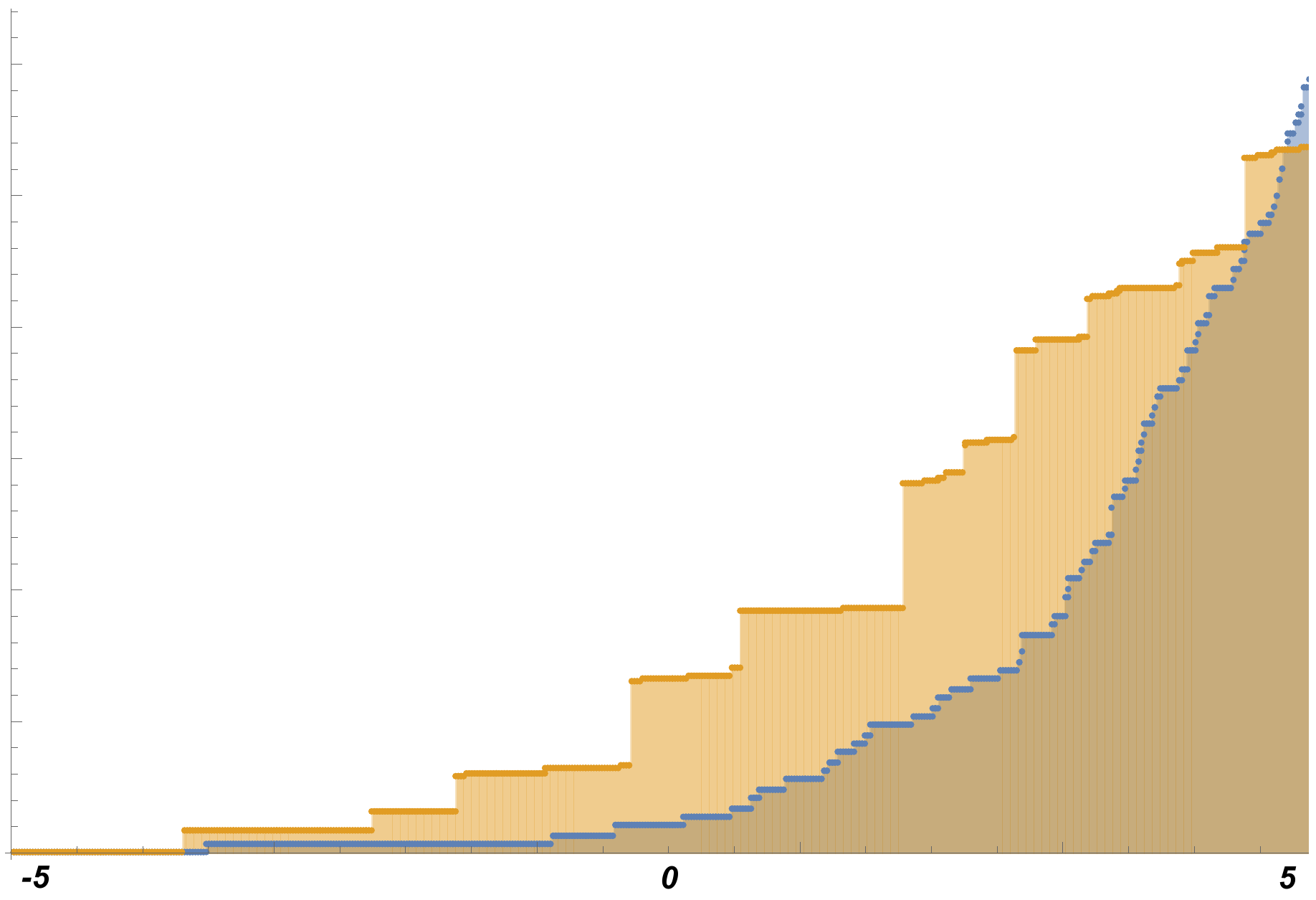}
 \caption{A simulation of the renormalized processes $ n^{-2/3} \cdot \|\mathrm{F}_{n}(\lambda)\|_{\bluecirc}$ (in orange) and $ n^{-1/3}\cdot \mathrm{D}_{n}(\lambda)$ (in blue) for $\lambda \in [-5,5]$ and $n=3000$. Notice that the first process is pure jump in the scaling limit and that the second one is the integral of the first. \label{fig:freezer+flux}}
 \end{center}
 \end{figure}

The main ingredient to prove the joint convergence of the number of discarded edges together with Theorem \ref{thm:FMC}  is the following consequence of Corollary \ref{cor:freezer}
\begin{equation}\label{eq:transitionflux}\mathbb{P} \left( \Delta D(n,m)  =1  \Big|  \|F(n,m)\|_{\bluecirc} \right) = \frac{ \| F(n,m)\|_{\bluecirc}}{n}, \end{equation} so that the above result formally follows from integrating  and passing to the limit. To make this precise, we shall start with a lemma controlling the flux at the bottom of the critical window:
\begin{lemma}  \label{lem:controlfluxbas}There exists a constant $C>0$ such that for any $n \geq 1$ and any $\lambda<-1$
$$ \mathbb{E}[ \mathrm{D}_{n}(\lambda)] \leq \frac{C}{|\lambda|} n^{1/3}.$$
\end{lemma}

 \proof  In this proof, for $n \geq 1$ and  $m \geq 0$ we denote by $ \mathrm{C}\hspace{-0.3mm}\ell(n,m)$ the cluster of the vertex $1$ in $ G(n,m)$ and  by $ \mathrm{Spl}(n,m) = \| \mathrm{C}\hspace{-0.3mm}\ell(n,m)\|_{\edge} - \| \mathrm{C}\hspace{-0.3mm}\ell(n,m)\|_{\bullet} +1$ its surplus. Recall  from \eqref{eq:inclusionbidon} that the number of blue vertices in $F(n,m)$ is less than the number of vertices which belong to a component which has a cycle in $G(n,m)$. By \eqref{eq:transitionflux} and taking expectations we deduce that 
 \begin{eqnarray} \label{eq:deltaDDD}
 \mathbb{E} \left[ \Delta D(n,m)  \right] = \frac{1}{n} \mathbb{E} \left[ \|F(n,m)\|_{\bluecirc} \right]  \leq \mathbb{P} \left(  \mathrm{Spl}(n,m) \geq 1 \right) \leq \mathbb{E} \left[  \mathrm{Spl}(n,m)  \right].
 \end{eqnarray}
Now we provide an upper bound for $\mathbb{E} \left[  \mathrm{Spl}(n,m) \right]$ following the proof of \cite[Theorem 1.2]{federico2020expansion}. Note that $\mathrm{Spl}(n,m) $ is bounded from above by the number of vertex-disjoint cycles (including self-loops) in the cluster of $1$. Given a graph $ \mathfrak{g}$ and $k$ distinct vertices $(v_1, \dots , v_k)$, we say that the graph $\mathfrak{g}$ \emph{contains the cycle} $(v_1, \dots v_k)$  of length $k \geq 1$ if $\mathfrak{g}$ contains the (unoriented) edges $(v_i, v_{i+1\,\mathrm{mod}\, k})$ for $ 1 \leq i \leq k$. For $k=1$, the expected number of self-loops (or cycle of length $1$) in $ \mathrm{C}\hspace{-0.3mm}\ell(n,m)$ is bounded above by $m \E \left[ \| \mathrm{C}\hspace{-0.3mm}\ell(n,m) \|_{\bullet} \right]/n^2.$ Since every cycle of length $ k \geq 3$ (resp.\ 2) corresponds to exactly $2k$ (resp.\ 2) $k$-uplet $(v_1, \dots, v_k)$, we have for $2m<n$, 
\begin{align} &\sum_{k=2}^n \sum_{L_k \mbox{\tiny{ cycle of length }} k} \mathbb{P}\left( G(n,m) \mbox{ contains } L_k \mbox{ and } 1 \mbox{ is connected to } L_k\right)\notag\\
&\leq  \sum_{k=2}^n \sum_{L_k = (v_1, \dots, v_k)} \frac{1}{k} \sum_{I \subset \{ 1, \dots, m\}, |I| = k}\mathbb{P}( (\vec{E}_i)_{i \in I}\mbox{  form } L_k  \mbox{ and } (\vec{E}_i)_{i \leq m, i \notin I}\mbox{ connect } 1 \mbox{ with } \{ v_1,\dots v_k\} ) \notag\\
&\leq   \sum_{k=2}^n \sum_{(v_1, \dots, v_k)} \frac{1}{k} \sum_{I \subset \{ 1, \dots, m\}, |I| = k} k! \left(\frac{2}{n}\right)^k \cdot \frac{2k}{n} \E\left[\|\mathrm{C}\hspace{-0.3mm}\ell(n,m)\|_{\bullet}  \right] \notag\\
&=  \sum_{k=2}^n \frac{n!}{k(n-k)!} {m \choose k} k! \left(\frac{2}{n^2}\right)^k \cdot \frac{2k}{n} \mathbb{E}\left[\|\mathrm{C}\hspace{-0.3mm}\ell(n,m)\|_{\bullet} \right] \notag\\
&\leq \frac{2}{n} \sum_{k=2}^n  \left(\frac{2m}{n}\right)^k \mathbb{E}\left[ \|\mathrm{C}\hspace{-0.3mm}\ell(n,m)  \|_{\bullet} \right] \leq 
  \frac{C}{n\left(1- \frac{2m}{n}\right)} \mathbb{E}\left[\|\mathrm{C}\hspace{-0.3mm}\ell(n,m)\|_{\bullet} \right],\nonumber 
\end{align}
for some constant $C>0$ that may vary in the following lines. An easy adaptation of \cite[Theorem 1.1]{janson2008susceptibility} using \cite[Remark 1.6]{janson2008susceptibility} to our model $ G(n,m)$ shows that 
\begin{equation} \label{eq:compbas} \mathbb{E}\left[ \|\mathrm{C}\hspace{-0.3mm}\ell(n,m)\|_{\bullet}  \right] \leq \frac{C}{\left(1- \frac{2m}{n}\right)},\end{equation}
so that combining these inequalities we deduce that for $ 2m < n$ 
  \begin{eqnarray} \label{eq:surplus}
   \mathbb{E}[ \mathrm{Spl}(n,m)] \leq \frac{C}{n (1- \frac{2m}{n})^{2}}.  \end{eqnarray} Coming back to $D$ and writing $  {m}_n( \lambda) = 0 \wedge \lfloor \frac{n}{2} +  \frac{\lambda}{2} n^{2/3} \rfloor$, we obtain using \eqref{eq:deltaDDD} that for $ \lambda <-1$, 
\begin{eqnarray*} \mathbb{E} \left[ \mathrm{D}_n( \lambda)  \right]  = \mathbb{E} \left[ D (n, \mathrm{m}_n( \lambda) )\right] = \sum_{m=0}^{ \mathrm{m}_n( \lambda)-1} \mathbb{E} \left[ \Delta D(n,m)  \right] \leq C \sum_{m=0}^{ \mathrm{m}_n( \lambda)-1} \frac{1}{n(1- \frac{2m}{n})^2} \leq C \frac{n^{1/3}}{|\lambda|},
\end{eqnarray*}
which concludes the proof.
 \endproof 
\proof[Proof of Proposition \ref{prop:flux}] Recall that the convergence of Theorem  \ref{thm:FMC} implies the convergence of  of $(n^{-2/3} \| \mathrm{F}_n ( \lambda)\|_{\bluecirc} : \lambda \in \R )$.  We now prove the joint convergence of $ (n^{-1/3} \mathrm{D}_n ( \lambda) : \lambda \in \R )$ using the probability transitions given in \eqref{eq:transitionflux}. Indeed, writing $0 \wedge \mathrm{m}_n( \lambda) =  \lfloor \frac{n}{2} +  \frac{\lambda}{2} n^{2/3} \rfloor$ as above, we have
 $$ \mathrm{D}_n ( \lambda) - \mathrm{D}_n ( \lambda_0) = \sum_{m=\mathrm{m}_n( \lambda_0)}^{\mathrm{m}_n( \lambda)-1} \Delta D(n,m),$$ 
and for all $m \geq 0$, conditionally on $\|F(n,m)\|_{\bluecirc}$, the variable $\Delta D(n,m)$ is a Bernoulli variable with parameter $ \|F(n,m)\|_{\bluecirc}/n$.
By the convergence of the process $(n^{-2/3} \| \mathrm{F}_n ( \lambda)\|_{\bluecirc} : \lambda \in \R )$ to $ ( \| \FM( \lambda)\|_{\bluecirc} : \lambda \in \mathbb{R})$ in the Skorokhod sense,  we have for every $- \infty < \lambda_0 < \lambda < + \infty$,
 $$ \sum_{m= \mathrm{m}_n( \lambda_0)}^{\mathrm{m}_n( \lambda)-1} n^{-4/3} \|F(n,m)\|_{\bluecirc} = \frac{1}{2} n^{-2/3}\int_{n^{-2/3}(\mathrm{m}_n( \lambda_0)- n/2)}^{n^{-2/3}(\mathrm{m}_n( \lambda)-n/2)} \| \mathrm{F}_n (s)\|_{\bluecirc} \mathrm{d}s   \xrightarrow[n\to\infty]{(d)} \frac{1}{2}\int_{\lambda_0}^{\lambda} \mathrm{d}s \, \| \FM( s)\|_{\bluecirc}.$$
In addition, since $ (\Delta D(n,m) - \|F(n,m)\|_{\bluecirc}/n : m \geq 0)$ are the increments of a martingale,
\begin{align*} \mathbb{E}&\left[ \left( n^{-1/3} \sum_{m= \mathrm{m}_n( \lambda_0)}^{\mathrm{m}_n( \lambda)-1} \Delta D(n,m) -  \sum_{m= \mathrm{m}_n( \lambda_0)}^{\mathrm{m}_n( \lambda)-1} n^{-4/3} \|F(n,m)\|_{\bluecirc} \right)^2 \right]  
\\&= n^{-2/3}\sum_{m= \mathrm{m}_n( \lambda_0)}^{\mathrm{m}_n( \lambda)-1}\mathbb{E}\left[ \Big(  \Delta D(n,m) -  \frac{\|F(n,m)\|_{\bluecirc}}{n} \Big)^2 \right]   \\
&= n^{-2/3} \sum_{m= \mathrm{m}_n( \lambda_0)}^{\mathrm{m}_n( \lambda)-1}\mathbb{E}\left[ \frac{\|F(n,m)\|_{\bluecirc}}{n}\left( 1 - \frac{\|F(n,m)\|_{\bluecirc}}{n}\right)^2 + \left( 1 - \frac{\|F(n,m)\|_{\bluecirc}}{n}\right)\cdot\left( \frac{\|F(n,m)\|_{\bluecirc}}{n}\right)^2\right]\\
&\leq  n^{-2/3} \sum_{m= \mathrm{m}_n( \lambda_0)}^{\mathrm{m}_n( \lambda)-1}\mathbb{E}\left[ \frac{\|F(n,m)\|_{\bluecirc}}{n} \right]\end{align*}
which converges to $0$ as $n \to \infty$ by the estimates of (the proof of) the previous lemma. It follows that 
 $$ \mathrm{D}_n ( \lambda) - \mathrm{D}_n ( \lambda_0)  \xrightarrow[n\to\infty]{(d)} \frac{1}{2}\int_{ \lambda_0}^{\lambda} \mathrm{d}s \, \| \FM( s)\|_{\bluecirc}.$$
Now we use Lemma \ref{lem:controlfluxbas}, which shows that $ \mathbb{E}\left[n^{-1/3} \mathrm{D}_n(\lambda_0)\right]$ can be made arbitrarily small if we choose $\lambda_{0}$ small enough, and this uniformly in $n$. Hence, by Fatou's lemma, we obtain for all $ \lambda_0 \in (- \infty, \lambda]$,
$$  \frac{1}{2}\int_{ \lambda_0}^{\lambda} \mathrm{d}s \, \| \FM( s)\|_{\bluecirc} < \infty \qquad \mbox{a.s.}, $$
and letting $ \lambda_0 \to - \infty$, we get 
\begin{eqnarray} \label{eq:cvdisc}
 \mathrm{D}_n ( \lambda) \xrightarrow[n\to\infty]{(d)} \frac{1}{2}\int_{-\infty}^{\lambda} \mathrm{d}s \, \| \FM( s)\|_{\bluecirc} = \mathscr{D}(\lambda).
 \end{eqnarray}
The above reasoning can be extended to prove that, jointly with the convergence of the first coordinate in Proposition \ref{prop:flux}, for each $ -\infty < \lambda_{1} < \lambda_{2} < \cdots < \lambda_{k} < \infty$ we have $ n^{-1/3}\cdot \mathrm{D}_{n}(\lambda_{i}) \to  \mathscr{D}(\lambda_{i})$. Since the processes $ \mathrm{D}_{n}( \cdot)$ are increasing and since $ \lambda \mapsto \mathscr{D}(\lambda)$ is continuous and increasing as well, this is sufficient to imply the joint convergence of $ n^{-1/3} \cdot \mathrm{D}_{n}(\cdot)$ to $ \mathscr{D}(\cdot)$ for the uniform norm over every compact of $ \mathbb{R}$. \endproof

\subsection{Proof of Theorem \ref{thm:composantes}}
\label{sec:proofthm1}

\paragraph{Near components.} Recall the notation $ \mathbb{F}_n ( \lambda) \in \mathcal{E}$ for the renormalized components sizes (first frozen, followed by the standard ones) in $ \mathrm{F}_n(\lambda)$. Accordingly, we write  $ \mathbb{T}_{ \mathrm{near},n} ( \lambda)$ for the vector $$ \left( n^{-2/3} \mathrm{C}_{*,n}(\lambda) ; \left( n^{-2/3} \cdot \mathrm{C}_{i,n}(\lambda) :  i \geq 1\right)\right),$$ where $ \mathrm{C}_{*,n}(\lambda)$ and $ \mathrm{C}_{i,n}(\lambda)$ are the sizes of the (blue)  root component followed by the other components in decreasing order of size in $ \mathrm{T}_{  \mathrm{near}, n}(\lambda)$. 
Recall that by Proposition \ref{prop:CouplingCayleyFrozen}, the white components of $F(n,m)$ are exactly the components of $T_{ \mathrm{near}}(n,m)$ which do not contain the root and the blue vertices of $F(n,m)$ are exactly the vertices of the parked component of the root. Furthermore, the flux of outgoing cars $D(n,m)$ corresponds to the number of discarded edges in $F(n,m)$. Since the mapping $ \mathbf{z} \in \mathcal{E} \mapsto ( \| \mathbf{z}\|_{\bluecirc}, [\mathbf{z}]_\circ) \in  \mathbb{R}_+ \times \ell^2_{\downarrow}$ is continuous, we can combine Theorem \ref{thm:FMC} and Proposition \ref{prop:flux}  to deduce that 
$$ \left( n^{-1/3} \cdot \mathrm{D}_n(\lambda) ; \mathbb{T}_{ \mathrm{near},n}(\lambda) \right)_{\lambda \in \mathbb{R}} \xrightarrow[n\to\infty]{(d)} \left( \mathscr{D}(\lambda) , \| \FM(\lambda)\|_{\bluecirc}, [\FM(\lambda)]_\circ\right)_{\lambda \in \mathbb{R}},$$ 
for the Skorokhod topology on $ \mathbb{C} \mathrm{adlag}( \mathbb{R} \times \mathbb{R} \times \ell^2_\downarrow)$. \qed

\paragraph{Full components.} Let us sketch how to obtain the equivalent of Theorem \ref{thm:composantes} for the full components of $ \mathrm{T}_{ \mathrm{full}}(n,\cdot)$ rather than the near components. To extend the above convergence to the case of full components, notice first that $ \mathrm{D}_n(\lambda)$ stays the same for near and full components, and that as soon as $ \mathrm{D}_n(\lambda)>0$ (which is the case with high probability in the whole critical window) we have with an obvious notation 
$$ \| \mathbb{T}_{ \mathrm{near},n}(\lambda)\|_{\bluecirc} =  \| \mathbb{T}_{ \mathrm{full},n}(\lambda)\|_{\bluecirc},$$ since the blue components of the root with flux are the same in $ \mathrm{T}_{ \mathrm{near}}$ and in $ \mathrm{T}_{ \mathrm{full}}$. We just have to show that for any fixed compact time interval $I$ we have $$ \sup_{\lambda \in I} \mathrm{d}_{\ell^2_\downarrow}\left( [ \mathbb{T}_{ \mathrm{near},n}(\lambda)]_\circ ,[ \mathbb{T}_{ \mathrm{full},n}(\lambda)]_\circ\right) \xrightarrow[n\to\infty]{( \mathbb{P})} 0,$$
 with an obvious notation. For any fixed $\lambda$, the fully parked trees of  $\mathrm{T}_{ \mathrm{full},n}(\lambda)$  are obtained by splitting the nearly parked trees of $\mathrm{T}_{ \mathrm{near},n}(\lambda)$ at their root vertices. Since by Proposition \ref{prop:FPT/AFPTuniform}, conditionally on their sizes, those are uniform nearly parked trees (this can also be seen by combining our coupling construction with Proposition \ref{prop:freeforest} and Proposition \ref{prop:RWcoding}), we deduce from Proposition \ref{prop:nearlytofully} that each large nearly parked tree of $\mathrm{T}_{ \mathrm{near},n}(\lambda)$ contains a unique fully parked tree of roughly the same size. Since $ [\mathbb{T}_{ \mathrm{near},n}(\lambda)]_{\circ}$ converges in $\ell^{\infty}_{\downarrow, 0}$ we easily deduce that for each $\lambda \in \mathbb{R}$ we have 
 $$ \mathrm{d}_{\ell^\infty_{\downarrow,0}}\left( [ \mathbb{T}_{ \mathrm{near},n}(\lambda)]_\circ ,[ \mathbb{T}_{ \mathrm{full},n}(\lambda)]_\circ\right) \xrightarrow[n\to\infty]{( \mathbb{P})} 0.$$
Actually, the last display also holds for any stopping time $\Lambda$ which belongs to some fixed time interval. Combining this we the monotony property of the processes $ \mathrm{T}_{ \mathrm{near}}$ and $ T_{ \mathrm{full}}$, standard but tedious arguments (which we shamefully leave to the reader) show that we in fact have 
 $$ \sup_{\lambda \in I} \mathrm{d}_{\ell^\infty_{\downarrow,0}}\left( [ \mathbb{T}_{ \mathrm{near},n}(\lambda)]_\circ ,[ \mathbb{T}_{ \mathrm{full},n}(\lambda)]_\circ\right) \xrightarrow[n\to\infty]{( \mathbb{P})} 0.$$ To boostrap the above convergence  by replacing the $\ell^{\infty}_{\downarrow, 0}$ metric with the $ \ell^{2}_{\downarrow}$ metric, we use Lemma \ref{lem:tightness} on $[ \mathbb{T}_{ \mathrm{near},n}(\lambda)]_\circ$ and remark that the proof straightforwardly extend to  $[ \mathbb{T}_{ \mathrm{full},n}(\lambda)]_\circ$. \qed
 
\paragraph{Strong components.} Obviously a version of Theorem \ref{thm:composantes} holds if we consider the strong components in the parking process to the cost of multiplying the scaling limit of the components sizes by $1/2$, since by Proposition \ref{prop:nearlytofully} each large nearly parked tree contains a giant strongly parked tree of roughly half its size. To be precise, one would need to establish the same behavior for the component of the root (which may have some outgoing flux). We refrain from doing so to keep the paper's length acceptable.  

\subsection{The freezer as a L\'evy-type process}

\label{sec:markovfreezer} \label{sec:scalingjump}

In this section we give an alternative description of the process $ \| \FM( \cdot)\|_{\bluecirc}$  by ``passing Corollary \ref{prop:freeforest} to the scaling limit''. Since we shall not use this in the rest of the paper, the proofs are only sketched  and this section can be skipped at first reading. Recall the function $g_{x, \lambda}(y)$ from \eqref{def:gxy}. 

\begin{proposition}[A pure-jump description of $\| \FM( \lambda)\|_{\bluecirc}$] \label{prop:fellerX} The process $ \lambda \mapsto \| \FM( \lambda)\|_{\bluecirc}$ is a Markov Feller processs with inhomogeneous jump measure $$ \mathbf{n}_{\lambda}( x,  \mathrm{d}y) = \frac{1}{2} \frac{ \mathrm{d}y}{ \sqrt{2\pi y^3}} g_{x, \lambda}(y)$$ 
started ``from $0$ at time $-\infty$''.
\end{proposition}

Heuristically, this means that the process $\lambda \mapsto  \| \FM( \lambda)\|_{\bluecirc}$ has no drift, no Brownian part and jumps according to a modification of the (infinite) measure $ y^{-3/2}{ \mathrm{d}y} \mathbf{1}_{y>0}$ depending on  time $\lambda$ and location $x$. This is an example of a so-called L\'evy-type process (quite simple in our case since we only have positive jumps) we refer to the monograph \cite{bottcher2013levy} for survey. We shall rather see it as the solution of a pure-jump stochastic differential equation driven by some Poisson measure. 

\proof[Sketch of proof.]  Let us first see why we can define a Feller Markov process $ \mathscr{P}$  with the above jump kernel  over a time interval $ [\lambda_{0}, \lambda_{1}] \subset  \mathbb{R}$ starting from the initial value $x_{0} \geq 0$ at time $\lambda_{0}$. To do this, we consider a  Poisson point process $\Pi$ over $\mathbb{R}_{+} \times \mathbb{R}_{+} \times [\lambda_{0}, \lambda_{1}]$ with (infinite) intensity
$$   \frac{1}{2} \frac{ \mathrm{d}y}{ \sqrt{2\pi y^{3}}}\cdot  \mathrm{dz} \mathbf{1}_{ z \geq 0} \cdot  \mathrm{d} \lambda.$$ 
We then consider the solution $ \mathscr{P}$ to a pure-jump stochastic differential equation driven by $\Pi$, obtained by starting from $x_{0}$ at time $\lambda_{0}$ and from every atom $(y,z,\lambda)$ of $\Pi$, the process $ \mathscr{P}$ has a jump of height $y$ at time $\lambda^{-}$ i.e.\ $ \mathscr{P}_{\lambda} = \mathscr{P}_{\lambda-} + y$ if 
$$z \leq g_{\mathscr{P}_{\lambda^{-}},\lambda}(y),$$ so that the jump kernel is indeed given by $ \mathbf{n}_{\lambda}(x, \mathrm{d}y)$. We now verify the usual Lipschitz conditions so that strong solution and pathwise uniqueness holds. For this, we shall first gather a few remarks on the function $p_{1}$:
\begin{description}
\item{(P1)} The function $ x \mapsto p_{1}(x)$ is bimodal: increasing from $-\infty$ to $ x_{\max} \approx -0.886$ and then decreasing from then on. 
\item{(P2)} For all $\lambda \in [\lambda_{0}, \lambda_{1}]$ and $x,y \geq 0$, the ratio $p_{1}(\lambda-x-y)/p_{1}(\lambda-x)$ is bounded by a constant $C = 1 \wedge( p_1 (x_{\max})/p_1 (\lambda_1))>0$ depending only on $\lambda_{1}$.
\item{(P3)} The function $g_{x,\lambda}(y)$ is a smooth function of any of its variable $x \geq 0, y \geq 0$ and $\lambda \in \mathbb{R}$.
\end{description}
Those properties are easily proven using a Math software such as Mathematica or Maple. In particular, using $(P2)$ we see that 
$$ \forall \lambda \in [\lambda_{0}, \lambda_{1}], \forall x \geq 0, \quad  \int_{ 0}^{1} y \cdot \mathbf{n}_{\lambda}(x, \mathrm{d}y) \leq c_{1}(1+x),$$ for some $c_{1}>0$ depending on $\lambda_{1}$ only so that the ``linear growth condition is satisfied" and the process does not explode in finite time.  By $(P3)$, it follows that for any $A>0$ we have 
$$\forall \lambda \in [\lambda_{0}, \lambda_{1}], \forall x,x'  \in [0,A], \quad \int_{0}^{1} y \cdot \frac{ \mathrm{d}y}{ \sqrt{2 \pi y^{3}}} |g_{x,\lambda}(y)-g_{x',\lambda}(y)| \leq c_{2} |x-x'|,$$ for some $c_{2}>0$ depending on $\lambda_{1}$ and $A$. We are thus in the classical Lipschitz and linear growth condition so that we have strong solution and pathwise uniqueness for $ \mathscr{P}$, see \cite[Theorem 9.1, page 245]{ikeda2014stochastic} or \cite[Chap III.2.c, page 155]{JS03}. It is also easy to check that the resulting process is a Feller Markov process $ \mathscr{P}$. 
Furthermore, the process $ \mathscr{P}$ is the scaling limit of the chain $ \|F(n,m)\|_{\bluecirc}$ in the sense that if we start the Markov chain $( \|F(n,m)\|_{\bluecirc}, \|F(n,m)\|_{\edge})$ from $m =  \frac{n}{2} + \frac{\lambda_{0}}{2}n^{2/3}$ with   $\|F(n,m)\|_{\bluecirc} = x_{0}n^{2/3}$ and $m - \|F(n,m)\|_{\edge} = o(n^{2/3})$ then
 \begin{eqnarray} \label{eq:cvdiscretecontinuous} (n^{-2/3} \mathrm{F}_{n}(\lambda) : \lambda \in [\lambda_{0}, \lambda_{1}]) \xrightarrow[n\to\infty]{} ( \mathscr{P}(\lambda)  : \lambda \in [\lambda_{0}, \lambda_{1}]) \mbox{ with } \mathscr{P}(\lambda_{0}) = x_{0}  \end{eqnarray} in the sense of Skorokhod. Indeed, the asymptotics \eqref{def:gxy} shows that the jump kernels of $(n^{-2/3}\| \mathrm{F}_n(\lambda)\|_{\bluecirc}, n^{-2/3} \mathrm{D}_n(\lambda))$ converge towards $(\mathbf{n}_\lambda(x, \mathrm{d}y), \mathbf{0})$ and for any $ \varepsilon>0$, all $m= \frac{n}{2} + \frac{\lambda}{2}n^{2/3}$ for  $ \lambda \in [\lambda_0, \lambda_1]$  and $n$ large enough
$$ \mathbb{E} \left[  \begin{array}{rcl}\min(\Delta \|F(n,m+1)\|_{\bluecirc}, \varepsilon n^{2/3})  \\ \Delta(m- \|F(n,m)\|_{\edge}) \end{array}  \left|  \begin{array}{l}\|F(n,m)\|_{\bluecirc} = xn^{2/3} \\  m-\|F(n,m)\|_{\edge}\leq n^{2/3} \end{array}\right]  \right. \leq  \left( \begin{array}{c} C \sqrt{ \varepsilon}\\ 2(1+x) n^{-1/3} \end{array}\right), $$ for some constant $C>0$ independent of $n$. The convergence \eqref{eq:cvdiscretecontinuous} is then a consequence of general convergence results on Feller processes, see \cite[Chapter 19]{Kal07} or  \cite[Chapter IX, 4]{JS03}. We leave the verifications to the reader.

 Finally, let us see why we can define the Feller process $ \mathscr{P}$ by starting from $0$ at time $-\infty$. To prove this, we need to show convergence of $ \mathscr{P}$ at a fixed time, say $\lambda=0$, when $ \mathscr{P}$ is started from $0$ at a very negative time $\lambda_{0} \ll 0$. This can be deduced by rather tedious calculations using $ \mathbf{n}_{\lambda}$ and asymptotics of $p_{1}$ but let us sketch another route using our cutoff construction of Section \ref{sec:gettingridoldcycles}. Specifically, recall the construction of the process $ \mathrm{F}_{n}^{[\lambda_{0}]}(\lambda)$ obtained by throwing all components with cycles in $ \mathrm{G}_{n}(\lambda_{0})$ and starting the construction of the frozen process from there. We shall use a variant of this construction by considering a random stopping time $\Lambda_{0}$ (with an implicit dependence in $n$) associated to $M_{0} = \frac{n}{2} + \frac{\Lambda_{0}}{2} n^{2/3}$ defined as follows
$$ M_{0} = \inf\left\{ m \geq 0 : 2\left(m - \frac{n}{2}\right) -  \|G(n,m)\|_{\bluecirc} \geq \lambda_{0} n^{2/3}\right\}.$$ In words, $M_0$ is the first instant $m$ where the felt time-parameter in the forest part $[G(n,m)]_{ \mathrm{tree}}$ is above $\lambda_0$.  We first claim that for $\lambda_{0}$ negative enough, $M_{0} \leq n/2$ and is actually close to $ \frac{n}{2} + \frac{\lambda_{0}}{2} n^{2/3}$ with high probability: indeed it follows from \eqref{eq:surplus} that $n^{-2/3}\cdot \| \mathrm{G}_{n}(\lambda_{0})\|_{\bluecirc}$ is of order $\lambda_{0}^{-2}$ and so the function $m \mapsto 2\left(m - \frac{n}{2}\right) -  \|G(n,m)\|_{\bluecirc} \geq \lambda_{0} n^{2/3}$ crosses $\lambda_{0}n^{2/3}$ around time $\lambda_{0} \pm \lambda_{0}^{-2}$. On this event,  by Proposition \ref{prop:freeforest}, the process $ \mathrm{F}_{n}^{[\Lambda_{0}]}( \cdot +\Lambda_{0})$ has the same transitions as $F$ started from $0$ at time $m' =  \frac{n'}{2} + \frac{\lambda_{0}}{2} n'^{2/3} + o(n^{2/3})$ over $n' = n - \|G(n,M_0)\|_{\bluecirc}$ vertices with a slight time change coming from the fact that certain edges are discarded (this does not persist in the limit). So by \eqref{eq:cvdiscretecontinuous} it converges after scaling towards the process $ \mathscr{P}(\cdot + \lambda_{0})$ started from $0$ at time $\lambda_{0}$. The convergence of $ n^{-2/3} \cdot (\mathrm{F}_{n}^{[\Lambda_{0}]}(\lambda) : \lambda \geq 0 )$ proved in Section \ref{sec:gettingridoldcycles} together with the fact that $ n^{-2/3} \| \mathrm{G}_n( \Lambda_0)\|_{ \bluecirc} \to 0$ as $\lambda_0 \to \infty$ and the above convergence imply that the law of the process $ \mathscr{P}$ started from $0$ at time $\lambda_0$ does converge as $\lambda_0 \to \infty$ and this enables us to start $ \mathscr{P}$ from $0$ at time $-\infty$. Combining those observations we deduce that the process $ \mathscr{P}$ started from $0$ at time $-\infty$ has the same law as $ \| \FM( \cdot)\|_{\bluecirc}$. We leave the many details to the fearless reader. 
\subsection{Scaling limit of random forest}
In this section we revisit the result of Martin and Yeo \cite{martin2018critical} to complete the Markovian description of the scaling limit of the frozen multiplicative coalescent. As for the preceding section, the results are not used in the rest of the paper and so this part may be skipped at first reading.\medskip

\label{sec:scalingforest}

As in the proof of Corollary \ref{prop:freeforest}, for $ n \geq 1$ and $m \geq 0$ we denote by $ W(n,m) \in \mathfrak{F}(n,m)$ a uniform random forest over the $n$ labeled vertices $\{1,2,\dots , n\}$ with $m$ edges in total. We chose the letter $W$ for the German ``Wald'' because there are already too many f's in the paper :) In particular, the forest $W(n,m)$ has $n-m$ components. Although there is \emph{no} obvious coupling of $W(n,m)$ for varying $m \geq 0$ (see \cite[Section 1.4.2]{martin2018critical}), we shall use our usual notation \eqref{eq:notationcriticalwindow} and write $ \mathrm{W}_{n}(\lambda)$ for a random forest with $m =  \lfloor \frac{n}{2} + \frac{\lambda}{2} n^{2/3} \rfloor$ edges and by  $  \mathbb{W}_{n}(\lambda) \in \ell^{\infty}_{\downarrow,0}$ the renormalized sequence of its component sizes in non-increasing order.  

Recall from Section \ref{sec:stabledef} that  $( \mathscr{S}_{t})_{t \geq 0}$ denotes the stable L\'evy process with index $3/2$ and only positive jumps,  which starts from $0$ and normalized so that  $ \mathbb{E}[\exp( - \ell  \mathscr{S}_{t})] = \exp(  \tfrac{2^{3/2}}{3}t \ell^{3/2})$ for any $\ell,t \geq 0$, see Figure \ref{fig:stableproc} for a simulation. The density of $ \mathscr{S}_{t}$ is $p_{t}(\cdot)$ for $t >0$. For any $\lambda \in \mathbb{R}$ we can use this function to define the process $(\mathscr{S}^{\lambda}_{t}: 0 \leq t \leq 1)$ called the $(0,0) \to (1,\lambda)$ bridge, obtained by conditioning $(\mathscr{S}_{t} : 0 \leq t \leq 1)$ to be equal to $\lambda$ at time $1$. Of course this is a degenerate conditioning, but it can be obtained by performing an inhomogeneous $h$-transform with respect to the function
$$ \frac{p_{1-t}(\lambda-\mathscr{S}_{t})}{ p_{1}(\lambda)},$$ see \cite[Theorem 4]{liggett1968invariance}. 
\begin{proposition}[Another route towards critical random forests] \label{prop:martinyeo} Fix $\lambda \in \mathbb{R}$. For all $ \varepsilon>0$, we have the following convergence  in distribution for the $\ell^{3/2 + \varepsilon}_{\downarrow}$- topology
 \begin{eqnarray} \label{eq:convUnifForest}  \mathbb{W}_{n}(\lambda) \xrightarrow[n\to\infty]{(d)}  \big( \Delta \mathscr{S}^{\lambda}_{t} : 0 \leq t \leq 1\big)^{\downarrow} \end{eqnarray}
 where $(x_{i} : i \geq 1)^{\downarrow}$ is the non-increasing rearrangement of the sequence $(x_{i} : i \geq 1)$.
\end{proposition}

\begin{figure}[!h]
 \begin{center}
 \includegraphics[width=10cm]{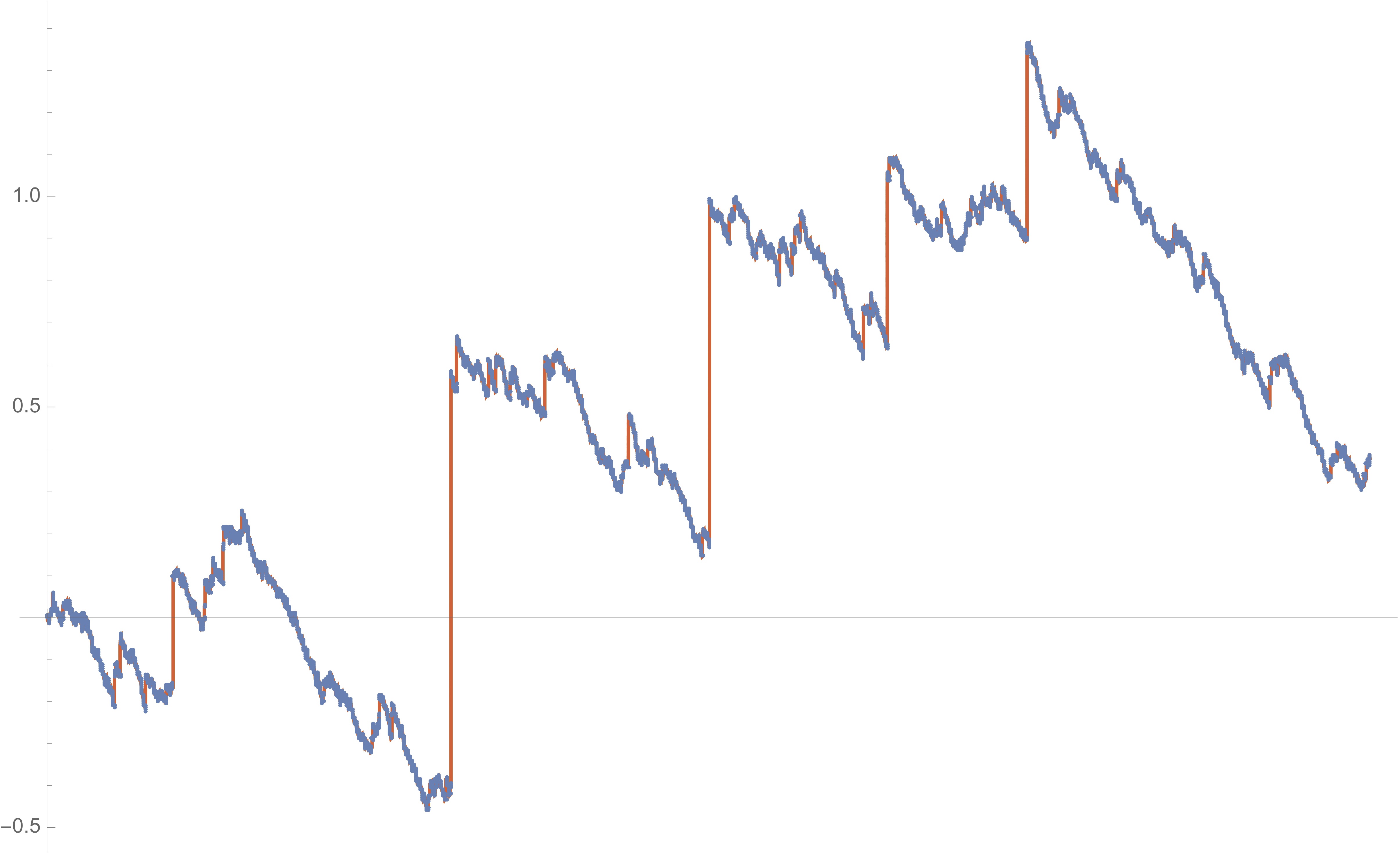}
 \caption{A simulation of a $ \tfrac{3}{2}$-stable spectrally positive L\'evy process over the time interval $[0,1]$. The jumps are displayed in orange. \label{fig:stableproc}}
 \end{center}
 \end{figure}
\proof 
 From Proposition \ref{prop:RWcoding}, the sizes of the components in the random forest $W(n,m)$ has the same law as the increments $+2$ of a random walk $( \tilde{S}_{i} : 0 \leq i \leq n-m)$ started from $0$, conditioned to hit $n-2(n-m) = \lambda n^{2/3}+o(n^{2/3})$ at time $n-m$ and with independent increments of law $\mu(k+2)$ for $k \in \{-1,0,1,2, \dots \}$ introduced in  \eqref{eq:defmu}.  Recall that the variable $ \tilde{S}_{1}$ is centered and in the domain of attraction of the $3/2$-stable spectrally positive random variable. The following convergence for the Skorokhod topology on $  \mathbb{C} \mathrm{adlag}([0,1], \mathbb{R})$ follows from the conditional invariance principle of Liggett \cite{liggett1968invariance}  
$$   \left(  n^{-2/3} \tilde{S}_{[(n-m) t]} : 0 \leq t \leq 1 \right) \xrightarrow[n\to\infty]{(d)}  \left(  \mathscr{S}^{\lambda}_{t} : 0 \leq t \leq 1 \right).$$
In particular by \cite[Corollary 2.8]{JS03}, the random point measure  $\sum_{ 0 \leq t \leq 1} \delta_{n^{-2/3}\cdot \Delta \tilde{S}_{[(n-m) t]}}$ converges weakly towards $\sum_{0 \leq t \leq 1} \delta_{\Delta  \mathscr{S}^{\lambda}_{t}}$ from which we deduce the convergence \eqref{eq:convUnifForest} for the $\ell_{\downarrow,0}^{\infty}$ topology. To bootstrap this into a convergence for the $\ell^{3/2 + \varepsilon}_{\downarrow}$ topology, it suffices to establish tightness in the later (since convergence in $\ell^{\infty}_{\downarrow,0}$ characterizes the limit in $\ell^{3/2 + \varepsilon}_{\downarrow}$). For this we claim that it is sufficient to prove that 
 \begin{eqnarray} \label{eq:goalesptightness} \sup_{n \geq 1} \mathbb{E}\left[\sum_{i=0}^{n-m-1} \left( n^{-2/3} \cdot \Delta {S}_{i} \right)^{3/2 +  \varepsilon} \Big| {S}_{n-m}=  n \right] < \infty,  \end{eqnarray} where $S$ has increments of law $\mu(k)$ given by \eqref{eq:defmu}. 
Indeed, for $ \xi >0$ and  $\ \varepsilon>0$
$$ \sum_{i \geq 1} \left( n^{-2/3} \cdot \Delta {S}_{i} \right)^{3/2 + 2 \varepsilon} \mathbf{1}_{n^{-2/3} \cdot \Delta {S}_{i} > \xi}  \leq \xi ^{  \varepsilon} \sum_{i = 0}^{n-m-1} \left( n^{-2/3} \cdot \Delta {S}_{i} \right)^{3/2 +  \varepsilon},$$ so that using \eqref{eq:goalesptightness} the expectation of the right-hand side of the last display can be made arbitrarily small by choosing $\xi$ small. Combining this with the tightness in  $\ell^{\infty}_{\downarrow, 0}$, it is easy to deduce tightness in $\ell_\downarrow^{3/2+ 2  \varepsilon}$ of $\left(  \mathbb{W}_{n}(\lambda) : n \geq 1 \right)$. 
To prove our claim, notice that by cyclic exchangeability we have 
 \begin{eqnarray*}\mathbb{E}\left[\sum_{i=0}^{n-m-1} \left( n^{-2/3} \cdot \Delta {S}_{i} \right)^{3/2 + \varepsilon} \Big| {S}_{n-m}=  n \right] &=& (n-m) \sum_{k \geq 1} \mu(k) \left| k n^{-2/3}\right|^{3/2 + \varepsilon} \frac{ \mathbb{P}(S_{n-m-1}=n-k)}{\mathbb{P}(S_{n-m}=n)}.  \end{eqnarray*}
Let us focus first on the $k$'s such that $k \ll n$. In this case, writing $k = y n^{2/3}$ and using \eqref{eq:defmu} and \eqref{eq:asymtptoPSnm}, we deduce that there are constants  $C,C'>0$ (which may depend on our fixed $\lambda$) such that the last display is bounded above by  
$$ C n^{-2/3} \sum_{k=1}^{\infty} (kn^{{-2/3}})^{-1+ \varepsilon}\cdot  \frac{p_{1}(\lambda-kn^{{-2/3}})}{p_{1}(\lambda)} \leq  C' \int_{0}^{\infty} y^{-1 + \varepsilon} \ \mathrm{d}y   < \infty.$$
On the other hand, if $k$ if of order $n$, 
rough large deviations estimates show that $\frac{ \mathbb{P}(S_{n-m-1}=n-k)}{\mathbb{P}(S_{n-m}=n)}$ is exponentially small (in $k$ and so in $n$) so that the contribution to the sum is negligible. This finishes the proof of \eqref{eq:goalesptightness} and of the proposition. \endproof 

As an application of this methodology, let us revisit a few of the results of \cite{luczak1992components} discussed in \cite[Section 1.4.3]{martin2018critical}. Consider a slightly supercritical random forest $ W(n,m)$ with $m = \frac{n}{2} +\frac{s}{2}$ with $ n^{2/3} \ll s \ll n$  whose component sizes are coded by $(\Delta{\tilde{S}}_{i}+2: 0 \leq n-m-1)$ conditioned on $\tilde{S}_{n-m}  = s$. According to standard ``big-jump" principles, since $\mu$ is a subexponential distribution such a random walk has a unique ``big-jump" of height of order $  s$ and once this jump has been removed, the remaining random walk is close in total variation distance to an unconditioned $\mu$-random walk, see \cite[Theorem 1]{armendariz2011conditional}. This gives another way to prove that the largest cluster in $ W(n,m)$ is  of size $(1+o(1))s$  and the remaining components converge after normalization by $n^{2/3}$ to the jumps of the \emph{unconditioned} L\'evy process $(\mathscr{S}_t : 0 \leq t \leq 1)$.

\part{Comments and perspectives}\label{sec:comments}
We end this paper by presenting several research directions and connections of our work. This part is informal and we do not claim any mathematical statement. We first draw a parallel between the enumeration of (strongly, fully or nearly) parked trees and random planar maps which gives another support for Conjecture \ref{conjectureGFT}. We then present a few fallouts of the study of (generalized) frozen process $F(n,\cdot)$ on the Erd{\H{o}}s--R\'enyi random graph $G(n,\cdot)$. We end with extensions of our work concerning the parking process on random trees.

\section{Links with planar maps and growth-fragmentation trees}
\label{sec:GF}
We shall consider strongly parked trees with \emph{outgoing flux}. More precisely, for $n, p \geq 0$ we denote by  $ \mathrm{SP}(n, n+p)$ the number of labeled rooted Cayley trees of size $n$ with $n+p$ labeled cars so that after parking, all edges have a positive flux and exactly $p$ cars exit the tree.  We encode these numbers into the generating function 
$$\mathbf{S}(x,y) = \sum_{n \geq 1, p \geq 0} \frac{\mathrm{SP}(n,n+p)}{n! (n+p)!} x^{n} y^{p},$$ which replaces the univariate generating function $ \mathbf{S}(x)$ which we considered in Section \ref{sec:enumconsq}. In particular since King \& Yan \cite{king2019prime} computed $ \mathrm{SP}(n, n) = (2n-2)!$ we have (Proposition \ref{prop:kingyan}) that  \begin{eqnarray} \label{eq:S(0,x)}\mathbf{S}(x,0) = 1- \ln (2) -\sqrt{1 - 4 x}  + \ln \left(1 + \sqrt{1 - 4 x}\right), \quad \mbox{ for } 0 \leq x \leq x_{c}= \frac{1}{4}.  \end{eqnarray}

\subsection{Tutte's  equation} To get a functional equation on $ \mathbf{S}$ one considers the decomposition of strongly parked trees at the root vertex (see Figure \ref{decomposition-strong-tutte} left) which shows that $ \mathrm{SP}(n,n+p)$ is equal to 

 \begin{eqnarray*} \label{eq:tutteSP} \sum_{a \geq 0} \sum_{k\geq 0} \sum_{\begin{subarray}{c} n_{1}, \dots, n_{k} \geq 1\\ p_{1}, \dots , p_{k} \geq 1 \end{subarray}} \frac{1}{k!} { n \choose 1,n_{1}, n_{2}, \dots,  n_{k}} {n+p \choose a, n_{1}+p_{1},  \dots, n_{k}+p_{k}} \prod_{i=1}^{k} \mathrm{SP}(n_{i},n_{i}+p_{i}) \mathbf{1}_{ \begin{subarray}{l} n_{1}+ \dots + n_{k} = n-1 \\ a + p_{1}+ \dots + p_{k} = p+1. \end{subarray}}  \end{eqnarray*}
Indeed, the integer $a$ counts the number of cars arriving at the root, the integer $k$ is the number of children of the root vertex and $n_{i}, p_{i}$ are the characteristics (number of vertices and {outgoing} flux) of the subtrees above it. This equation translates into the following equation on $\mathbf{S}$
 \begin{eqnarray} \label{eq:tutteS}  \mathbf{S}(x,y) = \frac{x}{y}\left( \mathrm{e}^{y} \mathrm{e}^{ \mathbf{S}(x,y)-  \mathbf{S}(x,0)} - 1\right).   \end{eqnarray}

At first sight, one may think that the series $ \mathbf{S}(x,0)$ is a necessary input (which we do have) to solve the equation, but a close inspection shows that the equation actually determines the coefficients of $ \mathbf{S}$ by induction on $n+p$. 

This type of equation is very common in the map enumeration literature where they are called ``Tutte'' equations, see \cite{bousquet2008rational} for a comprehensive survey. More precisely, recall that a map is a planar graph properly embedded in the plane given with one distinguished oriented edge. Following Tutte, when enumerating (various classes of) planar maps by their size $n$, it is convenient to introduce an external parameter $p$, the perimeter of the external face (lying on the right of the root edge). When performing the root erasure, certain situations  yield a splitting of a map of size $n$ and perimeter $p$ into two components of size $n_{1}$ and $n_{2}$ having perimeter $p_{1}$ and $p_{2}$ so that we have (on a high level) $n_{1}+n_{2} \approx n$ and $p_{1}+p_{2} \approx p$ which is similar to penultimate equation above, see Figure \ref{decomposition-strong-tutte} right. 
Similar equations arose in \cite{CoriSchaefferDescription,duchi2017fightingbis}.

\begin{figure}[!h]
 \begin{center}
 \includegraphics[height=8cm]{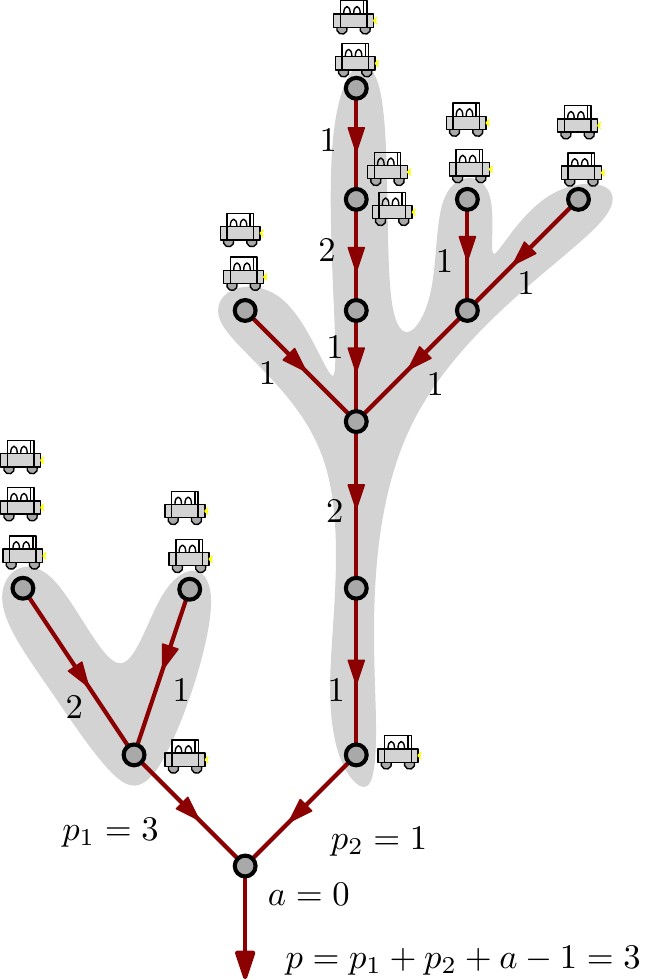} \hspace{2cm}
  \includegraphics[height=5cm]{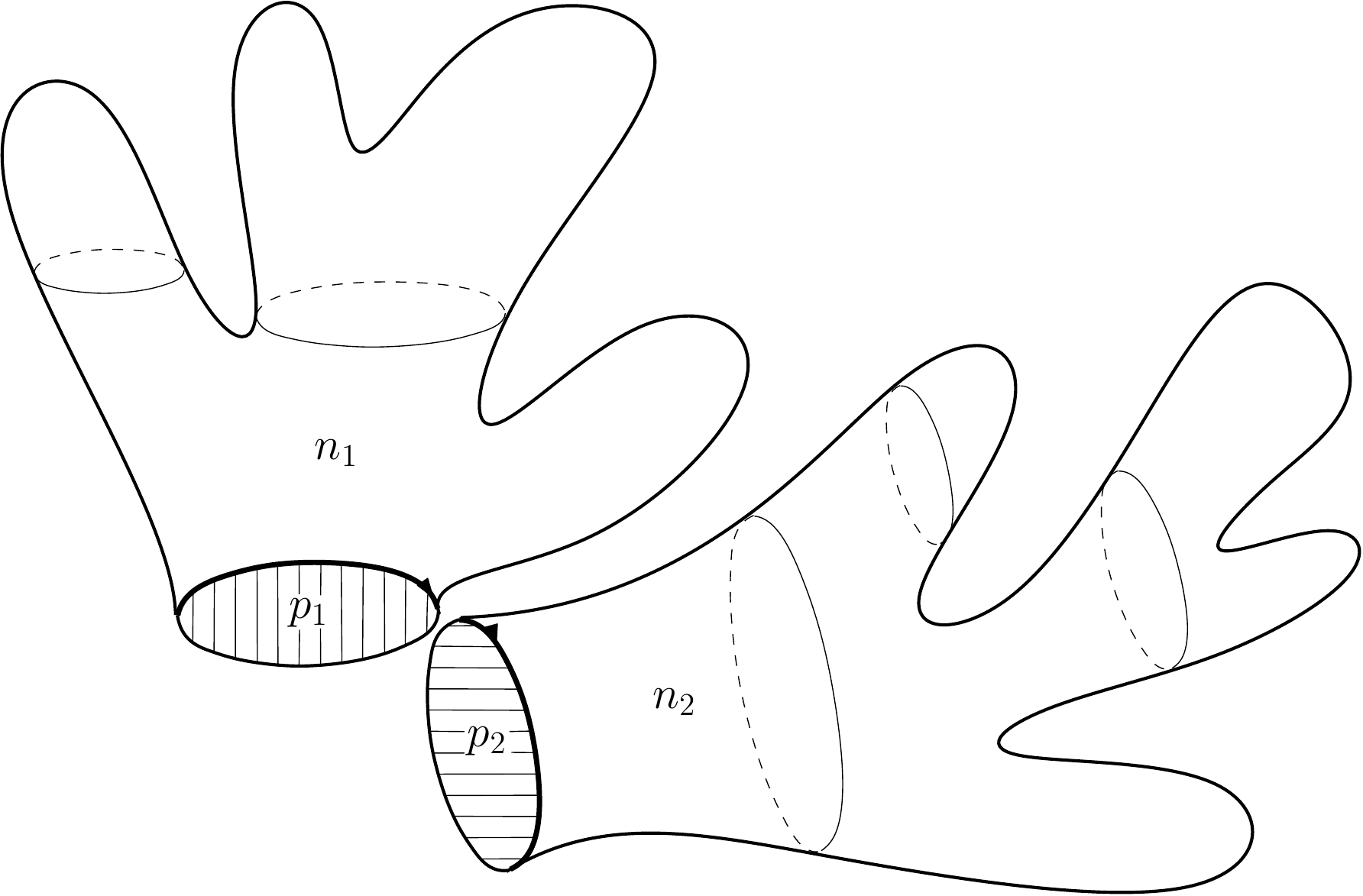}
 \caption{ \label{decomposition-strong-tutte}Left: Illustration  of the recursive decomposition at the root of strongly parked trees to get a functional equation on $ \mathbf{S}$. Right: heuristic representation of Tutte's equation in the theory of planar map enumeration.}
 \end{center}
 \end{figure}

In our case, Equation \eqref{eq:tutteS} can be solved using the Lambert function. Fix $(x,y)$ and observe that  $  \mathbf{S}=\mathbf{S}(x,y)$ is solution to an equation of the form $a  \mathrm{e}^{ \mathbf{S}} + b \mathbf{S} +c =0$. If we put  $$ \Delta = -  \exp \left(y - \mathbf{S}(x,0) -  \frac{x}{y}\right)  \frac{x}{y}  \leq 0$$ when $x \geq 0$ and $ y> 0$, then the above equation has solutions if  $\Delta \geq - \mathrm{e}^{-1}$ which are
$$- W_{i}( \Delta) - \frac{x}{y} $$
where $W_{i}$ is the $i$th branch of the Lambert function. There is actually a singularity and we need to change branch (see Figure \ref{fig:plotlambert}), more precisely, when $x < x_{c} =  1/4$ we have  
 \begin{eqnarray} \label{eq:solutionS}
  \mathbf{S}(x,y) = \left\{ \begin{array}{ll}
  - W_{-1}( \Delta) - \frac{x}{y} & \mbox{ if } y \leq \frac{1}{2} \left( 1 - \sqrt{1 - 4 x}\right) \\ 
    - W_{0}( \Delta) - \frac{x}{y} & \mbox{ if }  \left( 1 - \sqrt{1 - 4 x}\right) \leq  y \leq y_{c}(x),   \end{array} \right.
 \end{eqnarray} 
where $y_{c}(x)$ is the radius of convergence of the series in $y$ when $x$ is fixed which is the maximal solution of $\Delta = -  \mathrm{e}^{-1}$.
At $x= x_c = 1/4$, then $\Delta +  \mathrm{e}^{-1}$ vanishes  at $ y_c = y_{c}(1/4)=1/2$ and yields a singularity of type $(y-y_{c})^{3/2}$.
 
 \begin{figure}[!h]
  \begin{center}
  \includegraphics[width=8cm]{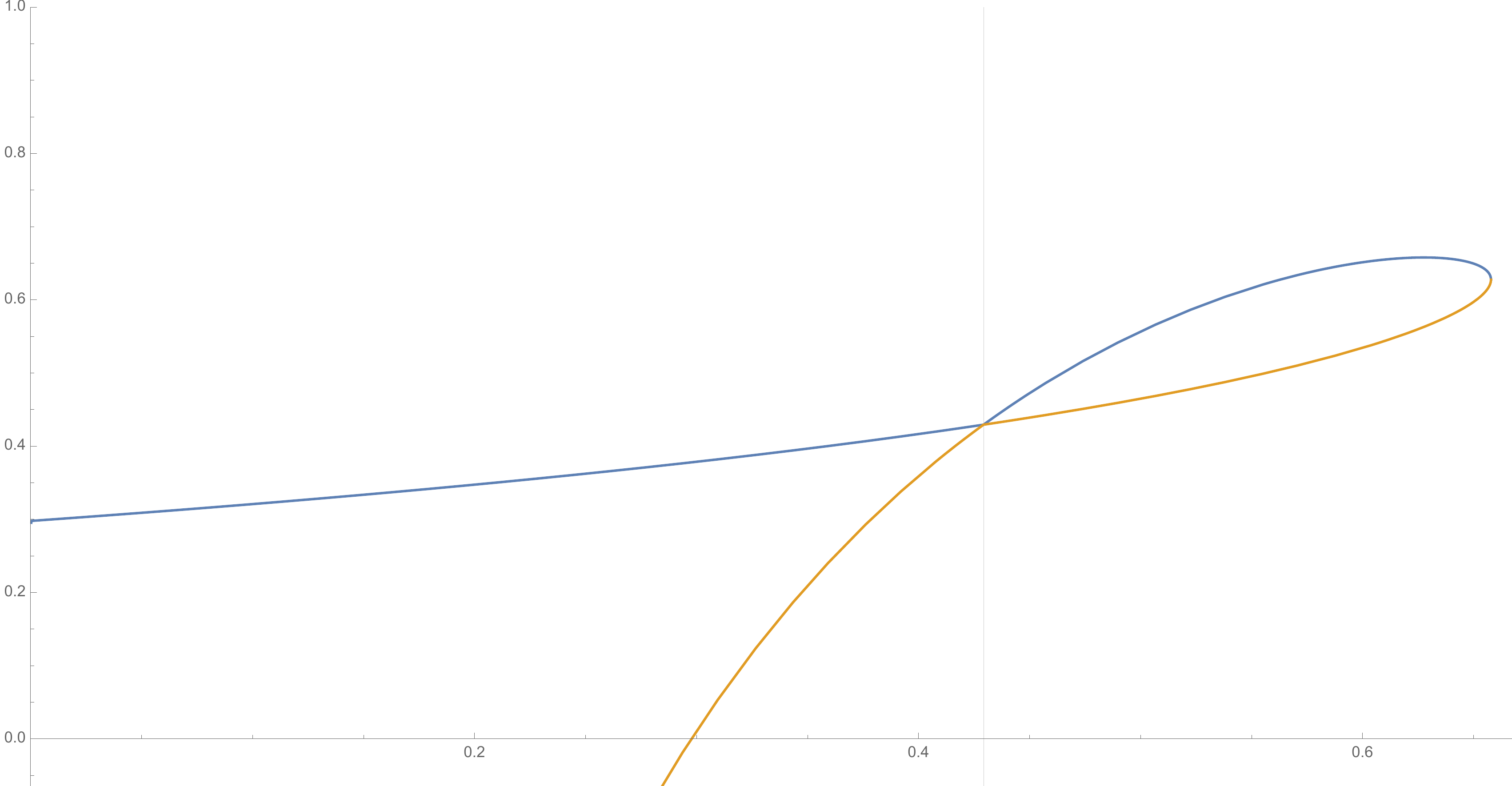}
  \caption{Plot of the function $ y \mapsto \mathbf{S}(x,y)$ for $x = 0.245$. To get an analytic function, $ \mathbf{S}$ changes from the blue to the orange branch at $y \approx 0.41$. \label{fig:plotlambert}}
  \end{center}
  \end{figure}

\subsection{Lackner \& Panholzer's decomposition and the KP hierarchy?} Actually, there is another completely different way to get a functional equation on $ \mathbf{S}$. Adapting an idea of  \cite{LaP16} (see  in particular Equation (4) there) one can decompose a strongly parked tree according to the travel of the car labeled $n+p$ (the last car) in a sequence of  strongly parked trees each given with a distinguished point, see Figure \ref{fig:decompos-strong-last}. 
This last car decomposition yields to the following equation on $ \mathbf{S}(x,y)$:
  \begin{eqnarray} \label{eq:lastcar}y \partial_{y} \mathbf{S}(x,y) + x \partial_{x} \mathbf{S}(x,y) - \mathbf{S}^{\bullet}(x,0) = \frac{ x y \partial_{x} \mathbf{S}(x,y)}{1 - \mathbf{S}^{\bullet} (x,0)},  \end{eqnarray} where $ \mathbf{S}^{\bullet} (x,y) = x \partial_{x} \mathbf{S} (x,y)$ is the generating series of strongly parked trees with an additional distinguished vertex. It should be possible to solve the above equation using the method of characteristics to recover \eqref{eq:solutionS} but we have not been able to carry the calculations. 
\begin{figure}[!h]
 \begin{center}
 \includegraphics[width=14cm]{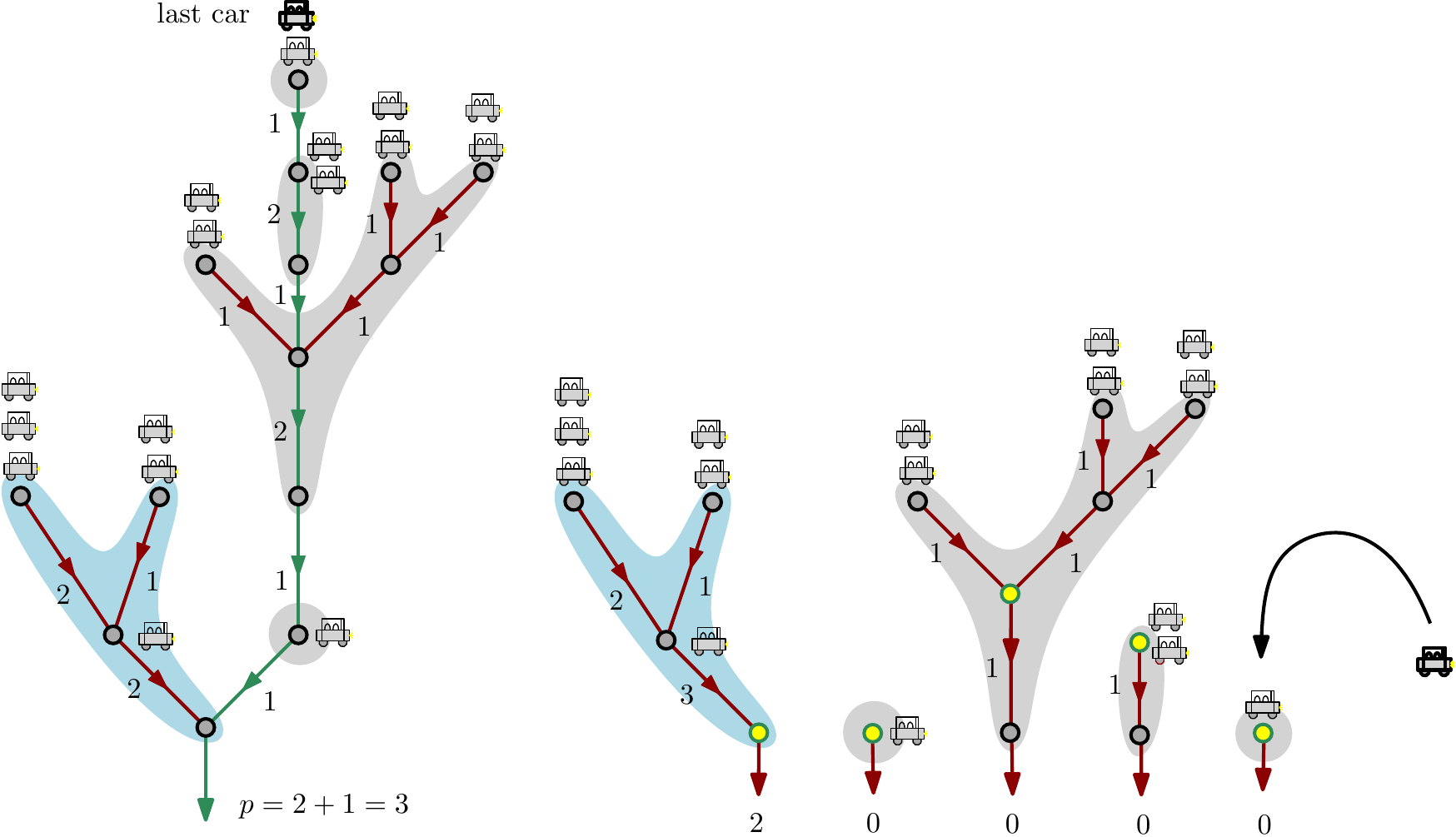}
 \caption{ \label{fig:decompos-strong-last}Illustration of the decomposition of a strongly parked tree according to the ride of the last car. If we remove the edges through which the last car had to go, then we end up with a sequence of strongly parked trees with distinguished points, where the last of those may have a positive flux at the root.}
 \end{center}
 \end{figure}
Also, applying the last car decomposition to the case $y=0$ (no flux) one finds the equation
 $$ \textbf{S}^{\bullet}(x,0) = \frac{x}{1-\textbf{S}^{\bullet} (x,0)}$$
involving $ \mathbf{S}(x,0)$ only and enables us to recover \eqref{eq:S(0,x)} very easily. In the theory of planar maps, there are similar inductive decompositions of planar maps of size $n$ involving planar maps of size $n_{1}$ and $n_{2}$ with $n_{1}+ n_{2} \approx n$ \emph{without} boundary. Those decompositions are obtained via the KP hierarchy, see \cite{GJ08} or \cite[Corollary 1]{louf2019new} for details. We plan on adapting the ``last car decomposition'' of Lackner \& Panholzer to the enumeration of random planar maps.

We expect that the information on the bivariate generating function $ \mathbf{S}(x,y)$ will enable us to perform the asymptotic enumeration of strongly parked trees with flux and prove similar results as in the planar map setting, i.e.
 \begin{eqnarray} \frac{\mathrm{SP}(n,n+p)}{n! (n+p)!} &\sim& c_1 \cdot 4^n\cdot 2^p\cdot n^{-5/2}\cdot p^{1/2} \exp\left(-c_2\frac{p^2}{n}\right)   \label{eq:enumFPT}\end{eqnarray}
for some constants $c_{1},c_{2}>0$ as long as $ p^{2}/n$ stays in a fixed compact interval of $(0,\infty)$. 
Those asymptotics are necessary to progress towards Conjecture \ref{conjectureGFT} but also would be a crucial input to compute, for fixed $\lambda \in \mathbb{R}$, the exact distribution of $ (  \| \FM( \lambda)\|_{\bluecirc}, \mathscr{D}(\lambda))$. The behavior \eqref{eq:enumFPT}  has been observed in a great generality for a related model  \cite{chen2021enumeration} which we now describe.

\paragraph{Chen's and Panholzer's generalizations of  fully parked trees.}\label{sec:linkchen}  Panholzer \cite{panholzer2020parking} studies the model of fully parked trees (with no flux) when the underlying Cayley tree is replaced by a combinatorial model such as $d$-ary trees, ordered trees... and he obtains remarkable explicit formulas. In particular, he finds connections with models of planar maps (OEIS  A000139, or OEIS A000260 via OEIS A294084), see Remark 2 in \cite{panholzer2020parking}. It is natural to extend the above discussion to those models. 
In \cite{chen2021enumeration}, Chen considers the enumeration of \emph{plane} trees (rather than Cayley trees) decorated with i.i.d.\ (not necessarily Poisson) car arrivals conditioned to be fully parked and with a possible flux at the root. He proves a phase transition for the enumeration of such structures appearing at the same location as the phase transition for the parking process \cite{CH19}. In the case of bounded car arrivals, he proves in \cite[Theorem 3]{chen2021enumeration} an asymptotic enumeration of plane fully parked trees with flux of the form \eqref{eq:enumFPT}. This supports the belief that the parking processes are in the same universality classes, see below. 

\subsection{Growth-fragmentation trees and conjectural scaling limits.} \label{sec:GFcomments}
Let us now  give some background for  Conjecture  \ref{conjectureGFT}. By the discussion in the last paragraph, the generating series of strongly parked trees is finite at $x=x_c= \frac{1}{4}$. Hence, for each $p \geq 0$, {we can define as in Section \ref{sec:componentsLLN}} a random strongly parked tree $ S_p$ with flux $p$ at the root under Boltzmann critical distribution  whose law is simply 
$$ \mathbb{P}(S_p =  \mathfrak{s}_p) =  \frac{1}{ [y^p] \mathbf{S}(x_{c},y)} \frac{1}{\|\mathfrak{s}_p\|_{\bullet}!(\|\mathfrak{s}_p\|_{\bullet}+p)!} \left(\frac{1}{4}\right)^{ \|\mathfrak{s}_p\|_{\bullet}}, $$ for each strongly parked tree $\mathfrak{s}_p$ with flux $p$ at the root.  Let us   forget the labels of the vertices and the cars and see such a tree as a rooted unordered tree where each vertex is labeled by the flux of car emanating from it (so that the root has label $p$).   It is easy to see that such trees are \emph{Markov branching trees}, that is, have the same law as the family tree of a system of particles evolving independently of each other. Each particle carries a non-negative integer label (the emanating flux of car from that vertex) and at each step, a particle of label $p$ ``splits" into $k$ particles of labels $p_1,p_2, \dots , p_k$ (ordered uniformly at random)  with probability 
$$ \frac{1}{[y^{p}] \mathbf{S}(x_{c},y)}  \frac{1}{(p-(p_{1}+ \cdots + p_{k})+1)!} \frac{1}{k!} \prod_{i=1}^{k} [x^{p_{i}}] \mathbf{S}(x_{c},y).$$
When the scaling limit of the labels along a branch\footnote{to be specific, one define a branch by following the locally largest label at each splitting} is given by a positive self-similar Markov process, the scaling limit of those trees are described by the growth-fragmentation trees\footnote{to be precise, Bertoin defines a growth-fragmentation process from to which we can associate a continuum random tree by \cite[Corollary 4.2]{rembart2018recursive}} of Bertoin \cite{Ber15}. In our case of random strongly parked trees, the labels evolve in the scaling limit as (versions of) the $3/2$-stable L\'evy process, exactly as for the Markov branching trees appearing the peeling exploration of random planar maps, see \cite{BCK18} and \cite[Section 6]{BBCK18}. To be more specific, the growth-fragmentation  mechanism involved in Conjecture \ref{conjectureGFT} is the one ``canonically associated" to the spectrally positive $3/2$-stable L\'evy process i.e.\ with the cumulant function $$\kappa(q) =\frac{\Gamma(q- \tfrac{3}{2})}{\Gamma(q-3)}$$ for $q > 3/2$ and self-similarity index $\alpha=-3/2$, see \cite[Section 5]{BBCK18}. Our Conjecture \ref{conjectureGFT} concerns \emph{conditioned version} of those Markov branching trees, see the forthcoming work \cite{BertoinCurienEtAl} for details.

\section{Back to Erd{\H{o}}s--R\'enyi}
Let us now formulate a few possible consequences of our work on the classical Erd{\H{o}}s--R\'enyi random graph. For this we need to generalize a little the frozen process  by introducing a parameter $p \in [0,1]$.

\subsection{Generalized frozen process} \label{sec:generalfrozen}
Given the sequence of \emph{unoriented} edges $( E_i : i \geq 1)$ and an independent sequence of uniform random variables $ (U_i : i \geq 1)$ we construct a generalization of the frozen Erd{\H{o}}s--R\'enyi process as follows. Fix a parameter $p \in [0,1]$ and define a growing graph process $F_p(n,m)$ with two colors, white and blue, in a way very similar to $F(n,m)$: Initially $F_p(n,0)$ is made of the $n$ labeled white vertices $\{1,2, \dots , n\}$ and for $m\geq 1$ 
\begin{itemize}
\item if both endpoints of $E_{m}$ are white vertices then the edge $E_{m}$ is added to $F_p(n,m-1)$ to form $F_p(n,m)$. If this addition creates a cycle in the graph then the vertices of its component are declared frozen and colored in blue.
\item if both endpoints of $E_{m}$ are blue (frozen vertices), then $E_{m}$ is discarded.
\item if $E_{m}$ connects a white and a blue vertex, then $E_{m}$ is discarded \emph{if $ U_m > p$} and kept otherwise, in which case the new connected component is declared frozen and colored in blue.
\end{itemize}

\begin{figure}[!h]
 \begin{center}
 \includegraphics[width=15cm]{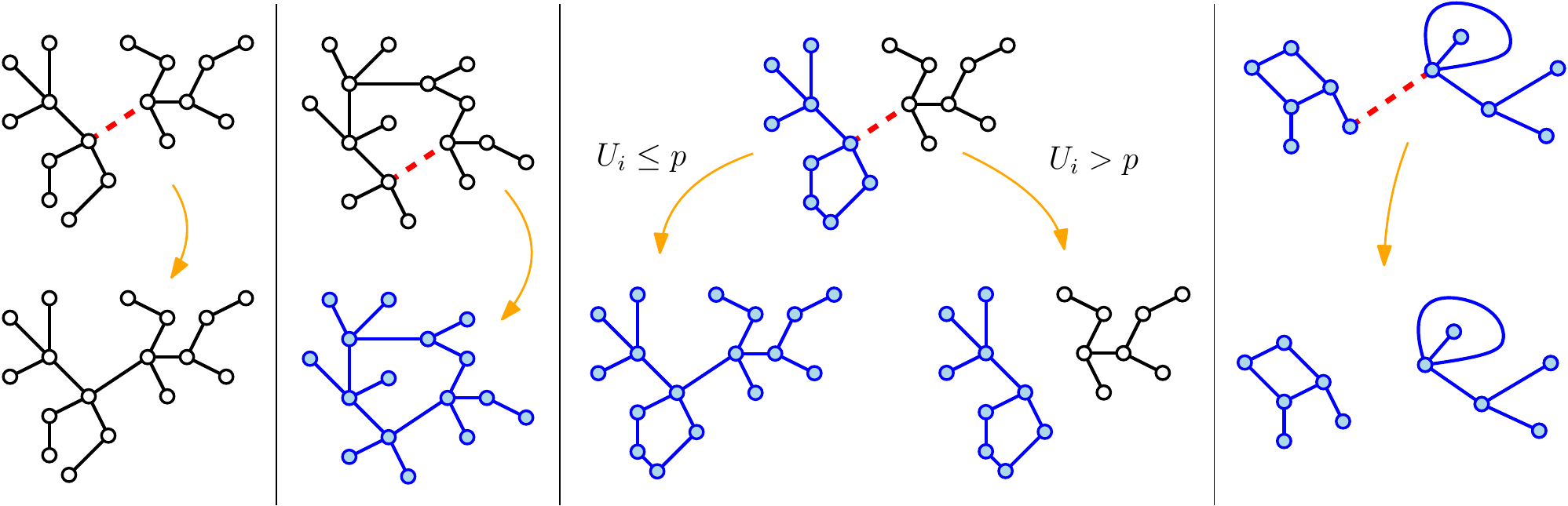}
 \caption{ \label{fig:transitionsp}Transitions in the frozen Erd{\H{o}}s-R\'enyi  with parameter $p$.}
 \end{center}
 \end{figure}
 
 Obviously, in the case $p= \frac{1}{2}$, the process $F_{1/2}$ has the same law as the frozen Erd{\H{o}}s--R\'enyi that we used in this paper\footnote{in the paper we used the orientation of the edges $ \vec{E}_{i}$ and did not require the additional randomness of the $U_{i}$.}. For $p=0$, the process corresponds to completely stopping the connected components once they create a cycle. In the case $p=1$, the process is obtained from $G(n,m)$ by discarding the edges which would create a surplus of $2$. In particular, we have \begin{eqnarray} \label{eq:linkFG}[F_{1}(n,m)]_{\circ} = [F_1(n,m)]_{ \mathrm{tree}} = [G(n,m)]_{ \mathrm{tree}}, \quad \mbox{ for all } m \geq 0. \end{eqnarray} Notice however, that the obvious coupling of $F_p$ for all $p\in [0,1]$ is not monotonic in $p$.
 
 It should be easy to extend our analysis to the frozen Erd{\H{o}}s--R\'enyi processes with parameter $p$ and in particular Theorem \ref{thm:FMC}, Propositions \ref{prop:freeforest} and \ref{prop:fellerX} and Corollary \ref{cor:freezer} should hold with the proper changes. E.g., the scaling limit of the rescaled total size of the frozen components in $ \mathrm{F}_{n,p}(\lambda)$ should be a pure-jump Feller process $\mathscr{X}_{p}(\lambda)$ with jump kernel given by 
 \begin{eqnarray} \label{eq:jumpkernelp} \frac{1}{2}  \frac{1}{ \sqrt{2\pi}}  \frac{ \mathrm{d}y }{y^{3/2}} (y +  2p \cdot x)\frac{p_{1}(\lambda-x-y)}{p_{1}(\lambda-x)}.  \end{eqnarray}
 Specifying those results for $p=1$ and using \eqref{eq:linkFG} we deduce that the process of the total mass of the particles with surplus in the multiplicative coalescent $ \mathscr{M}(\lambda)$ has law $ \mathscr{X}_{1}(\lambda)$ and that conditionally on it, the remaining components are  distributed as the jump of the conditioned L\'evy process $ \mathscr{S}^{\lambda - \mathscr{X}_{p}(\lambda)}$. We were not aware of such a description prior to this work. %A consequence is a that the standard multiplicative coalescent should actually live in $\ell^{3/2 + \varepsilon}$ for all $ \varepsilon >0$.
 
\subsection{Asymptotics when $\lambda \to \infty$}
If the above description of the scaling limit $ \mathscr{X}_p(\lambda)$ of   $ n^{-2/3} \cdot \|\mathrm{F}_{n,p}(\lambda)\|_{\bluecirc}$ is granted, then one can perform the analysis in the near supercritical regime $\lambda \to \infty$. It is easy to see that $ \mathscr{X}_{p}(\lambda)$ tends to $\infty$ and it is not hard to see that it is asymptotically larger than $\lambda$. Using \eqref{eq:asympp1} and the definition of the jump kernel \eqref{eq:jumpkernelp} we can compute formally
 \begin{eqnarray*}  \frac{ \mathrm{d}}{ \mathrm{d}\lambda} \mathbb{E}[ \mathscr{X}_{p}(\lambda)] &=& \mathbb{E}\left[ \frac{1}{2} \int_{0}^{\infty}  \frac{ \mathrm{d}y}{ \sqrt{2\pi} y^{3/2}} \left( 2p \mathscr{X}_{p}(\lambda) + y\right) \cdot y \cdot \frac{p_{1}( \lambda- \mathscr{X}_{p}(\lambda) -y) }{p_{1}( \lambda- \mathscr{X}_{p}(\lambda))} \right]\\
 & \underset{ \begin{subarray}{c} \lambda \to \infty\\
  \mathscr{X}_{p}(\lambda) - \lambda \to \infty 
\end{subarray}}{\sim} &  \mathbb{E}\left[ \frac{p \mathscr{X}_{p}(\lambda)}{\sqrt{2 \pi}} \int_{0}^{\infty}  \frac{ \mathrm{d}y}{y^{3/2}} \cdot y \cdot \mathrm{e}^{- y ( \lambda- \mathscr{X}_{p}(\lambda))^{2}/2}\right] =  \mathbb{E}\left[ \frac{p \mathscr{X}_{p}(\lambda)}{ \mathscr{X}_{p}(\lambda)- \lambda} \right].  \end{eqnarray*}
From this we conjecture the asymptotic rate of growth of the process $ \lambda \mapsto \mathscr{X}_{p}(\lambda)$ for $p \in [0,1]$
 \begin{eqnarray}
 \label{eq:rateofgrwoth} \frac{ \mathscr{X}_{p}(\lambda)}{\lambda} \xrightarrow[\lambda\to\infty]{ (\mathbb{P})} 1+p.
  \end{eqnarray}
 
 In particular, when $p=0$ i.e.\,when we stop the cluster growth when they have a positive surplus, we believe that the total mass of the frozen part is of order $\lambda$ and furthermore that $ \mathscr{X}_{0}(\lambda)-\lambda$ converge in distribution as $\lambda \to \infty$ towards a stationary law (an example of self-organized criticality).%: if the available number of edges in the forest part becomes too large compared to its size then some trees are frozen decreasing the remaining number of edges.
 
In the case $p=1$ we should have $ \mathscr{X}_{1}(\lambda) \approx 2\lambda$. This is coherent with the result of Luczak \cite{luczak1994structure} saying that the largest cluster in $ \mathrm{G}_{n}(\lambda)$ is of order $2 \lambda n^{2/3}$ when $ \lambda \to \infty$ (this cluster is likely to be formed by the majority of the connected components of the frozen part).

%One can also wonder about the relative size of the components in the frozen part. More precisely, extending Theorem \ref{thm:FMC}, it should be plain to prove that $ n^{-2/3}\cdot \mathrm{Comp}( [\mathrm{F}_{n,p}(\lambda)]_{\bluecirc})$ converge in $\ell^{1}_{\downarrow}$ towards a process $ [\FM_{p}(\lambda)]_{\bluecirc}$ whose total mass is given by $ \| \FM_{p}(\lambda)\|_{\bluecirc} = \mathscr{X}_{p}(\lambda)$. We wonder whether the renormalized process $[\FM_{p}(\lambda)]_{\bluecirc} / \| \FM_{p}(\lambda)\|_{\bluecirc}$ with values in the simplex $\{ x_{1}\geq x_{2} \geq \dots  \geq 0 : \sum x_{i} = 1\}$ converges as $\lambda \to \infty$. In the case $p=1/2$, via our coupling with parking on random mappings, it is tempting to conjecture that we asymptotically discover the connected component of the underlying random mapping and so   \begin{eqnarray*}  \frac{[\FM_{p}(\lambda)]_{\bluecirc} }{ \| \FM_{p}(\lambda)\|_{\bluecirc}} \xrightarrow[\lambda \to\infty]{ (\mathbb{P})}  \mathrm{PoissonDirichlet}(1/2),  \end{eqnarray*}  since the last law is known to describe the asymptotic proportion of the mass of components in a uniform random mapping. In the case $p=0$, the limit should be $ \mathbf{0}$, but in other cases we have no clue.
 \subsection{Process construction}
In the case of the multiplicative coalescent, there has been a substantial amount of work describing the process in terms of collections of excursion lengths of evolving random functions \cite{Armendariz,broutin2016new,limic2019eternal,martin2017rigid,Bravo}. We do not know whether such  a construction is doable for our frozen processes.

\section{Extension of parking process} 
In this work we used a coupling between the Erd{\H{o}}s--R\'enyi random graph and uniform parking on Cayley trees to study the later. Our results obviously call for generalizations for other models of random trees and other arrival distributions of cars. The ideas of this paper can indeed be extended to cover the model of \cite{contat2020sharpness} and this is the subject of a forthcoming work of the first author \cite{Contat21+}. In particular, although the precise location of the phase transition depends on the combinatorial details, we believe that the scaling limits unraveled in this paper are common to a large class of models as long as the degree distribution and the car arrivals have a sufficiently light tail. In the presence of {group arrival of cars with heavy tail}, new scaling limits should occur related to the different universality classes observed for component sizes in configuration models \cite{bhamidi2018multiplicative,broutin2018limits,goldschmidt2020stable,joseph2014component}. 

Once Conjecture \ref{conjectureGFT} has been addressed, one hope to describe a dynamical scaling limit for the geometry of the parking process involving the geometry of the components as well as the flux of cars. This should involve spiraling frozen fractals all around \cite{frozen}. We we will try to address those questions in future works.

%If we believe in universality, the work \cite{chen2021enumeration} can be used to compute the law of $ \mathscr{X}(\lambda)$ and $ \mathscr{D}(\lambda)$ for any fixed $\lambda$. \nico{faut que je refasse le calcul et les simulations\dots}
%\nico{Est-ce qu'on devrait faire la distance totale parcourue directement sur $T_n$ en sommant sur les composantes? Cas de la racine chiant.} \nico{Je vais le mettre dans la section comments}

\clearpage
\section*{Table of notation}

\subsection*{General notation}

\begin{tabular}{ll}
$\mathrm{X}_n( \lambda)$ & $ \mathrm{X}_{n}(\lambda) = X\left(n, \left\lfloor \frac{n}{2} + \frac{\lambda}{2} n^{2/3} \right\rfloor \vee 0\right)$  shorthand notation for a process $X(n,m)$ \\
$\Delta X(n,m) $&$ \Delta X(n,m) = X(n,m+1)- X(n,m)$ \\ & shorthand notation for the increments of a process $X(n,m)$ \\
$T_{n}$ & uniform rooted Cayley tree over $\{1,2, \dots , n \}$\\
$T_{ \star}(n,m)$ & for $\star \in \{ \mathrm{near}, \mathrm{full}, \mathrm{strong}\}$ different types of components\\ &  in $T_{n}$ after parking $m$ cars, see Figure \ref{fig:components}.\\
$X_{1}, Y_{1}, X_{2}, Y_{2} \dots $ & independent uniform numbers over $\{1,2, \dots , n \}$\\
&  the $X_{i}$'s are seen as car arrivals and are independent of $T_{n}$\\
& while the $Y_{i}$'s are coupled non-trivially with $T_{n}$\\
$ \vec{E}_{i} = (X_{i}, Y_{i})$ & $i$th oriented edge\\
$G(n,m)$ & Erd{\H{o}}s--R\'enyi random graph built by adding the first $m$ unoriented edges\\
$F(n,m)$ & frozen Erd{\H{o}}s--R\'enyi random graph  built from  the first $m$ edges\\
$D(n,m)$ & number of discarded edges in the construction of $F(n,m)$\\
 & or equivalently of cars that did not manage to park on $T_{n}$\\
 $W(n,m)$ & uniform unrooted labeled forest with $n$ vertices and $m$ edges\\
  $ \mathbb{W}_{n}(\cdot)$ & sequence of renormalized sizes of components in $ \mathrm{W}_{n}(\cdot)$.\\

$ \mathfrak{g}, [ \mathfrak{g}]_{\circ}, [ \mathfrak{g}]_{\bluecirc}, [ \mathfrak{g}]_{ \mathrm{tree}}$ & a multigraph, its subgraph made of white/blue vertices, and its forest part\\
$ \| \mathfrak{g}\|_{\bullet}, \| \mathfrak{g}\|_{\circ}, \| \mathfrak{g}\|_{\bluecirc},\| \mathfrak{g}\|_{\edge}$ & number of vertices, white vertices, blue vertices and edges of $ \mathfrak{g}$\\

$ \mathfrak{F}(n,m)$ & unrooted forests over $\{1,2, \dots , n\}$ with $m$ edges\\
$\# \mathfrak{F}(n,m)$ & number of unrooted forests over $\{1,2, \dots , n\}$ with $m$ edges\\
$\mu(k) = 2   \mathrm{e}^{-k} \frac{k^{k-2}}{k!}$ & step distribution in the random walk $S$ coding  the forests\\

\end{tabular}
\subsection*{Continuous processes notation}
The random variables in the ``continuous world'' are usually denoted with a mathscr font.\\ 

\begin{tabular}{ll}
$\mathscr{S}$ & $3/2$-stable spectrally positive L\'evy process with L\'evy measure $ \frac{ \mathrm{d}y}{ \sqrt{2\pi} y^{5/2}}$\\
$ \mathscr{S}^{u}$ & version of $ \mathscr{S}$ conditioned on $  \mathscr{S}_{1} = u$\\
$ p_{1},p_{s}$ & density of $ \mathscr{S}$ at time $1$ (Airy distribution) resp.\ $s \geq 0$\\
$ \mathbf{n}_{\lambda}(x, \mathrm{d}y), g_{x,\lambda}(y)$ & jump kernel, see Section \ref{sec:markovfreezer} and  \eqref{def:gxy}\\
$ (\mathscr{M}(\lambda) : \lambda \in \mathbb{R})$& (standard)  multiplicative coalescent\\
$ (\FM(\lambda) : \lambda \in \mathbb{R})$& frozen multiplicative coalescent\\
$ [\FM(\lambda)]_{\bluecirc}, \| \FM(\lambda)\|_{\bluecirc}$ & $\ell^{1}$- part of $ \FM(\lambda)$ and its total mass\\
$ [\FM(\lambda)]_{\circ}$ & $\ell^{2}$- part of $ \FM(\lambda)$\\
\end{tabular}

\subsection*{Generating functions and counting functions}
Generating functions are denoted by a mathbf symbol.

\noindent \begin{tabular}{ll}
$ \mathrm{PF}(n,m)$ & for $0 \leq m \leq n$ number of parking functions, \\ & i.e.\ of Cayley trees of size $n$ and $m$ cars so that all cars park\\
$ \mathrm{PF}_{ \mathrm{root}}(n,m)$ & for $0 \leq m \leq n$ number of parking functions with empty root, \\ & i.e.\ of Cayley trees of size $n$ and $m$ cars so that all cars park and the root stays void\\
$\mathrm{FP}(n,n+p)$ & for $n, p \geq 0$ number of fully parked trees with flux $p$ \\ & i.e.\ of Cayley tree of size $n$ and $n+p$ cars so that  exactly $p$ cars do not park\\
$\mathrm{SP}(n,n+p)$ & for $n, p \geq 0$ number of strongly parked trees with flux $p$ \\ & i.e.\ of Cayley trees of size $n$ and $n+p$ cars so that  exactly $p$ cars do not park \\ & and so that all edges have positive flux\\

$ \mathbf{T}(x)$& Exponential GF $  \sum_{ n \geq 1} \frac{x^{n} n^{n-2}}{n!}$ for unrooted Cayley trees\\
$ \mathbf{N}(x) $& Exponential GF $ \sum_{ n \geq 1} \frac{\mathrm{PF}(n,n-1)}{n! (n-1)!}x^{n}$ for nearly parked trees \\ & which is equal to $ \frac{1}{2} \mathbf{T}(2x)$ by Proposition \ref{prop:countFP}\\
$ \mathbf{F}(x) $& Exponential GF $ \sum_{ n \geq 1} \frac{\mathrm{FP}(n,n)}{(n!)^{2}}x^{n}$ for fully parked trees\\
$ \mathbf{S}(x) $& Exponential GF $ \sum_{ n \geq 1} \frac{ \mathrm{SP}(n,n)}{(n!)^{2}}x^{n}$ for strongly parked trees \\ & which is equal to $ 1 - \ln(2) - \sqrt{1 - 4 x} + \ln\left(1 + \sqrt{1 - 4 x}\right)$ by Proposition \ref{prop:kingyan}\\
$ \mathbf{S}(x,y) $& Exponential GF $\sum_{ n \geq 1} \frac{ \mathrm{SP}(n,n+p)}{ n! (n+p)!}x^{n}y^{p}$  for strongly parked trees with flux\\
\end{tabular}

\begin{tabular}{ll}

\end{tabular}

\subsection*{Notation for Section \ref{sec:FMC}}

\begin{tabular}{ll}
$\mathbb{C}\mathrm{adlag}( I,  \mathrm{Pol})$& Space of c\`adl\`ag function from an interval $I \subset \mathbb{R}$ to some Polish space $ \mathrm{Pol}$\\
$ (\mathscr{M}(\lambda) : \lambda \in \mathbb{R})$& (standard)  multiplicative coalescent\\
$ \mathcal{E} = \ell^{1}_{\downarrow} \times \ell^{2}_{\downarrow}, \quad \mathrm{d}_{ \mathcal{E}}$ & state space of the frozen multiplicative coalescent and its metric\\
$ \mathcal{E}_{0} = \ell^{\infty}_{\downarrow,0} \times \ell^{\infty}_{\downarrow,0},  \quad \mathrm{d}_{ \mathrm{sup}}$ & proxy state space of pair of decreasing sequences tending to $0$ and its metric\\
$ (\FM(\lambda) : \lambda \in \mathbb{R})$& frozen multiplicative coalescent\\
$  [\FM(\lambda)]_{\bluecirc}, \| \FM(\lambda)\|_{\bluecirc}$ & $\ell^{1}$- part of $ \FM(\lambda)$ and its total mass\\
$  \mathscr{X}(\lambda) = \| \FM(\lambda)\|_{\bluecirc}$ & shorthand notation for the total mass of the frozen particles\\

$ \mathrm{O}_{n}(\lambda_{0})$ & number of vertices  of $ \mathrm{G}_{n}(0)$ whose components \\ & carry surplus appeared before time $\lambda_{0}$\\
$ \mathrm{F}_{n}^{[\lambda_{0}]}, \mathrm{G}_{n}^{[\lambda_{0}]}$ & frozen (resp standard Erd{\H{o}}s--R\'enyi process)  process started from time $\lambda_{0}$\\ & by removing the components with surplus at time $\lambda_{0}$\\
$ \mathrm{F}_{n}^{[\lambda_{0}, \delta]},\mathrm{G}_{n}^{[\lambda_{0},\delta]}$ & same process as above restricted to the specks\\
&  (i.e.\ components of $ [ \mathrm{G}_{n}(\lambda_{0})]_{ \mathrm{tree}})$ of size at least $\delta n^{2/3}$\\
$ \mathbb{F}_{n}, \mathbb{F}_{n}^{[\lambda_{0}]},\mathbb{F}_{n}^{[\lambda_{0}, \delta]}$& Sequences of renormalized sizes of components (frozen followed by standard)\\
$ \eta$-skeleton & graph spanned by the specks of $  \mathrm{G}_{n}(0)$ of mass at least $\eta$\\
& or belonging to a cycle of $ \mathrm{G}_{n}^{[\lambda_{0}]}(0)$\\
$\gamma= \gamma_{n}(\lambda_{0}, \eta)$ & minimal weight of a speck of the $\eta$-skeleton\\
\end{tabular}

\clearpage

 \bibliographystyle{siam}
%\bibliography{bibli}

\end{document}